\documentclass{amsart}
\pdfoutput=1
\usepackage{amssymb,latexsym,amsmath,amscd,graphicx,graphics,epic,eepic,bm,color,array,mathrsfs,fullpage}
\usepackage{enumerate}
\usepackage[all,knot,poly]{xy}
\usepackage{caption}
\usepackage{ragged2e}
\usepackage{url}
\theoremstyle{plain}
\newtheorem{theorem}{Theorem}[section]
\newtheorem{lemma}[theorem]{Lemma}
\newtheorem{proposition}[theorem]{Proposition}
\newtheorem{corollary}[theorem]{Corollary}

\newtheorem{problem}[theorem]{Problem}

\theoremstyle{definition}
\newtheorem{definition}[theorem]{Definition}

\newtheorem{algorithm}[theorem]{Algorithm}

\theoremstyle{remark}
\newtheorem{remark}[theorem]{Remark}

\numberwithin{theorem}{section}
\numberwithin{equation}{section}

\newcommand{\R}{\mathbb{R}}

\newcommand{\W}{\mathcal{W}}
\newcommand{\U}{\mathcal{U}}

\newcommand{\vf}{\varphi}

\newcommand{\id}{\mathrm{id}}

\newcommand{\rank}{\mathrm{rank}}

\newcommand{\Tr}{\mathrm{Tr}}

\newcommand{\qf}{\mathrm{qf}}
\newcommand{\diag}{\mathrm{diag}}
\newcommand{\Span}{\mathrm{Span}}
\newcommand{\Hess}{\mathrm{Hess}}

\graphicspath{{./Figures-Main-Body/}{./Figures-Supp/}{./figures/}}

\begin{document}

\title{Weighted Low-rank Approximation via Stochastic Gradient Descent on Manifolds}

\author[1]{Conglong Xu}
\author[2]{Peiqi Yang} 
\author[3]{Hao Wu}

\thanks{Correspondence to: Hao Wu (\texttt{haowu@gwu.edu})}

\thanks{The authors would like to thank Yingfeng Hu for interesting and helpful conversations.}

\address{Department of Mathematics, The George Washington University, Phillips Hall, Room 739, 801 22nd Street, N.W., Washington DC 20052, USA. Telephone: 1-202-994-0653, Fax: 1-202-994-6760}

\email{xuconglong@gwmail.gwu.edu, pqyang@gwmail.gwu.edu, haowu@gwu.edu}



\begin{abstract}
We solve a regularized weighted low-rank approximation problem by a stochastic gradient descent on a manifold. To guarantee the convergence of our stochastic gradient descent, we establish a convergence theorem on manifolds for retraction-based stochastic gradient descents admitting confinements. On sample data from the Netflix Prize training dataset, our algorithm outperforms the existing stochastic gradient descent on Euclidean spaces. We also compare the accelerated line search on this manifold to the existing accelerated line search on Euclidean spaces.
\end{abstract}

\maketitle

\section{Introduction}\label{sec-intro}

In this paper, we study the weighted low-rank approximation problem below: 

\begin{problem}[The Weighted Low-Rank Approximation Problem]\label{prob-WLRA}
Assume that:
\begin{enumerate}
    \item \label{problem-assumption-1} $m$, $n$ and $k$ are fixed positive integers satisfying $k \leq \min\{m,n\}$,
    \item \label{problem-assumption-2} $A=[a_{i,j}] \in \R^{m\times n}$ is a given matrix of constants,
    \item \label{problem-assumption-3} $W=[w_{i,j}] \in \R^{m\times n}$ is a given matrix of weights satisfying $w_{i,j} \geq 0$ for ${(i,j) \in \{1,2,\dots,m\}\times \{1,2,\dots,n\}}$ and $\sum_{i=1}^m \sum_{j=1}^n w_{i,j} = 1$,
\end{enumerate}
where $\R^{m\times n}$ is the space of $m \times n$ real matrices. Define $\hat{F}:\R^{m\times n} \rightarrow \R$ by \newline
$\hat{F}(P) = \sum_{i=1}^m \sum_{j=1}^n w_{i,j}(a_{i,j}-p_{i,j})^2$ for $P=[p_{i,j}] \in \R^{m\times n}$. Solve for
\begin{equation*}
\mathrm{argmin}\{\hat{F}(P) ~|~ P\in \R^{m\times n},~  \rank P \leq k\}.
\end{equation*}
\end{problem}

When all the positive weights in $W$ take the same value, we say Problem \ref{prob-WLRA} has binary weights. The Matrix Completion Problem, a problem extensively studied in the recent literatures (for instance, \cite{Bertsimas-Cory-Lo-Pauphilet:2024}, \cite{Yan-Tang-Li:2024}, \cite{Yan-Zhang:2024}, \cite{Lee:2024}, \cite{Kelner-Li-Liu-Sidford-Tian:2023}, \cite{Chakraborty-Dey:2023}, \cite{Yang-Ma:2023}, \cite{Bordenave-Coste-Nadakuditi:2023}, \cite{Wan-Cheng:2023}, \cite{Boumal-Absil:2011}), is a special case of Problem \ref{prob-WLRA} with binary weights.

Problem \ref{prob-WLRA} is known to be NP-hard (\cite{Gillis-Glineur:2011}). The existing approach to this problem in the previous literatures is to regularize it to Problem \ref{prob-RWLRA-2} on the Euclidean space $\R^{m\times k}\times \R^{n\times k}$ and estimate the solution to the problem by various first or second order algorithms on that Euclidean space (for instance, \cite{Ban-Woodruff-Zhang:2019}, \cite{Boumal-Absil:2011}, \cite{Srebro-Jaakkola:2003}). Different from this existing approach, we regularize Problem \ref{prob-WLRA} to Problem \ref{prob-RWLRA-1}.
\begin{problem}[A Regularized Weighted Low-Rank Approximation Problem]\label{prob-RWLRA-1}
Under Assumptions (\ref{problem-assumption-1})-(\ref{problem-assumption-3}) in Problem \ref{prob-WLRA}, fix a positive number $\lambda$ and define $F:\R^{m\times n} \rightarrow \R$ by ${F(P) = \sum_{i=1}^m \sum_{j=1}^n w_{i,j}(a_{i,j}-p_{i,j})^2+\lambda \|P\|_F^2}$ for $P=[p_{i,j}] \in \R^{m\times n}$, where $\|P\|_F :=\sqrt{\sum_{i=1}^m \sum_{j=1}^n p_{i,j}^2}$ is the Frobenius norm of $P$. Solve for
\[
\mathrm{argmin}\{F(P) ~|~ P\in \R^{m\times n},~  \rank P \leq k\}.
\]
\end{problem}

To design our convergent stochastic gradient descent for Problem \ref{prob-RWLRA-1}, we reformulate this problem to an unconstrained problem on a Riemannian manifold via the Reduced Singular Value Decomposition. Denote by $O_n$ the real $n\times n$ orthogonal group and by $V_k(\R^n)$ the real $n\times k$ Stiefel manifold. That is
\begin{eqnarray*}
O_n & = & \{U\in \R^{n\times n}~|~ U^TU=I_n\}, \\
V_k(\R^n) & = & \{V \in \R^{n\times k}~|~ V^T V =I_k\},
\end{eqnarray*}
where $I_n$ is the $n\times n$ identity matrix. Define the diagonal function $D^{m\times n}_k: \R^k \rightarrow \R^{m\times n}$ by 
\begin{equation}\label{eq-def-D}
D^{m\times n}_k([x_1,\dots,x_k]^T)=[d_{i,j}], \text{ where } d_{i,j}=\begin{cases} x_i & \text{if } i=j\leq k, \\ 0 & \text{otherwise}, \end{cases}
\end{equation} 
and in particular $D^{k\times k}_k: \R^k \rightarrow \R^{k\times k}$ by
\begin{equation}\label{eq-def-D-k-by-k}
D^{k\times k}_k([x_1,\dots,x_k]^T)=[d_{i,j}], \text{ where } d_{i,j}=\begin{cases} x_i & \text{if } i=j, \\ 0 & \text{otherwise}. \end{cases}
\end{equation} 

By the Singular Value Decomposition, for any $P \in \R^{m\times n}$, $\rank P \leq k$ if any only if $P$ can be decomposed into $P = \tilde{U} D^{m\times n}_k(\mathbf{x}) \tilde{V}^T$ for some $\tilde{U} \in O_m$, $\tilde{V} \in O_n$ and $\mathbf{x}\in \R^k$. (The Singular Value Decomposition actually further states that the diagonal elements of $D^{m\times n}_k(\mathbf{x})$ are non-negative and in descending order. We ignore such constraints on $D^{m\times n}_k(\mathbf{x})$ to avoid adding boundary restrictions to the gradient search.) Write $\tilde{U}={[\mathbf{u}_1,\dots, \mathbf{u}_m]}$ and $\tilde{V} = {[\mathbf{v}_1,\dots, \mathbf{v}_n]}$. Define $U= {[\mathbf{u}_1,\dots, \mathbf{u}_k]}$ and $V = {[\mathbf{v}_1,\dots, \mathbf{v}_k]}$. Then $\tilde{U} D^{m\times n}_k(\mathbf{x}) \tilde{V}^T = U D^{k\times k}_k(\mathbf{x}) V^T$. Thus, we have the following lemma.

\begin{lemma}[Reduced Singular Value Decomposition]\label{lemma-reduced-svd}
For $P \in \R^{m\times n}$, $\rank P \leq k$ if and only if it admits the following Reduced Singular Value Decomposition
\begin{equation}\label{eq-rsvd}
P = U D^{k\times k}_k(\mathbf{x}) V^T,
\end{equation}
where $U \in V_k(\R^m)$, $V \in V_k(\R^n)$ and $\mathbf{x} \in \R^k$.
\end{lemma}

By Lemma \ref{lemma-reduced-svd}, we reformulate Problem \ref{prob-RWLRA-1} into Problem \ref{prob-RWLRA-1-reform}, which is an unconstrained problem on the Riemannian manifold $V_k(\R^m)\times \R^k \times V_k(\R^n)$. Note that, Lemma \ref{lemma-reduced-svd} indicates that any solution to Problem \ref{prob-RWLRA-1} is equivalent a solution to Problem \ref{prob-RWLRA-1-reform}, and vice versa.

\begin{problem}[A Reformulated Regularized Weighted Low-Rank Approximation Problem]\label{prob-RWLRA-1-reform}
Under Assumptions (\ref{problem-assumption-1})-(\ref{problem-assumption-3}) in Problem \ref{prob-WLRA}, fix a positive number $\lambda$ and let $F:\R^{m\times n} \rightarrow \R$ be as in Problem \ref{prob-RWLRA-1}. Solve for
\[
\mathrm{argmin}\{F(U D^{k\times k}_k(\mathbf{x}) V^T) ~|~ (U,\mathbf{x},V)  \in V_k(\R^m)\times \R^k \times V_k(\R^n) \}.
\]
\end{problem}
For $\mathbf{x} = [x_1,\dots,x_k]^T \in \R^k$, denote standard Euclidean norm of $\mathbf{x}$ by $\|\mathbf{x}\|: = \sqrt{\sum_{l=1}^k x_l^2}$. Note that, for $(U,\mathbf{x},V)  \in V_k(\R^m)\times \R^k \times V_k(\R^n)$ and $P=U D^{k\times k}_k(\mathbf{x}) V^T$,
\begin{equation}\label{eq-norm-equal}
\|P\|_F = \|\mathbf{x}\|.
\end{equation}

To estimate the solution to Problem \ref{prob-RWLRA-1-reform}, we start by establishing convergence Theorem \ref{thm-confined-SGD} for stochastic gradient descents on manifolds admitting confinements (see Definition \ref{def-confinement}). Then, we apply Theorem \ref{thm-confined-SGD} to design stochastic gradient descent Algorithm \ref{alg-confined-SGD-RWLRA-1-direct} for Problem \ref{prob-RWLRA-1-reform}. Our numerical results indicate that Algorithm \ref{alg-confined-SGD-RWLRA-1-direct} outperforms the existing stochastic gradient descent on Euclidean spaces on sample data from the Netflix Prize training dataset. We also apply the accelerated line search on manifolds (Algorithm \ref{alg-ALS}) to design accelerated line search Algorithm \ref{alg-ALS-RWLRA-1} for Problem \ref{prob-RWLRA-1-reform}. Our numerical results indicate that Algorithm \ref{alg-ALS-RWLRA-1} converges faster than the existing accelerated line search on Euclidean spaces on the same sample data. But, when $\lambda$ in Problem \ref{prob-RWLRA-1-reform} is sufficiently small, the latter converges to a local minimum with lower value for the unregularized cost function.

\textbf{Overview and main contributions.}
In Section \ref{sec-confined-SGD-mfd}, we establish convergence Theorem \ref{thm-confined-SGD} for stochastic gradient descents on manifolds admitting confinements. Theorem \ref{thm-confined-SGD} is the main theoretical contribution of this paper. This theorem is inspired by Section $5$ of Bottou (1999). 
Intuitively speaking, there are two obvious ways for a gradient descent to diverge. First, the gradient flow-line could be unbounded in the descending direction. Bottou (1999) avoids this by assuming the existence of a confinement (\cite[(5.1)]{Bottou}). We do the same by adapting Bottou's ideas to manifolds. Second, the gradient vector of the cost function could grow so rapidly that the gradient descent fails to be a good approximation of the gradient flow-line. To avoid this, Bottou elected to impose constraints to limit the growth of the gradient vector (\cite[(5.2)]{Bottou}). These constraints somewhat limit the applications of his version of the convergence theorem though. We choose to resolve this issue by a ``semi-adaptive" approach. That is, we start with a preferred sequence of step sizes and then, depending on the outcome of each iteration of the algorithm, adjust the step size for the following iteration to guarantee convergence.

In Section \ref{sec-WLRA-confined-SGD}, we interpret Problem \ref{prob-RWLRA-1-reform} as an expectation cost problem and apply Theorem \ref{thm-confined-SGD} to design a convergent stochastic gradient descent (Algorithm \ref{alg-confined-SGD-RWLRA-1-direct}) for Problem \ref{prob-RWLRA-1-reform}. 

In Section \ref{sec-numerical-results}, we explore the numerical performances of Algorithms \ref{alg-confined-SGD-RWLRA-1-direct} and \ref{alg-ALS-RWLRA-1}. In particular, we apply these algorithms to the Matrix Completion Problem, which is a special case of Problem \ref{prob-WLRA} with binary weights. We first compare Algorithm \ref{alg-confined-SGD-RWLRA-1-direct} to the benchmark given by the stochastic gradient descent Algorithm \ref{alg-confined-SGD-RWLRA-2-direct} on the Euclidean space $\R^{m\times k}\times \R^{n\times k}$ over the same number of iterations. Algorithm \ref{alg-confined-SGD-RWLRA-2-direct} was introduced in the previous literatures (for instance, \cite{Ban-Woodruff-Zhang:2019}). Our numerical results indicate that, on sample data from the Netflix Prize training dataset, Algorithm \ref{alg-confined-SGD-RWLRA-1-direct} outperforms Algorithm \ref{alg-confined-SGD-RWLRA-2-direct}.  We also compare Algorithms \ref{alg-confined-SGD-RWLRA-1-direct} and \ref{alg-ALS-RWLRA-1} to the benchmark given by the accelerated line search Algorithm \ref{alg-ALS-RWLRA-2} on the Euclidean space $\R^{m\times k}\times \R^{n\times k}$ over the same runtime. Our numerical results indicate that, on the same sample data, 
\begin{itemize}
\item Algorithm \ref{alg-ALS-RWLRA-1} converges faster than Algorithm \ref{alg-ALS-RWLRA-2}, 
\item Algorithm \ref{alg-ALS-RWLRA-2} converges to a local minimum with lower value for the unregularized cost function when $\lambda$ is sufficiently small,
\item Algorithm \ref{alg-ALS-RWLRA-1} outperforms Algorithm \ref{alg-confined-SGD-RWLRA-1-direct},
\item Algorithm \ref{alg-ALS-RWLRA-2} also outperforms Algorithm \ref{alg-confined-SGD-RWLRA-1-direct} except when $\lambda$ is sufficiently large.
\end{itemize}
Appendix \ref{sec-proof-sec-2} presents the proofs for the results in Section \ref{sec-confined-SGD-mfd}. Appendix \ref{sec-proof-sec-3} presents the proofs for the results in Section \ref{sec-WLRA-confined-SGD}. Appendix \ref{sec-proof-sec-4} presents the accelerated line search on manifolds and its application to Problem \ref{prob-RWLRA-1-reform}. Appendix \ref{sec-RWLRA-2} presents the stochastic gradient descent and the accelerated line search on Euclidean spaces and their applications to Problem \ref{prob-RWLRA-2}. In Appendix \ref{sec-WLRA-ALS}, we design a convergent stochastic gradient descent and an accelerated line search for Problem \ref{prob-WLRA-PW}, a special case of Problem \ref{prob-WLRA} with all weights being positive. In Appendix \ref{sec-new-proof-eyt}, we present a new proof of the Eckart-Young Theorem that we came up with while studying Problem \ref{prob-WLRA}.

\section{Stochastic Gradient Descent Algorithms on Riemannian Manifolds}\label{sec-confined-SGD-mfd}

\subsection{Some Relevant Concepts}\label{subsec-concepts}

Before stating our convergence theorem and its applications, we first review some basic concepts necessary for our discussions. In Sections \ref{sec-confined-SGD-mfd} and \ref{sec-WLRA-confined-SGD}, we only discuss the type of random functions given in Definition \ref{def-random-function} below.

\begin{definition}\label{def-random-function}
Let $M$ be a differentiable manifold and $\Omega$ a probability space. A function $f:M\times \Omega \rightarrow \R$ is called a random function on $M$. We say that $f$ is $k$th-order differentiable in the $M$-direction if, for every $\omega \in \Omega$, $f(\ast,\omega):M \rightarrow \R$ is $k$th-order differentiable. We also say that $f$ is locally bounded on $M$ if, for every compact set $K\subset M$, there is a $C_K>0$ such that $|f(x,\omega)|\leq C_K$ for every $x \in K$ and $\omega \in \Omega$.
\end{definition}

In Definition \ref{def-random-function}, $M$ is the manifold on which we search for the minimum of an expectation cost function, and $\Omega$ is the space of samples used to guide our search. While designing gradient descent algorithms on manifolds, retractions are often used to replace the often computationally expensive exponential maps of Riemannian manifolds (\cite{Li-Zhen-Pan-Zhao:2023}, \cite{Yamada-Sato:2023}, \cite{Bonnabel:2013}). We recall the definition of retractions on differentiable manifolds below.

\begin{definition}(\cite[Definition 4.1.1]{AMS}\label{def-retraction})
Let $M$ be a differentiable manifold. A retraction on $M$ is a $C^1$ map $R:TM \rightarrow M$ such that, for every $x \in M$, the restriction $R_x=R|_{T_x M}$ satisfies
\begin{itemize}
	\item $R_x(\mathbf{0}_x)=x$, where $\mathbf{0}_x$ is the zero vector in $T_x M$,
	\item $dR_x(\mathbf{0}_x)=\id_{T_x M}$ under the canonical identification $T_{\mathbf{0}_x}T_x M \cong T_x M$, where $dR_x$ is the differential of $R_x$ and $\id_{T_x M}$ is the identity map of $T_x M$.
\end{itemize}
\end{definition}

From a theoretical perspective, retractions pull back the differential calculations in a manifold onto its tangent spaces. Our work shows that, if one is careful, most such computations can be done in single tangent spaces of the manifold. And computations on such spaces are simply differential calculus in inner product spaces. This allows us to mostly avoid the more delicate aspects of Riemannian geometry such as connections and geodesics. To do this, it is convenient to use the concept of retraction-dependent Lipschitz gradients to replace the common notion of Lipschitz gradients.

\begin{definition}\label{def-R-lipschitz}
Let $M$ be a Riemannian manifold, $R$ a given retraction on $M$, and $\Omega$ a probability space.  Denote by $\left\langle \ast,\ast \right\rangle_x$ the Riemannian inner product on $T_x M$ for all $x \in M$ and by $\|\ast\|_x$ the norm it induces on $T_x M$. Suppose that the random function $f:M \times \Omega \rightarrow \R$ is first-order differentiable in the $M$-direction. Then, for each $x \in M$ and each $\omega \in \Omega$, the function $f_{x,\omega}:=f(R_x(\ast),\omega):T_x M \rightarrow \R$ is a first-order differentiable function. Its gradient $\nabla f_{x,\omega}$ is the vector in $T_x M$ dual to the differential $df_{x,\omega}$ via the inner product $\left\langle \ast,\ast\right\rangle_x$. We say:

\begin{itemize}
	\item  $f$ has $R$-Lipschitz gradient in the $M$-direction if there is a constant $C>0$ such that 
	\[
	\|\nabla f_{x,\omega}(\mathbf{v}) - \nabla f_{x,\omega}(\mathbf{0}_x)\|_x \leq C\|\mathbf{v}\|_x
	\]
	for every $x \in M$, $\mathbf{v} \in T_x M$ and $\omega \in \Omega$;
	\item  $f$ has locally $R$-Lipschitz gradient in the $M$-direction if, for every compact subset $K$ of $M$ and every $r>0$, there is a constant $C_{K,r}>0$ such that 
	\[
	\|\nabla f_{x,\omega}(\mathbf{v}) - \nabla f_{x,\omega}(\mathbf{0}_x)\|_x \leq C_{K,r}\|\mathbf{v}\|_x
	\]
	for every $x \in K$, every $\mathbf{v} \in T_x M$ satisfying $\|\mathbf{v}\|_x\leq r$ and every $\omega \in \Omega$.
\end{itemize}
In the case $\Omega=\{\omega\}$ is a probability space of a single point, the above gives the definitions of $R$-Lipschitz gradient and locally $R$-Lipschitz gradient of a deterministic function on $M$.
\end{definition}

\begin{remark}\label{remark-gradient-coincide}
In Definition \ref{def-R-lipschitz}, denote by $\nabla_M f$ the gradient of $f$ with respenct to $M$, that is, $\nabla_M f(x,\omega) = \nabla f_w(x)$, where $f_\omega := f(\ast,\omega):M \rightarrow \R$. Note that $f_{x,\omega} = f_{\omega}\circ R_x$. Since $dR_x(\mathbf{0}_x)=\id_{T_x M}$, we have that, for any $\mathbf{v}\in T_x M$,
\begin{align*}
\left\langle \nabla f_{x,\omega}(\mathbf{0}_x), \mathbf{v} \right\rangle_x & = (df_{x,\omega})|_{\mathbf{0}_x} (\mathbf{v}) = d(f_{\omega}\circ R_x)|_{\mathbf{0}_x} (\mathbf{v})
= (df_{\omega})|_{x}\circ (dR_x)|_{\mathbf{0}_x} (\mathbf{v}) \\
& = df_{\omega}|_{x}(\mathbf{v}) =\left\langle \nabla f_{\omega}(x) ,\mathbf{v}\right\rangle_x =\left\langle \nabla_M f(x,\omega) ,\mathbf{v}\right\rangle_x.
\end{align*}
This shows that 
\begin{equation}\label{eq-gradient-coincide}
\nabla f_{x,\omega}(\mathbf{0}_x) = \nabla_M f(x,\omega).
\end{equation} 
\end{remark}

Finally we introduce the concept of confinements of random functions on Remannian manifolds, which plays a central role in this manuscript.

\begin{definition}\label{def-confinement}
Let $M$ be a Riemannian manifold, $\Omega$ a probability space, and $f:M\times \Omega \rightarrow \R$ a random function on $M$ that is first-order differentiable in the $M$-direction. A confinement of $f$ on $M$ is a first-order differentiable function $\rho: M \rightarrow \R$ satisfying:
\begin{itemize}
	\item for every $\delta \in \R$, the set $\{x \in M~\big{|}~ \rho(x)\leq \delta\}$ is compact;
	\item there exists a $\rho_0\in \R$ such that $\left\langle \nabla \rho (x), \nabla_M f(x,\omega) \right\rangle_x \geq 0$ for every $\omega \in \Omega$ and every $x \in M$ satisfying $\rho(x)\geq \rho_0$, where $\left\langle \ast,\ast \right\rangle_x$ is the Riemannian inner product on  $T_x M$.
\end{itemize} 
\end{definition}

Existence of confinement of a random function guarantees that the stochastic gradient descent happens in a compact subset of the manifold and therefore has convergent subsequences.

\subsection{A Convergence Theorem}\label{subsec-confined-convergence-thm}
\begin{theorem}\label{thm-confined-SGD}
Assume that:
\begin{enumerate}
	\item \label{assumption-1} $M$ is an $m$-dimensional Riemannian manifold equipped with a $C^2$ retraction $R:TM\rightarrow M$. Denote by $\left\langle \ast,\ast \right\rangle_x$ the Riemannian inner product on $T_x M$ for all $x \in M$ and by $\|\ast\|_x$ the metric it induces on $T_x M$.
	\item \label{assumption-probability-space} $\Omega$ is a probability space with probability measure $\mu$. 
	\item \label{assumption-random-function} $f:M \times \Omega \rightarrow \R$ is a random function on $M$ satisfying:
	\begin{enumerate}[(i)]
	  \item $f$ is first-order differentiable in the $M$-direction,
	  \item $f$ and $\|\nabla_M f\|$ are locally bounded on $M$,
		\item $f$ has locally $R$-Lipschitz gradient in the $M$ direction.
	\end{enumerate}
	\item $F:M\rightarrow \R$ is the expectation of $f(x,\omega)$ with respect to $\omega$ under the probability distribution $\mu$. That is, $F(x)=E_{\omega \sim \mu}(f(x,\omega))=\int_\Omega f(x,\omega)d\mu$ for all $x\in M$.
	\item \label{assumption-confinement-function} $\rho:M\rightarrow \R$ is a $C^2$ confinement of $f$. Fix a $\rho_0\in \R$ satisfying that $\left\langle \nabla \rho (x), \nabla f_\omega(x) \right\rangle_x \geq 0$ for every $\omega \in \Omega$ and every $x \in M$ satisfying $\rho(x)\geq \rho_0$.
	\item $\{\omega_t\}_{t=0}^\infty$ is a sequence of independent random variables taking values in $\Omega$ with identical probability distribution $\mu$.
	\item \label{assumption-parameters} Fix positive constants $a,b,\Theta$ and a sequence $\{c_t\}_{t=0}^\infty$ of positive numbers satisfying $\sum_{t=0}^\infty c_t =\infty$ and $\sum_{t=0}^\infty c_t^2 <\infty$. Define $c= \max \{c_t~\big{|}~t\geq 0\}$, $\sigma=\sum_{t=0}^\infty c_t^2$ and $\rho_1 =\rho_0+ ca + \frac{b^2\sigma}{2}$.  
	\item \label{assumption-initial} $x_0 \in M$ is a fixed point satisfying $\rho(x_0)\leq \rho_0$.
\end{enumerate}

Define a sequences $\{x_t\}_{t=0}^\infty$ of random elements of $M$ and a sequence $\{\vf_t\}_{t=0}^\infty$ of random positive numbers so that, for $t\geq 0$,
\begin{eqnarray}
\label{eq-def-x-t} x_{t+1} & = & R_{x_t}\left(-\frac{c_t}{\vf_t}\nabla_M f(x_t,\omega_t)\right), \\
\label{eq-def-vf-t}	\vf_t & \geq & \max\{A_t,B_t,\frac{c_t}{\Theta}\},
\end{eqnarray}
where
	\begin{eqnarray}
	\label{eq-vf-bound-a-t} && A_t:= \frac{1}{a}\sup\left\{\max\{0, ~-\left\langle  \nabla\rho(x_t),\nabla_M f(x_t,\omega)\right\rangle_{x_t}\}~\big{|}~\omega \in \Omega\right\}, \\
	\label{eq-vf-bound-b-t} && B_t:= \frac{1}{b} \sup\left\{ \sqrt{\max\{0, ~\Hess(\rho\circ R_{x_t})|_{\theta\nabla_M f(x_t,\omega)}(\nabla_M f(x_t,\omega),\nabla_M f(x_t,\omega))\}}~\big{|}~|\theta|\leq \Theta,~\omega\in \Omega\right\}
	\end{eqnarray}
and $\Hess(\rho\circ R_x)$ is the Hessian of the function $\rho\circ R_x:T_x M \rightarrow \R$, which is defined on the inner product space $T_x M$. Then 
\begin{itemize}
  \item $\{x_t\}_{t=0}^\infty$ is contained in the compact subset $\{x\in M~\big{|}~ \rho(x) \leq \rho_1\}$ of $M$. In particular, it has convergent subsequences.
\end{itemize}

Assume that $\{\vf_t\}_{t=0}^\infty$ also satisfies:
	\begin{enumerate}[(a)]
	\item each $\vf_t$ is independent of $\{\omega_\tau\}_{\tau=t}^\infty$,
	\item  $\{\vf_t\}_{t=0}^\infty$ is bounded above and below by positive numbers. That is, there are $\Phi_{\min}, \Phi_{\max} >0$ such that 
	\begin{equation}\label{eq-vf-t-bounds}
	\Phi_{\min} \leq \vf_t \leq \Phi_{\max} \text{ for all } t\geq 0.
	\end{equation}
	\end{enumerate}
Then we further have that:
\begin{itemize}
	\item $\{F(x_t)\}_{t=0}^\infty$ converges almost surely to a finite number;
	\item $\{\|\nabla F(x_t)\|_{x_t}\}_{t=0}^\infty$ converges almost surely to $0$;
	\item any limit point of $\{x_t\}_{t=0}^\infty$ is almost surely a stationary point of $F$.
\end{itemize}
\end{theorem}

We can replace the sequence $\{\vf_t\}_{t=0}^\infty$ by a single factor $\vf$ and still get a convergent stochastic gradient descent algorithm, which is formulated in Corollary \ref{cor-confined-SGD} below. This makes the step sizes deterministic instead of ``semi-adaptive". Of course, this may sometimes shrink the step sizes by too much and potentially slows down the convergence.

\begin{corollary}\label{cor-confined-SGD}
With Assumptions (\ref{assumption-1}) to (\ref{assumption-initial}) stated in Theorem \ref{thm-confined-SGD}, fix a positive constant $\vf$ satisfying 
	\begin{equation}\label{eq-def-vf}
	\vf\geq \max\{A,B,\frac{c}{\Theta}\},
	\end{equation}
	where
	\begin{eqnarray*}
	&& A:= \frac{1}{a}\sup\left\{\max\{0, ~-\left\langle  \nabla\rho(x),\nabla_M f(x,\omega)\right\rangle_{x}\}~\big{|}~\rho(x)\leq \rho_0,~\omega \in \Omega\right\}, \\
	&& B:= \frac{1}{b} \sup\left\{\sqrt{\max\{0, ~\Hess(\rho\circ R_{x})|_{\theta\nabla_M f(x,\omega)}(\nabla_M f(x,\omega),\nabla_M f(x,\omega))\}}~\big{|}~|\theta|\leq \Theta,~x\in M,~\rho(x)\leq \rho_1,~\omega\in \Omega\right\},
	\end{eqnarray*}
and $\Hess(\rho\circ R_x)$ is the Hessian of the function $\rho\circ R_x:T_x M \rightarrow \R$, which is defined on the inner product space $T_x M$.	

Define a sequence $\{x_t\}_{t=0}^\infty$ of random elements of $M$ by
\begin{equation}\label{eq-def-x-t-deterministic}
x_{t+1} = R_{x_t}\left(-\frac{c_t}{\vf}\nabla_M f(x_t,\omega_t)\right) \text{ for } t\geq 0.
\end{equation}
Then: 
\begin{itemize}
    \item$\{x_t\}_{t=0}^\infty$ is contained in the compact subset $\{x\in M~\big{|}~ \rho(x) \leq \rho_1\}$ of $M$. In particular, it has convergent subsequences,
    \item $\{F(x_t)\}_{t=0}^\infty$ converges almost surely to a finite number,
    \item $\{\|\nabla F(x_t)\|_{x_t}\}_{t=0}^\infty$ converges almost surely to $0$,
    \item any limit point of $\{x_t\}_{t=0}^\infty$ is almost surely a stationary point of $F$.
\end{itemize}
\end{corollary}

The proofs of Theorem \ref{thm-confined-SGD} and Corollary \ref{cor-confined-SGD} are presented in Appendix \ref{sec-proof-sec-2}.

\section{A Stochastic Gradient Descent Algorithm for Problem \ref{prob-RWLRA-1-reform}}\label{sec-WLRA-confined-SGD}
In this section, we apply Theorem \ref{thm-confined-SGD} to design a convergent stochastic gradient descent for Problem \ref{prob-RWLRA-1-reform}. The proofs for the results in this section are presented in Appendix \ref{sec-proof-sec-3}. To apply Theorem \ref{thm-confined-SGD}, we define a probability space first. 

\begin{definition}\label{def-Omega-m-n-measure}
Let matrix $W=[w_{i,j}]$ be the matrix of weights given in Problem \ref{prob-WLRA}. Recall that $w_{i, j} \geq 0$ for $(i,j) \in \{1, 2, \ldots, m\}\times \{1, 2, \ldots, n\}$ and $\sum_{i=1}^m\sum_{j=1}^n w_{i,j} = 1$. Define the probability space $(\Omega_{m,n}, \mu)$ by $\Omega_{m,n} := \{1,2,\dots,m\}\times \{1,2,\dots,n\}$ and $\mu(\{(i,j)\})=w_{i,j}$ for every $(i,j) \in \Omega_{m,n}$.
\end{definition}

Next, we introduce a retraction on the manifold $V_k(\R^m)\times \R^k \times V_k(\R^n)$. To do that, we recall the $QR$ decomposition first.

\begin{definition}\label{def-qf}
Assume that $k\leq m$. For an $m\times k$ real matrix $C$ with linearly independent columns, define $\qf(C)=Q$ in the $QR$ decomposition $C=QR$, where 
\begin{itemize}
	\item $Q\in V_k(\R^m)$,
	\item $R$ is a $k \times k$ upper triangular matrix with positive entries along its diagonal.
\end{itemize}
\end{definition}
Computationally, one way to obtain $\qf(C)$ is by applying the Gram-Schmidt Process on the columns of $C$, scaling the resulting orthogonal set of vectors into an orthonormal set and then using this orthonormal set of vectors as the columns of $\qf(C)$.

\begin{lemma}(\cite[Examples 3.5.2 and 4.1.3]{AMS}\label{lemma-stiefel})
The $n \times k$ Stiefel manifold $V_k(\R^n)$ is a Riemannnian submanifold of $\R^{n\times k}$. Moreover:
\begin{itemize}
	\item For $X \in V_k(\R^n)$, the tangent space of $V_k(\R^n)$ at $X$ is $T_X V_k(\R^n) = \{ Z \in \R^{n\times k} ~|~X^TZ+Z^TX=0\}$.
	\item For any $X\in V_k(\R^n)$ and $Z \in T_X  V_k(\R^n)$, $X+Z\in \R^{n\times k}$ has linearly independent columns. Define $R^{V_k(\R^n)}:TV_k(\R^n)\rightarrow V_k(\R^n)$ by $R^{V_k(\R^n)}_X(Z) = \qf(X+Z)$. Then $R^{V_k(\R^n)}$ is a retraction on $V_k(\R^n)$.  
\end{itemize}
\end{lemma}

The tangent space of the manifold $V_k(\R^m)\times \R^k \times V_k(\R^n)$ at $(U,\mathbf{x},V) \in V_k(\R^m)\times \R^k \times V_k(\R^n)$ is $T_U V_k(\R^m)\times \R^k \times T_VV_k(\R^n)$. We are ready to describe a retraction on the manifold $V_k(\R^m)\times \R^k \times V_k(\R^n)$.

\begin{definition}\label{def-retraction-GS-prod}
For $(U,\mathbf{x},V) \in V_k(\R^m)\times \R^k \times V_k(\R^n)$ and $(Y,\hat{\mathbf{x}},Z) \in  T_{U}V_k(\R^m)\times \R^k \times T_{V}V_k(\R^n)$, define $R: TV_k(\R^m)\times \R^k \times TV_k(\R^n) \rightarrow V_k(\R^m)\times \R^k \times V_k(\R^n)$ by $R_{(U,\mathbf{x},V)}(Y,\hat{\mathbf{x}},Z) = (\qf(U+Y),\mathbf{x} +\hat{\mathbf{x}}, \qf(V +Z))$. 
\end{definition}

By Lemma \ref{lemma-stiefel}, the map $R$ in Definition \ref{def-retraction-GS-prod} is a retraction on $V_k(\R^m)\times \R^k \times V_k(\R^n)$. To design our stochastic gradient descent for Problem \ref{prob-RWLRA-1-reform}, we also need to define a random function satisfying Assumption (\ref{assumption-random-function}) in Theorem \ref{thm-confined-SGD}.

\begin{definition}\label{def-random-functions-eta-gamma-RWLRA-1}
Define $\hat{f}:\R^{m\times n}\times \Omega_{m,n}\rightarrow \R$ and $f:\R^{m\times n}\times \Omega_{m,n}\rightarrow \R$ by 
\begin{eqnarray}
\label{eq-def-hat-f-eta-gamma-RWLRA-1} \hat{f}(P; \eta, \gamma): & = & (a_{\eta,\gamma}-p_{\eta,\gamma})^2,\\
\label{eq-def-f-eta-gamma-RWLRA-1} f(P; \eta, \gamma): & = & \hat{f}(P; \eta, \gamma) + \lambda \|P\|_F^2 
= (a_{\eta,\gamma}-p_{\eta,\gamma})^2 + \lambda \|P\|_F^2,
\end{eqnarray}
for $\lambda > 0$ given in Problem \ref{prob-RWLRA-1-reform}, and for all $P=[p_{i,j}]\in \R^{m\times n}$ and $(\eta,\gamma)\in \Omega_{m,n}$. \\ Define $g: V_k(\R^m)\times \R^k \times V_k(\R^n)\times \Omega_{m,n} \rightarrow \R$ by
\begin{equation}\label{eq-def-hat-g-eta-gamma-RWLRA-1} 
g(U,\mathbf{x},V; \eta,\gamma): = f (U D^{k\times k}_k(\mathbf{x}) V^T; \eta,\gamma)
= \hat{f}(U D^{k\times k}_k(\mathbf{x}) V^T; \eta,\gamma) + \lambda \|\mathbf{x}\|^2,
\end{equation}
for all $(U,\mathbf{x},V) \in V_k(\R^m)\times \R^k \times V_k(\R^n)$ and $(\eta,\gamma)\in \Omega_{m,n}$, where $D^{k\times k}_k(\mathbf{x})$ is defined in Equation \eqref{eq-def-D-k-by-k}. For notational convenience, we define the functions $\hat{f}_{\eta, \gamma}:\R^{m\times n}\rightarrow \R$, $f_{\eta, \gamma}:\R^{m\times n}\rightarrow \R$ and ${g_{\eta, \gamma}:V_k(\R^m)\times \R^k \times V_k(\R^n) \rightarrow \R}$ for $(\eta, \gamma)\in \Omega_{m,n}$ by
\begin{equation}\label{eq-def-random-function-eta-gamma-RWLRA-1} 
\hat{f}_{\eta, \gamma}(P) = \hat{f}(P; \eta, \gamma), \
f_{\eta, \gamma}(P) = f(P; \eta, \gamma) \ \text{and }
g_{\eta, \gamma}(U,\mathbf{x},V) = g(U,\mathbf{x},V; \eta, \gamma),
\end{equation}
for all $P=[p_{i,j}]\in \R^{m\times n}$ and $(U,\mathbf{x},V) \in V_k(\R^m)\times \R^k \times V_k(\R^n)$.
\end{definition}

\begin{definition}\label{def-func-G}
Let $F:\R^{m\times n} \rightarrow \R$ be as in Problem \ref{prob-RWLRA-1-reform}, define $G:V_k(\R^m)\times \R^k \times V_k(\R^n)\rightarrow \R$ by
\begin{equation}\label{eq-def-G-RWLRA-1}
G(U,\mathbf{x},V) = F(U D^{k\times k}_k(\mathbf{x}) V^T) = \hat{F}(U D^{k\times k}_k(\mathbf{x}) V^T) + \lambda \|\mathbf{x}\|^2,
\end{equation}
for all $(U,\mathbf{x},V) \in V_k(\R^m)\times \R^k \times V_k(\R^n)$. 
\end{definition}

\begin{lemma}\label{lemma-g-eta-gamma-RWLRA-1-expectation}
Let $\hat{F}: \R^{m\times n} \rightarrow \R$ be as in Problem \ref{prob-WLRA}, $F: \R^{m\times n} \rightarrow \R$ as in Problem \ref{prob-RWLRA-1-reform} and \\ $G: V_k(\R^m)\times \R^k \times V_k(\R^n) \rightarrow \R$ as in Definition \ref{def-func-G}. For $P\in \R^{m\times n}$ and $(U,\mathbf{x},V) \in V_k(\R^m)\times \R^k \times V_k(\R^n)$, we have
\begin{equation*}
E_{(\eta,\gamma) \sim \mu}(\hat{f} (P; \eta,\gamma)) = \hat{F}(P), \
E_{(\eta,\gamma) \sim \mu}(f(P; \eta,\gamma)) = F(P) \ \text{and }
E_{(\eta,\gamma) \sim \mu}(g(U,\mathbf{x},V; \eta,\gamma)) = G(U,\mathbf{x},V),
\end{equation*}
where the expectations are taken over the probability space $\Omega_{m,n}$ with respect to the probability distribution $\mu$ given in Definition \ref{def-Omega-m-n-measure}. That is, $\hat{f}$, $f$ and $g$ are random functions with expectations $\hat{F}$, $F$ and $G$.
\end{lemma}

Since $V_k(\R^m)\times \R^k \times V_k(\R^n)$ is a submanifold of the Euclidean space $\R^{m\times k} \times \R^{k} \times \R^{n\times k}$, we use the following lemmas to find the gradient of $g_{\eta,\gamma}$ given in Equation \eqref{eq-def-random-function-eta-gamma-RWLRA-1} .

\begin{lemma}(\cite[Equation (3.37)]{AMS}\label{lemma-gradient-submfd})
Let $\overline{M}$ be a Riemannian manifold and $f: \overline{M} \rightarrow \R$ a differentiable function. Assume that $M$ is a Riemannian submanifold of $\overline{M}$. For any $x \in M$, denote by $\pi_x:T_x\overline{M} \rightarrow T_x M$ the orthogonal projection. Then $\nabla (f|_M)(x) = \pi_x(\nabla f(x))$ for any $x \in M$.
\end{lemma}

\begin{lemma}(\cite[Example 3.6.2]{AMS}\label{lemma-stiefel-orthogonal-proj})
For any $X \in V_k(\R^n)$, denote by $\Pi_X$ the orthogonal projection $\Pi_X: \R^{n\times k} \rightarrow T_X V_k(\R^n)$. Then, for any $\xi \in \R^{n\times k}$, $\Pi_X(\xi)= (I_n -XX^T)\xi +\frac{1}{2}X(X^T\xi-\xi^TX) = \xi - \frac{1}{2}X(X^T\xi + \xi^T X)$.
\end{lemma}

With the orthogonal projection $\Pi_X: \R^{n\times k} \rightarrow T_X V_k(\R^n)$ given in Lemma \ref{lemma-stiefel-orthogonal-proj}, we are ready to give the gradient of the function $g_{\eta,\gamma}: V_k(\R^m)\times \R^k \times V_k(\R^n) \rightarrow \R$.

\begin{corollary}\label{cor-random-g-gradient-RWLRA-1}
Given $(\eta, \gamma)\in \Omega_{m,n}$, define the matrices
\begin{eqnarray*}
\nabla_U \hat{f}_{\eta,\gamma} & = & \left[\frac{\partial \hat{f}_{\eta,\gamma}}{\partial u_{i,l}}\right]_{m\times k}= \left[-2\delta_{\eta,i} (a_{\eta,\gamma}-p_{\eta,\gamma})x_l v_{\gamma,l}\right]_{m\times k}, \\
\nabla_V \hat{f}_{\eta,\gamma} & = & \left[\frac{\partial \hat{f}_{\eta,\gamma}}{\partial v_{j,l}}\right]_{n\times k} =  \left[-2\delta_{\gamma,j} (a_{\eta,\gamma}-p_{\eta,\gamma})x_l u_{\eta,l}\right]_{n\times k}, \\
\nabla_{\mathbf{x}} \hat{f}_{\eta,\gamma} & = & \left[\begin{array}{c}
\frac{\partial \hat{f}_{\eta,\gamma}}{\partial x_{1}} \\
\frac{\partial \hat{f}_{\eta,\gamma}}{\partial x_{2}} \\
\vdots \\
\frac{\partial \hat{f}_{\eta,\gamma}}{\partial x_{k}}
\end{array}\right]
= \left[\begin{array}{c}
-2 (a_{\eta,\gamma}-p_{\eta,\gamma})u_{\eta,1} v_{\gamma,1} \\
-2 (a_{\eta,\gamma}-p_{\eta,\gamma})u_{\eta,2} v_{\gamma,2} \\
\vdots \\
-2 (a_{\eta,\gamma}-p_{\eta,\gamma})u_{\eta,k} v_{\gamma,k}
\end{array}\right],
\end{eqnarray*}
where $\delta_{p,q} = \begin{cases} 1 & \text{if }p=q, \\ 0  & \text{if }p\neq q.\end{cases}$ Then the gradient of $g_{\eta,\gamma}$ at any $(U,\mathbf{x},V) \in V_k(\R^m)\times \R^k \times V_k(\R^n)$ is 
\begin{equation}\label{eq-g-gradient-RWLRA-1}
\nabla g_{\eta,\gamma}(U,\mathbf{x},V) = (\Pi_U(\nabla_U \hat{f}_{\eta,\gamma}), \nabla_{\mathbf{x}} \hat{f}_{\eta,\gamma} + 2\lambda \mathbf{x}, \Pi_V(\nabla_V \hat{f}_{\eta,\gamma})) \in T_U V_k(\R^m)\times \R^k \times T_V V_k(\R^n),
\end{equation}
where $\Pi_U$ and $\Pi_V$ are the orthogonal projections given in Lemma \ref{lemma-stiefel-orthogonal-proj}.
\end{corollary}

Next we define a confinement function for the random function $g$ given in Equation \eqref{eq-def-hat-g-eta-gamma-RWLRA-1}.

\begin{definition}\label{def-confinement-RWLRA-1}
Define $\rho : V_k(\R^m)\times \R^k \times V_k(\R^n)\rightarrow \R$ by 
\begin{equation}\label{eq-def-confinement-RWLRA-1}
\rho(U,\mathbf{x},V) = \|\mathbf{x}\|^2
\end{equation}
for all $(U,\mathbf{x},V) \in V_k(\R^m)\times \R^k \times V_k(\R^n)$. 
\end{definition}

\begin{lemma}\label{lemma-rho-RWLRA-1-confinement}
For $A=[a_{i,j}] \in \R^{m\times n}$ given in Problem \ref{prob-RWLRA-1-reform}, define
\begin{equation}\label{eq-def-alpha-RWLRA-1}
\alpha: =\max\{a_{i,j}^2 ~\big{|}~ (i,j) \in \Omega_{m,n}\} \geq 0,
\end{equation}
and 
\begin{equation}\label{eq-def-rho-0-RWLRA-1}
\rho_0: = \max\{\|\mathbf{x}_0\|^2, \frac{\alpha}{4\lambda}\} \geq \frac{\alpha}{4\lambda},
\end{equation}
where $\lambda > 0$ is given in Problem \ref{prob-RWLRA-1-reform} and $\mathbf{x}_0$ is given by the initial iterate $(U_0,\mathbf{x}_0,V_0)$ of Algorithm \ref{alg-confined-SGD-RWLRA-1-direct} below. With $\rho_0$ given in Equation \eqref{eq-def-rho-0-RWLRA-1}, the function $\rho$ in Definition \ref{def-confinement-RWLRA-1} and the random function $g$ in Definition \ref{def-random-functions-eta-gamma-RWLRA-1} satisfy that
\begin{itemize}
    \item $\rho(U_0,\mathbf{x}_0,V_0) \leq \rho_0$, 
    \item $\langle \nabla \rho(U,\mathbf{x},V), \nabla g_{\eta, \gamma}(U,\mathbf{x},V) \rangle \geq 0$ for $(\eta,\gamma) \in \Omega_{m,n}$ and $(U,\mathbf{x},V)\in V_k(\R^m)\times \R^k \times V_k(\R^n)$ satisfying $\rho(U,\mathbf{x},V) \geq \rho_0$.
\end{itemize}
That is, the function $\rho$ is a confinement function for the random function $g$ in Definition \ref{def-random-functions-eta-gamma-RWLRA-1} on the manifold $V_k(\R^m)\times \R^k \times V_k(\R^n)$, and $\rho_0$ satisfies Assumptions (\ref{assumption-confinement-function}) and (\ref{assumption-initial}) of Theorem \ref{thm-confined-SGD}.
\end{lemma}

Now we fix the scalars $a$ and $b$ in Assumption (\ref{assumption-parameters}) of Theorem \ref{thm-confined-SGD} as:
\begin{equation}\label{eq-a-b-choices-RWRLA-1}
a = \frac{1}{\lambda} \text{ and } b = \frac{1}{\sqrt{\lambda}},
\end{equation}
where $\lambda$ is given in Problem \ref{prob-RWLRA-1-reform}. Fix the sequence $\{c_t\}_{t=0}^\infty$ in Assumption (\ref{assumption-parameters}) of Theorem \ref{thm-confined-SGD} as $c_t = \frac{1}{t+1}$. Thus, $c = \max \{c_t~\big{|}~t\geq 0\} = 1$ and $\sigma = \sum_{t=0}^\infty c_t^2 = \frac{\pi^2}{6}$. Further, we choose a positive scalar $\Phi_{\min}$ satisfying
\begin{equation}\label{eq-phi-min-choice-RWRLA-1}
\Phi_{\min} = K \max\left\{(\lambda +2\sqrt{\lambda}+1)\alpha, \sqrt{32k\alpha\lambda+ 8k(2 + \lambda^2)
\left(2\lambda\rho_0 + \frac{\pi^2 + 12}{6}\right)}\right\}
\end{equation}
where $\alpha$ is given in Equation \eqref{eq-def-alpha-RWLRA-1}, $\rho_0$ is given in Equation \eqref{eq-def-rho-0-RWLRA-1} 
and $K \geq 1$ is one of the constant inputs for Algorithm \ref{alg-confined-SGD-RWLRA-1-direct} below. Fix positive scalar $\Theta$ in Assumption (\ref{assumption-parameters}) of Theorem \ref{thm-confined-SGD} as
\begin{equation}\label{eq-theta-choice-RWRLA-1}
\Theta = \frac{1}{\Phi_{\min}}.
\end{equation}

With the preparations above, we are now ready to apply Theorem \ref{thm-confined-SGD} to Problem \ref{prob-RWLRA-1-reform}.

\begin{algorithm}[A Stochastic Gradient Descent for Problem \ref{prob-RWLRA-1-reform}]\label{alg-confined-SGD-RWLRA-1-direct} \

\noindent\makebox[\linewidth]{\rule{\textwidth}{1pt}}

\textbf{Input:}  
\begin{itemize}
    \item[-] the random function $g$ given in Definition \ref{def-random-functions-eta-gamma-RWLRA-1},
    \item[-] the retraction $R$ given in Definition \ref{def-retraction-GS-prod}, 
    \item[-] the positive integers $m, n$ and $k$ given in Problem \ref{prob-RWLRA-1-reform}, 
    \item[-] the positive scalar $\lambda$ given in Problem \ref{prob-RWLRA-1-reform}, 
    \item[-] the matrices $A$ and $W$ given in Problem \ref{prob-RWLRA-1-reform}, 
    \item[-] a scalar $K \geq 1$,
    \item[-] an initial iterate $(U_0,\mathbf{x}_0,V_0) \in V_k(\R^m)\times \R^k \times V_k(\R^n)$. 
\end{itemize}

\textbf{Output:} A sequence of iterates $\{(U_t,\mathbf{x}_t,V_t)\}_{t=0}^\infty  \subset V_k(\R^m)\times \R^k \times V_k(\R^n)$.
\begin{itemize}
	\item \emph{for $t=0,1,2\dots$ do}
	\begin{enumerate}[1.]
        \item Select a random element $(\eta_t, \gamma_t)$ from $\Omega_{m, n}$ with the probability distribution $\mu$ independent of $\{(\eta_{\tau}, \gamma_{\tau})\}_{\tau=0}^{t-1}$.
	\item Set
	\begin{equation}\label{eq-confined-SGD-RWLRA-1-recursion}
	(U_{t+1},\mathbf{x}_{t+1},V_{t+1})= R_{(U_t,\mathbf{x}_t,V_t)}\left(-\frac{1}{(t+1)\Phi_{\min}}\nabla g_{\eta_t,\gamma_t}(U_t,\mathbf{x}_t,V_t)\right),
	\end{equation}
        where $\nabla g_{\eta_t,\gamma_t}(U_t,\mathbf{x}_t,V_t)$ is given in Corollary \ref{cor-random-g-gradient-RWLRA-1} and $\Phi_{\min}$ is given in Equation \eqref{eq-phi-min-choice-RWRLA-1}.
	\end{enumerate}
	\item \emph{end for}
\end{itemize}
\noindent\makebox[\linewidth]{\rule{\textwidth}{1pt}}
\end{algorithm}

\begin{proposition}\label{prop-SGD-mfd-RWLRA-1-direct}
Let $G$ be the function given in Equation \eqref{eq-def-G-RWLRA-1} and ${\{(U_t,\mathbf{x}_t,V_t)\}_{t=0}^\infty \subset V_k(\R^m)\times \R^k \times V_k(\R^n)}$ be the sequence from the Algorithm \ref{alg-confined-SGD-RWLRA-1-direct}. Then:
\begin{enumerate}
	\item $\{(U_t,\mathbf{x}_t,V_t)\}_{t=0}^\infty$ is contained in the compact subset ${\{(U,\mathbf{x},V)\in V_k(\R^m)\times
    \R^k \times V_k(\R^n) ~\big{|}~\|\mathbf{x}\|^2\leq \rho_0 + \frac{\pi^2 + 12}{12\lambda}\}}$ of $V_k(\R^m)\times \R^k \times V_k(\R^n)$, where $\rho_0$ is given in Equation \eqref{eq-def-rho-0-RWLRA-1} and $\lambda$ in Problem \ref{prob-RWLRA-1-reform}. Therefore, $\{(U_t,\mathbf{x}_t,V_t)\}_{t=0}^\infty$ has convergent subsequences;
	\item $\{G(U_t,\mathbf{x}_t,V_t)\}_{t=0}^\infty$ converges almost surely to a finite number;
	\item $\{\|\nabla G(U_t,\mathbf{x}_t,V_t)\|\}_{t=0}^\infty$ converges almost surely to $0$;
	\item any limit point of $\{(U_t,\mathbf{x}_t,V_t)\}_{t=0}^\infty$ is almost surely a stationary point of $G$.
\end{enumerate}
\end{proposition}

\begin{remark}\label{rk-SGD-mfd-RWLRA-1-single-vf-t}
Proposition \ref{prop-SGD-mfd-RWLRA-1-direct} follows from Corollary \ref{cor-confined-SGD} with
\[
\vf = \Phi_{\min} = K \max\left\{(\lambda +2\sqrt{\lambda}+1)\alpha, \sqrt{32k\alpha\lambda+ 8k(2 + \lambda^2)
(2\lambda\rho_0 + \frac{\pi^2 + 12}{6})}\right\}.
\]
Theoretically, we can hold $K$ constant for different $\lambda$ and still have convergent stochastic gradient descents. 
In this way, we can make
\begin{equation}\label{eq-def-vf-t-RWLRA-1-single-vf-t}
\vf  = O(1) \text{ as } \lambda \rightarrow 0^+,
\end{equation}
which is not true for stochastic gradient descents on $\R^{m\times k}\times \R^{n\times k}$ (Algorithm \ref{alg-confined-SGD-RWLRA-2-direct}). In practice, we choose different $K$ for different $\lambda$ to optimize the performance of Algorithm \ref{alg-confined-SGD-RWLRA-1-direct}.
\end{remark}

\section{Numerical Results}\label{sec-numerical-results}
In this section, we evaluate the performances of Algorithms \ref{alg-confined-SGD-RWLRA-1-direct} and \ref{alg-ALS-RWLRA-1} on the Matrix Completion Problem. A well known example of the Matrix Completion Problem is the Netflix Prize. Netflix provided a training data set of $100480507$ ratings that $480189$ users gave to $17770$ movies, with each rating being a integer from 1 to 5 (\cite{kaggle-Netflix-Prize:2024}). Our goal is to best approximate the missing ratings with a low-rank assumption. This data set is essentially a matrix in $\R^{480189 \times 17770}$ with $100480507$ observed entries, while the other entries of the matrix are missing. Each row of the matrix consists of a single user's ratings on all the movies, each column of the matrix consists of a single movie's ratings from all the users, and each observed entry corresponds to a rating in the data set. The percentage of observed entries in the matrix of all ratings is about $1.17\%$. 

All the numerical results in this section are generated on a Macbook Air with M3 Apple silicon chip and $8$ gigabytes memory. Due to limits to our computing power, instead of using the whole matrix, we randomly sampled a submatrix $A = [a_{i,j}] \in \R^{27000\times 1000}$ with $278338$ observed entries from the matrix of all ratings. The percentage of observed entries in $A$ is about $1.03\%$, which is close to the percentage of observed entries in the matrix of all ratings. We fix the low-rank constraint as $k = 32$, and fix the matrix of weights $W = [w_{i,j}] \in \R^{27000\times 1000}$ by setting $w_{i,j} = 0$, if $a_{i,j}$ is missing, and $w_{i,j} = \frac{1}{278338}$, if $a_{i,j}$ is observed. With these $A$, $W$ and $k$, we estimate the solution to Problem \ref{prob-RWLRA-1-reform}, by applying Algorithms \ref{alg-confined-SGD-RWLRA-1-direct} and \ref{alg-ALS-RWLRA-1}. To apply the stochastic gradient descent and the accelerated line search on Euclidean spaces to the Matrix Completion Problem as benchmarks, we recall another regularization of Problem \ref{prob-WLRA}. In Ban, Woodruff, and Zhang (2019), Problem \ref{prob-WLRA} is regularized into Problem \ref{prob-RWLRA-2} below:

\begin{problem}[Another Regularized Weighted Low-Rank Approximation Problem]\label{prob-RWLRA-2}
Under Assumptions (\ref{problem-assumption-1})-(\ref{problem-assumption-3}) in Problem \ref{prob-WLRA}, fix a positive number
$\lambda$ and define $H:\R^{m\times k} \times \R^{n\times k} \rightarrow \R$ by 
\begin{equation}\label{def-function-H-RWLRA-2}
H(X,Y) = \sum_{i=1}^m \sum_{j=1}^n w_{i,j}(a_{i,j}-\sum_{l=1}^k x_{i,l}y_{j,l})^2+ \lambda (\|X\|_F^2+\|Y\|_F^2)
\end{equation}
for $X=[x_{i,j}] \in \R^{m\times k}$ and $Y=[y_{i,j}] \in \R^{n\times k}$. Where $\|X\|_F :=\sqrt{\sum_{i=1}^m \sum_{j=1}^k x_{i,j}^2}$ and ${\|Y\|_F :=\sqrt{\sum_{i=1}^n \sum_{j=1}^k y_{i,j}^2}}$ are the Frobenius norms of $X$ and $Y$.
Solve for
\[
\mathrm{argmin}\{H(X,Y) ~\big{|}~ (X,Y)  \in \R^{m\times k} \times \R^{n\times k} \}.
\]
\end{problem}

Problem \ref{prob-RWLRA-2} is formulated on the Euclidean space $\R^{m\times k} \times \R^{n\times k}$, which significantly simplifies the differential geometry involved. We estimate the solution to this problem by the stochastic gradient descent (Algorithm \ref{alg-confined-SGD-RWLRA-2-direct}) and the accelerated line search (Algorithm \ref{alg-ALS-RWLRA-2}) on $\R^{m\times k} \times \R^{n\times k}$, with the same $A$, $W$ and $k$ given above. 

To initialize algorithms evaluated in this section, we fill in the missing entries of $A \in \R^{27000\times 1000} $ with the column average of the observed entries. We find that $\rank A = 992$. The Reduced Singular Value Decomposition of matrix $A$ is $U = {[\mathbf{u}_1,\dots, \mathbf{u}_{992}]} \in V_{992}(\R^{27000})$, $V = {[\mathbf{v}_1,\dots, \mathbf{v}_{992}]} \in V_{992}(\R^{1000})$ and $\mathbf{x} = [s_1, \ldots, s_{992}]^T \in \R^{992}$ satisfying $A = U D^{992\times 992}_{992}(\mathbf{x}) V^T$, where $s_1, \ldots, s_{992}$ are the positive singular values of $A$ in descending order and $D^{992\times 992}_{992}: \R^{992} \rightarrow \R^{992\times 992}$ is given by Equation \eqref{eq-def-D-k-by-k}. Then, we set
\begin{equation}\label{eq-def-eyt-initial}
U_0 = {[\mathbf{u}_1,\dots, \mathbf{u}_{32}]} \in V_{32}(\R^{27000}), V_0 = {[\mathbf{v}_1,\dots, \mathbf{v}_{32}]} \in V_{32}(\R^{1000}) \text{ and } \mathbf{x}_0 = [s_1, \ldots, s_{32}]^T \in \R^{32}, 
\end{equation}
and use $(U_0, \mathbf{x}_0 , V_0)$ given by Equation \eqref{eq-def-eyt-initial} as the initial iterate of Algorithms \ref{alg-confined-SGD-RWLRA-1-direct} and \ref{alg-ALS-RWLRA-1}. Let $P_0 = U_0 D^{32\times 32}_{32}(\mathbf{x}_0) V_0^T \in \R^{27000\times 1000}$, where $D^{32\times 32}_{32}: \R^{32} \rightarrow \R^{32\times 32}$ is given by Equation \eqref{eq-def-D-k-by-k}. By the Eckart-Young Theorem (Theorem \ref{thm-eyt}), $P_0$ minimizes $\|A - P\|_F^2$ under the constraint $\rank P \leq 32$ with $A$ given above. 
Further, we set
\begin{equation}\label{eq-def-another-eyt-initial}
X_0 = U_0 \sqrt{D^{32\times 32}_{32}(\mathbf{x}_0)} \in \R^{27000\times 32} \text{ and } Y_0 = V_0 \sqrt{D^{32\times 32}_{32}(\mathbf{x}_0)} \in \R^{1000\times 32},
\end{equation}
where $U_0$, $\mathbf{x}_0$ and $V_0$ are given in Equation \eqref{eq-def-eyt-initial} and $\sqrt{D^{32\times 32}_{32}(\mathbf{x}_0)} \in \R^{32\times 32}$ is given by 
\[
\sqrt{D^{32\times 32}_{32}(\mathbf{x}_0)} = \sqrt{D^{32\times 32}_{32}([s_1, \ldots, s_{32}]^T)} = [d_{i,j}], \text{ with } d_{i,j}=\begin{cases} \sqrt{s_i} & \text{if } i=j, \\ 0 & \text{otherwise}. \end{cases}
\]
We use $(X_0, Y_0)$ given by Equation \eqref{eq-def-another-eyt-initial} as the initial iterate of Algorithms \ref{alg-confined-SGD-RWLRA-2-direct} and \ref{alg-ALS-RWLRA-2}.

Let $\hat{F}$ be as in Problem \ref{prob-WLRA}. We refer to $\hat{F}$ as the ``cost function with regularization removed" in all the figures of this section. Since our goal is to compare the performances of Algorithms \ref{alg-confined-SGD-RWLRA-1-direct}, \ref{alg-ALS-RWLRA-1}, \ref{alg-confined-SGD-RWLRA-2-direct} and \ref{alg-ALS-RWLRA-2} on Problem \ref{prob-WLRA}, we plot the sequences $\{\hat{F}(U_t D^{32\times 32}_{32}(\mathbf{x}_t) V_t^T)\}$ and $\{\hat{F}(X_t Y_t^T)\}$ in all the figures of this section instead of the values of the regularized functions. Moreover, with initial iterates $(U_0, \mathbf{x}_0, V_0)$ given by Equation \eqref{eq-def-eyt-initial} and $(X_0, Y_0)$ given by Equation \eqref{eq-def-another-eyt-initial}, the sequences $\{\hat{F}(U_t D^{32\times 32}_{32}(\mathbf{x}_t) V_t^T)\}$ and $\{\hat{F}(X_t Y_t^T)\}$ share the same initial value, but the regularized functions do not.

\begin{figure}[ht]
\centering
\includegraphics[width=\textwidth, height=4.5cm]{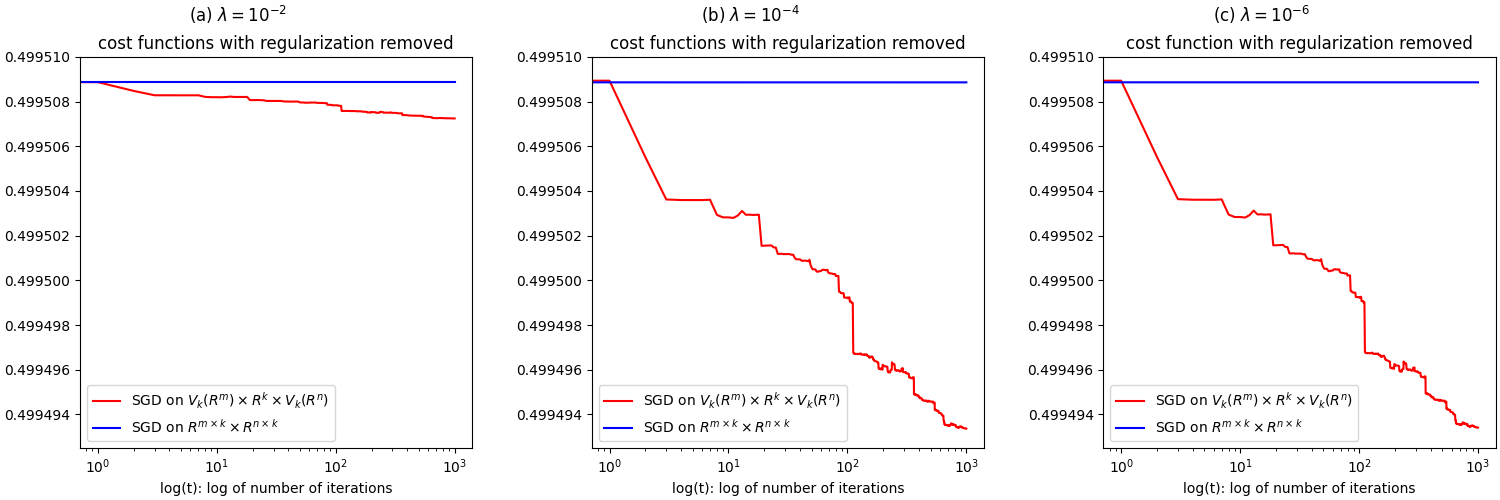}
\caption{The performance profiles of stochastic gradient descent Algorithms \ref{alg-confined-SGD-RWLRA-1-direct} and \ref{alg-confined-SGD-RWLRA-2-direct}}
\label{fig-comparing-CSGD-alg-netflix}
\end{figure}

In Figure \ref{fig-comparing-CSGD-alg-netflix}, we present numerical results comparing the performances of Algorithms \ref{alg-confined-SGD-RWLRA-1-direct} and \ref{alg-confined-SGD-RWLRA-2-direct} over the same number of iterations since the runtime of each iteration of these algorithms is similar. We denote by $t$ the number of iterations. Algorithm \ref{alg-confined-SGD-RWLRA-1-direct} is referred to as the ``SGD on $V_k(\R^m)\times \R^k \times V_k(\R^n)$", and Algorithm \ref{alg-confined-SGD-RWLRA-2-direct} is referred to as the ``SGD on $\R^{m\times k}\times \R^{n\times k}$". In the subfigures (a), (b) and (c) of Figure \ref{fig-comparing-CSGD-alg-netflix}, 
\begin{itemize}
    \item the scalar $\lambda$ in Problems \ref{prob-RWLRA-1-reform} and \ref{prob-RWLRA-2} is fixed as $10^{-2}$, $10^{-4}$ and $10^{-6}$, 
    \item the scalar $K$ in Algorithm \ref{alg-confined-SGD-RWLRA-1-direct} is fixed as $10^3$, $10^3$ and $10^4$ to optimize the performance of the algorithm, 
    \item the scalar $K$ in Algorithm \ref{alg-confined-SGD-RWLRA-2-direct-specific} is fixed as $10^4$, $1$ and $1$ to optimize the performance of the algorithm. 
\end{itemize}
In each subfigure, 
\begin{itemize}
    \item the red curve depicts the sequence $\{\hat{F}(U_t D^{32\times 32}_{32}(\mathbf{x}_t) V_t^T)\}_{t=0}^{1000}$ over the log (with base $10$) of $t$ with the sequence $\{(U_t,\mathbf{x}_t,V_t)\}_{t=0}^{1000}$ from Algorithm \ref{alg-confined-SGD-RWLRA-1-direct},
    \item the blue curve depicts the sequence $\{\hat{F}(X_t Y_t^T)\}_{t=0}^{1000}$ over the log (with base $10$) of $t$ with the sequence $\{(X_t,Y_t)\}_{t=0}^{1000}$ from Algorithm \ref{alg-confined-SGD-RWLRA-2-direct}. 
\end{itemize}
It appears that Algorithm \ref{alg-confined-SGD-RWLRA-1-direct} outperforms Algorithm \ref{alg-confined-SGD-RWLRA-2-direct}.
\begin{figure}[ht]
\centering
\includegraphics[width=\textwidth, height=4.5cm]{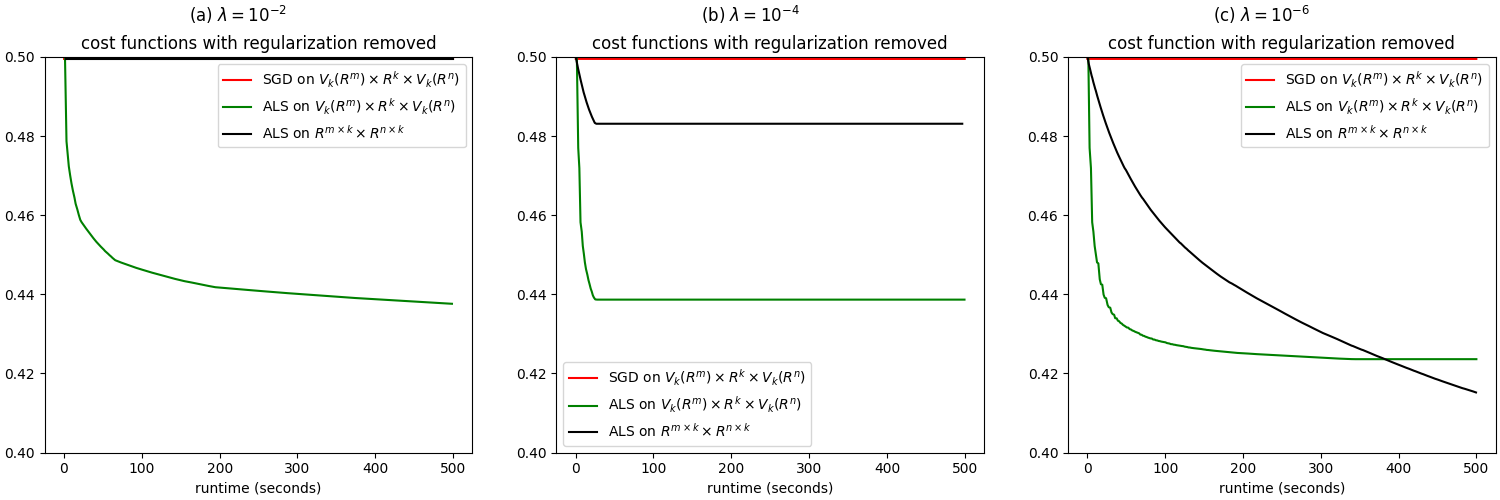}
\caption{The performance profiles of Algorithms \ref{alg-confined-SGD-RWLRA-1-direct}, \ref{alg-ALS-RWLRA-1} and \ref{alg-ALS-RWLRA-2}}
\label{fig-comparing-runtime}
\end{figure}

In Figure \ref{fig-comparing-runtime}, we present numerical results comparing the performances of Algorithms \ref{alg-confined-SGD-RWLRA-1-direct}, \ref{alg-ALS-RWLRA-1} and \ref{alg-ALS-RWLRA-2} over the same runtime. Since each iteration of Algorithm \ref{alg-ALS-RWLRA-1} or Algorithm \ref{alg-ALS-RWLRA-2} takes much more time than Algorithm \ref{alg-confined-SGD-RWLRA-1-direct}, it is more practical to compare the performances of these three algorithms over the same runtime rather than the same number of iterations. Algorithm \ref{alg-confined-SGD-RWLRA-1-direct} is referred to as the ``SGD on $V_k(\R^m)\times \R^k \times V_k(\R^n)$", Algorithm \ref{alg-ALS-RWLRA-1} is referred to as the ``ALS on $V_k(\R^m)\times \R^k \times V_k(\R^n)$", and Algorithm \ref{alg-ALS-RWLRA-2} is referred to as the ``ALS on $\R^{m\times k}\times \R^{n\times k}$". To implement the Armijo step size in Algorithms \ref{alg-ALS-RWLRA-1} and \ref{alg-ALS-RWLRA-2}, we always fix scalars $\overline{\alpha}=1$ and $\beta= 0.5$. In the subfigures (a), (b) and (c) of Figure \ref{fig-comparing-runtime},
\begin{itemize}
    \item the scalar $\lambda$ in Problems \ref{prob-RWLRA-1-reform} and \ref{prob-RWLRA-2} is fixed as $10^{-2}$, $10^{-4}$ and $10^{-6}$,
    \item the scalar $K$ in Algorithm \ref{alg-confined-SGD-RWLRA-1-direct} is again fixed as $10^3$, $10^3$ and $10^4$ to optimize the performance of the algorithm,
    \item the scalar $\iota$ in Algorithms \ref{alg-ALS-RWLRA-1} and \ref{alg-ALS-RWLRA-2} is fixed as $\frac{108}{270000}$, $\frac{11}{270000000}$ and $\frac{1}{54000000000}$ to optimize the performances of these algorithms.
\end{itemize}
In each subfigure, 
\begin{itemize}
    \item the red curve depicts the sequence $\{\hat{F}(U_t D^{32\times 32}_{32}(\mathbf{x}_t) V_t^T)\}$ over the runtime with the sequence $\{(U_t,\mathbf{x}_t,V_t)\}$ from Algorithm \ref{alg-confined-SGD-RWLRA-1-direct}, 
    \item the green curve depicts the sequence $\{\hat{F}(U_t D^{32\times 32}_{32}(\mathbf{x}_t) V_t^T)\}$ over the runtime with the sequence $\{(U_t,\mathbf{x}_t,V_t)\}$ from Algorithm \ref{alg-ALS-RWLRA-1},
    \item the black curve depicts the sequence $\{\hat{F}(X_t Y_t^T)\}$ over the runtime with the sequence $\{(X_t,Y_t)\}$ from Algorithm \ref{alg-ALS-RWLRA-2}. 
\end{itemize}
It appears that
\begin{enumerate}
    \item Algorithm \ref{alg-ALS-RWLRA-1} outperforms Algorithm \ref{alg-confined-SGD-RWLRA-1-direct},
    \item Algorithm \ref{alg-ALS-RWLRA-2} outperforms Algorithm \ref{alg-confined-SGD-RWLRA-1-direct} except for $\lambda = 10^{-2}$,
    \item Algorithm \ref{alg-ALS-RWLRA-1} converges faster than Algorithm \ref{alg-ALS-RWLRA-2},
    \item for $\lambda = 10^{-6}$, Algorithm \ref{alg-ALS-RWLRA-2} converges to a local minimum with lower value of $\hat{F}$ than Algorithm \ref{alg-ALS-RWLRA-1}.
\end{enumerate}

Combining Figures \ref{fig-comparing-CSGD-alg-netflix} and \ref{fig-comparing-runtime}, we conclude that, when estimating the solution to the Matrix Comparison Problem with sample data from the Netflix Prize training data, Algorithms \ref{alg-confined-SGD-RWLRA-1-direct} outperforms the existing stochastic gradient descent on Euclidean spaces, and Algorithm \ref{alg-ALS-RWLRA-1} converges faster than the existing accelerated line search on Euclidean spaces but converges to a local minimum with larger cost function value for sufficiently small $\lambda$. The Python programs generating the numerical results in this section are available via: \url{https://github.com/ConglongXu-GWU/rie-sgd-code}.

\newpage

\raggedright    
\bibliographystyle{acm}
\bibliography{refs}  

\begin{thebibliography}{10}

\bibitem{AMS}
{\sc Absil, P., Mahony, R., and Sepulchre, R.}
\newblock {\em Optimization Algorithms on Matrix Manifolds}.
\newblock Princeton University Press, 2009.

\bibitem{Ban-Woodruff-Zhang:2019}
{\sc Ban, F., Woodruff, D., and Zhang, R.}
\newblock Regularized weighted low rank approximation.
\newblock In {\em Advances in Neural Information Processing Systems\/} (2019),
  H.~Wallach, H.~Larochelle, A.~Beygelzimer, F.~d\textquotesingle
  Alch\'{e}-Buc, E.~Fox, and R.~Garnett, Eds., vol.~32, Curran Associates, Inc.

\bibitem{Bertsimas-Cory-Lo-Pauphilet:2024}
{\sc Bertsimas, D., Cory-Wright, R., Lo, S., and Pauphilet, J.}
\newblock Optimal low-rank matrix completion: Semidefinite relaxations and
  eigenvector disjunctions, 2024.

\bibitem{Bonnabel:2013}
{\sc Bonnabel, S.}
\newblock Stochastic gradient descent on {R}iemannian manifolds.
\newblock {\em IEEE Trans. Automat. Control 58}, 9 (2013), 2217--2229.

\bibitem{Bordenave-Coste-Nadakuditi:2023}
{\sc Bordenave, C., Coste, S., and Nadakuditi, R.~R.}
\newblock Detection thresholds in very sparse matrix completion.
\newblock {\em Found. Comput. Math. 23}, 5 (2023), 1619--1743.

\bibitem{Bottou}
{\sc Bottou, L.}
\newblock {\em On-line Learning and Stochastic Approximations}.
\newblock Publications of the Newton Institute. Cambridge University Press,
  1999, p.~9–42.

\bibitem{Boumal-Absil:2011}
{\sc Boumal, N., and Absil, P.}
\newblock Rtrmc: A riemannian trust-region method for low-rank matrix
  completion.
\newblock In {\em Advances in Neural Information Processing Systems\/} (2011),
  J.~Shawe-Taylor, R.~Zemel, P.~Bartlett, F.~Pereira, and K.~Weinberger, Eds.,
  vol.~24, Curran Associates, Inc.

\bibitem{Chakraborty-Dey:2023}
{\sc Chakraborty, D., and Dey, S.}
\newblock Matrix completion: approximating the minimum diameter.
\newblock In {\em 34th {I}nternational {S}ymposium on {A}lgorithms and
  {C}omputation}, vol.~283 of {\em LIPIcs. Leibniz Int. Proc. Inform.} Schloss
  Dagstuhl. Leibniz-Zent. Inform., Wadern, 2023, pp.~Art. No. 17, 19.

\bibitem{Gillis-Glineur:2011}
{\sc Gillis, N., and Glineur, F.}
\newblock Low-rank matrix approximation with weights or missing data is
  np-hard.
\newblock {\em SIAM Journal on Matrix Analysis and Applications 32}, 4 (2011),
  1149--1165.

\bibitem{kaggle-Netflix-Prize:2024}
{\sc Kaggle}.
\newblock {N}etflix {P}rize data --- kaggle.com, 2024.
\newblock [accessed 01-12-2024].

\bibitem{Kelner-Li-Liu-Sidford-Tian:2023}
{\sc Kelner, J.~A., Li, J., Liu, A., Sidford, A., and Tian, K.}
\newblock Matrix completion in almost-verification time.
\newblock In {\em 2023 {IEEE} 64th {A}nnual {S}ymposium on {F}oundations of
  {C}omputer {S}cience---{FOCS} 2023}. IEEE Computer Soc., Los Alamitos, CA,
  2023, pp.~2102--2128.

\bibitem{Lee:2024}
{\sc Lee, G.}
\newblock Smooth singular value thresholding algorithm for low-rank matrix
  completion problem.
\newblock {\em J. Korean Math. Soc. 61}, 3 (2024), 427--444.

\bibitem{Li-Zhen-Pan-Zhao:2023}
{\sc Li, H., Peng, Z., Pan, C., and Zhao, D.}
\newblock Fast gradient method for low-rank matrix estimation.
\newblock {\em J. Sci. Comput. 96}, 2 (2023), Paper No. 41, 35.

\bibitem{Srebro-Jaakkola:2003}
{\sc Srebro, N., and Jaakkola, T.}
\newblock Weighted low-rank approximations.
\newblock In {\em Proceedings of the Twentieth International Conference on
  International Conference on Machine Learning\/} (2003), ICML'03, AAAI Press,
  p.~720–727.

\bibitem{Wan-Cheng:2023}
{\sc Wan, X., and Cheng, G.}
\newblock Weighted hybrid truncated norm regularization method for low-rank
  matrix completion.
\newblock {\em Numer. Algorithms 94}, 2 (2023), 619--641.

\bibitem{Wikipedia-Low-Rank-Approximation:2024}
{\sc Wikipedia}.
\newblock {Low-rank approximation} --- {W}ikipedia{,} the free encyclopedia,
  2024.
\newblock [Online; accessed 20-July-2024].

\bibitem{Wu:2021}
{\sc Wu, H.}
\newblock A stochastic gradient descent theorem and the back-propagation
  algorithm, 2021.

\bibitem{Yamada-Sato:2023}
{\sc Yamada, M., and Sato, H.}
\newblock Conjugate gradient methods for optimization problems on symplectic
  {S}tiefel manifold.
\newblock {\em IEEE Control Syst. Lett. 7\/} (2023), 2719--2724.

\bibitem{Yan-Tang-Li:2024}
{\sc Yan, X., Tang, X., and Li, C.}
\newblock An improved inertial alternating direction method for low rank matrix
  completion problems.
\newblock {\em Math. Numer. Sin. 46}, 2 (2024), 144--155.

\bibitem{Yan-Zhang:2024}
{\sc Yan, X., and Zhang, N.}
\newblock A modified primal-dual algorithm for matrix completion problems.
\newblock {\em Acta Math. Appl. Sin. 47}, 2 (2024), 175--192.

\bibitem{Yang-Ma:2023}
{\sc Yang, Y., and Ma, C.}
\newblock Optimal tuning-free convex relaxation for noisy matrix completion.
\newblock {\em IEEE Trans. Inform. Theory 69}, 10 (2023), 6571--6585.

\end{thebibliography}

\newpage

\appendix

\setcounter{footnote}{0} 

\section{Proofs for Results in Section \ref{sec-confined-SGD-mfd}}\label{sec-proof-sec-2}
While Theorem \ref{thm-confined-SGD} is inspired by Section 5 of Bottou (1999), our proof follows largely the strategy laid out in Bertsekas and Tsitsiklis (2000). Similar arguments were used by the third author to prove a special case of Theorem \ref{thm-confined-SGD} (\cite[Theorem 1.1]{Wu:2021}). Unless otherwise specified, all notations in this appendix are given in Theorem \ref{thm-confined-SGD}.

\begin{lemma}\label{lemma-Taylor}
Let $h:M\rightarrow\R$ be any $C^2$ function. Then, for any $x\in M$ and $\mathbf{v}\in T_x M$, there exists an $s^\star \in [0,1]$ such that $h(R_x(\mathbf{v})) = h(x) + \left\langle  \nabla h(x), \mathbf{v}\right\rangle_x +\frac{1}{2}\Hess(h\circ R_x)|_{s^\star \mathbf{v}} (\mathbf{v},\mathbf{v})$.
\end{lemma}

\begin{proof}
Define $p(s)=h(R_x(s\mathbf{v}))$. Then, by Taylor's Theorem,  there exists an $s^\star \in [0,1]$ such that $p(1) = p(0) + p'(0) + \frac{1}{2}p''(s^\star)$. Note that:
\begin{itemize}
	\item $p(0) = h(R_x(\mathbf{0}_x)) =h(x)$ and $p(1) = h(R_x(\mathbf{v}))$;
	\item $p'(s) = \left\langle \nabla (h\circ R_x)(s\mathbf{v}) ,\mathbf{v}\right\rangle_x$ and therefore $p'(0) = \left\langle \nabla (h\circ R_x)(\mathbf{0}_x) ,\mathbf{v}\right\rangle_x=\left\langle \nabla h(x) ,\mathbf{v}\right\rangle_x$, where the last equation follows from Equation \eqref{eq-gradient-coincide};
	\item $p''(s) = \frac{d}{ds}\left\langle \nabla (h\circ R_x)(s\mathbf{v}) ,\mathbf{v}\right\rangle_x = \left\langle \frac{d}{ds}\nabla (h\circ R_x)(s\mathbf{v}) ,\mathbf{v}\right\rangle_x = \Hess(h\circ R_x)|_{s \mathbf{v}} (\mathbf{v},\mathbf{v})$.
\end{itemize}
Combining the above, we get that $h(R_x(\mathbf{v})) = h(x) + \left\langle  \nabla h(x), \mathbf{v}\right\rangle_x +\frac{1}{2}\Hess(h\circ R_x)|_{s^\star \mathbf{v}} (\mathbf{v},\mathbf{v})$.
\end{proof}

\begin{lemma}\label{lemma-x-t-confined}
$\rho(x_t) < \rho_1$ for all $t\geq 0$. In particular, the sequence $\{x_t\}_{t=0}^\infty$ is contained in the compact subset $\{x\in M~\big{|}~ \rho(x) \leq \rho_1\}$ of $M$.
\end{lemma}

\begin{proof}
We prove inductively that 
\begin{equation}\label{eq-x-t-confined-induction}
\rho(x_t) + \frac{b^2}{2}\sum_{j=t}^\infty c_j^2 \leq \rho_1,
\end{equation} 
which implies the lemma.

For $t=0$, $\rho(x_0)\leq \rho_0$ by our choice. So $\rho(x_0) + \frac{b^2}{2}\sum_{j=0}^\infty c_j^2 \leq \rho_0 +\frac{b^2 \sigma}{2} < \rho_1$. Assume that Inequality \eqref{eq-x-t-confined-induction} is true for some $t\geq 0$. Now we prove it for $t+1$. By Lemma \ref{lemma-Taylor}, there is an $s^\star \in [0,1]$ such that
\begin{eqnarray*}
	&& \rho(x_{t+1}) = \rho (R_{x_t}(-\frac{c_t}{\vf_t}\nabla_M f(x_t,\omega_t))) \\ 
	& = & \rho (x_t) - \frac{c_t}{\vf_t}\left\langle  \nabla \rho(x_t), \nabla_M f(x_t,\omega_t)\right\rangle_x  + \frac{c_t^2}{2\vf_t^2}\Hess(\rho\circ R_x)|_{-s^\star \frac{c_t}{\vf_t}\nabla_M f(x_t,\omega_t)} (\nabla_M f(x_t,\omega_t),\nabla_M f(x_t,\omega_t)). \\
\end{eqnarray*}
Recall that $c= \max \{c_t~\big{|}~t\geq 0\}$ and $\sigma=\sum_{t=0}^\infty c_t^2$. Note that $0\leq s^\star \frac{c_t}{\vf_t}\leq \Theta$ by Inequality \eqref{eq-def-vf-t}. Let us consider the following two cases.
\begin{itemize}
	\item[Case 1:] $\rho(x_t)\leq \rho_0$. Then, by Inequality \eqref{eq-def-vf-t}, we get 
	\[
	\rho(x_{t+1})  \leq   \rho (x_t) + \frac{c_t}{\vf_t} a A_t + \frac{c_t^2}{2\vf_t^2} b^2B_t^2 \leq \rho (x_t) + ca + \frac{b^2c_t^2}{2} \leq \rho_0 + ca + \frac{b^2c_t^2}{2}.
	\]
	Thus, 
	\[
	\rho(x_{t+1}) + \frac{b^2}{2}\sum_{j=t+1}^\infty c_j^2 \leq \rho_0 + ca + \frac{b^2}{2}\sum_{j=t}^\infty c_j^2<\rho_1.
	\]
	\item[Case 2:] $\rho_0<\rho(x_t) < \rho_1$. Then $\left\langle  \nabla \rho(x_t), \nabla_M f(x_t,\omega_t)\right\rangle_x\geq 0$ by our choice of $\rho_0$. So, by Inequality \eqref{eq-def-vf-t},  
	\begin{eqnarray*}
	\rho(x_{t+1}) & \leq & \rho (x_t)+ \frac{c_t^2}{2\vf_t^2}\Hess(\rho\circ R_x)|_{-s^\star \frac{c_t}{\vf_t}\nabla_M f(x_t,\omega_t)} (\nabla_M f(x_t,\omega_t),\nabla_M f(x_t,\omega_t)) \\
	& \leq &  \rho (x_t)+ \frac{c_t^2}{2\vf_t^2} b^2B_t^2 \leq \rho (x_t)+ \frac{b^2c_t^2}{2}.
	\end{eqnarray*}
	Thus, 
	\[
	\rho(x_{t+1}) + \frac{b^2}{2}\sum_{j=t+1}^\infty c_j^2 \leq \rho(x_t) + \frac{b^2}{2}\sum_{j=t}^\infty c_j^2 \leq \rho_1.
	\]
\end{itemize}
This proves Inequality \eqref{eq-x-t-confined-induction} for $t+1$ and completes the proof of the lemma.
\end{proof}

\begin{remark}
Note that we did not use Assumptions (a) and (b) from Theorem \ref{thm-confined-SGD} in the proof of Lemma \ref{lemma-x-t-confined}.
\end{remark}

\begin{lemma}\label{lemma-F-locally-Lipschitz}
$F$ is first-order differentiable and has locally $R$-Lipschitz gradient. Moreover, there exists a $C_1>0$ such that $\|\nabla (F\circ R_x)(\mathbf{v}) - \nabla F(x)\|_x \leq C_1\|\mathbf{v}\|_x$ and $F(R_x(\mathbf{v})) \leq F(x) + \left\langle \nabla F(x), \mathbf{v} \right\rangle_x + \frac{C_1}{2}\|\mathbf{v}\|_x^2$ for every $x\in M$ satisfying $\rho(x)\leq \rho_1$ and every $\mathbf{v}\in T_x M$ satisfying $\|\mathbf{v}\|_x\leq G_1$, where
\begin{equation}\label{eq-def-G-1}
G_1 := \sup \{\|\nabla_M f(x,\omega)\|_x~\big{|}~ x\in M \text{ satisfying }\rho(x)\leq \rho_1,~\omega\in \Omega\}.
\end{equation}
\end{lemma}

\begin{proof}
Since $\|\nabla_M f(x,\omega)\|$ is locally bounded, it follows from the Donimated Convergence Theorem that $F$ is first-order differentiable and $\nabla F(x) = \int_\Omega \nabla_M f(x,\omega)d\mu$. Since $R$ is $C^2$, $\|\nabla_{T_x M} f(R_x(\mathbf{v}),\omega)\|$ is also locally bounded. So, again by the Donimated Convergence Theorem, we have
\[
\nabla (F\circ R_x)(\mathbf{v}) = \int_\Omega \nabla_{T_x M} f(R_x(\mathbf{v}),\omega)d\mu = \int_\Omega \nabla f_{x,\omega}(\mathbf{v})d\mu,
\]
where $f_{x,\omega}:=f(R_x(\ast),\omega):T_x M \rightarrow \R$ is given in Definition \ref{def-R-lipschitz}. 

Since $f$ has locally $R$-Lipschitz gradient, we know that, for every compact subset $K$ of $M$ and every $r>0$, there is a constant $C_{K,r}>0$ such that $\|\nabla f_{x,\omega}(\mathbf{v}) - \nabla f_{x,\omega}(\mathbf{0}_x)\|_x \leq C_{K,r}\|\mathbf{v}\|_x$ for every $x \in K$, every $\mathbf{v} \in T_x M$ satisfying $\|\mathbf{v}\|_x\leq r$ and every $\omega \in \Omega$. Thus,
\begin{eqnarray*}
&& \|\nabla (F\circ R_x)(\mathbf{v}) - \nabla (F\circ R_x)(\mathbf{0}_x)\|_x = \|\int_\Omega \nabla f_{x,\omega}(\mathbf{v})d\mu - \int_\Omega \nabla f_{x,\omega}(\mathbf{0}_x)d\mu\|_x \\
& \leq & \int_\Omega\| \nabla f_{x,\omega}(\mathbf{v})-\nabla f_{x,\omega}(\mathbf{0}_x)\|_xd\mu \leq \int_\Omega C_{K,r}\|\mathbf{v}\|_x d\mu = C_{K,r}\|\mathbf{v}\|_x.
\end{eqnarray*}
This shows that $F$ has locally $R$-Lipschitz gradient.

Now consider the special case $K=\{x\in M~\big{|}~ \rho(x)\leq \rho_1\}$ and $r=G_1$. Let $C_1= C_{K,G_1}$. Then, by Equation \eqref{eq-gradient-coincide} and the above inequality, $\|\nabla (F\circ R_x)(\mathbf{v}) - \nabla F(x)\|_x \leq C_1\|\mathbf{v}\|_x$ for every $x\in M$ satisfying $\rho(x)\leq \rho_1$ and every $\mathbf{v}\in T_x M$ satisfying $\|\mathbf{v}\|_x\leq G_1$. Let $p(s)=F(R_x(s\mathbf{v}))$ for $s\in [0,1]$. Then $p'(s)=\left\langle  \nabla (F\circ R_x)(s\mathbf{v}), \mathbf{v}\right\rangle_x$. Thus,
\begin{eqnarray*}
F(R_x(\mathbf{v})) - F(x) & = & p(1) -p(0) = \int_0^1 p'(s) ds = \int_0^1 \left\langle  \nabla (F\circ R_x)(s\mathbf{v}), \mathbf{v}\right\rangle_x ds \\
& = & \int_0^1 \left\langle  \nabla F(x)+\nabla (F\circ R_x)(s\mathbf{v})-\nabla F(x), \mathbf{v}\right\rangle_x ds \\
& = &  \left\langle  \nabla F(x), \mathbf{v}\right\rangle_x + \int_0^1 \left\langle  \nabla (F\circ R_x)(s\mathbf{v})-\nabla F(x), \mathbf{v}\right\rangle_x ds \\
& \leq & \left\langle  \nabla F(x), \mathbf{v}\right\rangle_x + \int_0^1 \|\nabla (F\circ R_x)(s\mathbf{v})-\nabla F(x)\|_x \|\mathbf{v}\|_x ds \\
& \leq & \left\langle  \nabla F(x), \mathbf{v}\right\rangle_x + \int_0^1 C_1s \|\mathbf{v}\|_x^2 ds  = \left\langle  \nabla F(x), \mathbf{v}\right\rangle_x + \frac{C_1}{2}\|\mathbf{v}\|_x^2
\end{eqnarray*}
for every $x\in M$ satisfying $\rho(x)\leq \rho_1$ and every $\mathbf{v}\in T_x M$ satisfying $\|\mathbf{v}\|_x\leq G_1$. This completes the proof.
\end{proof}

\begin{lemma}\label{lemma-expectation-converges}
$\sum_{t=0}^\infty c_t E(\|\nabla F(x_t)\|_{x_t}^2)$ converges and, consequently, $\sum_{t=0}^\infty c_t \|\nabla F(x_t)\|_{x_t}^2$ and $\sum_{t=0}^\infty \frac{c_t}{\vf_t} \|\nabla F(x_t)\|_{x_t}^2$ both converge almost surely.
\end{lemma}

\begin{proof}
By Lemma \ref{lemma-x-t-confined}, $\rho(x_t) \leq \rho_1$. So $\|\nabla_M f(x_t,\omega_t)\|_{x_t} \leq G_1$, where $G_1$ is defined in Equation \eqref{eq-def-G-1}. By Lemma \ref{lemma-F-locally-Lipschitz}, we have
\[
F(x_{t+1}) = F(R_{x_t}(-\frac{c_t}{\vf_t}\nabla_M f(x_t,\omega_t))) \leq F(x_t) - \frac{c_t}{\vf_t}\left\langle \nabla F(x_t), \nabla_M f(x_t,\omega_t) \right\rangle_{x_t} + \frac{C_1 c_t^2}{2\vf_t^2} \|\nabla_M f(x_t,\omega_t)\|_{x_t}^2.
\]
So 
\begin{equation}\label{eq-lemma-expectation-converges-1}
F(x_{t+1}) \leq  F(x_t) - \frac{c_t}{\vf_t}\left\langle \nabla F(x_t), \nabla_M f(x_t,\omega_t) \right\rangle_{x_t} + \frac{C_1 G_1^2}{2\Phi_{\min}^2} c_t^2 .
\end{equation}
Taking expectations, we get
\begin{eqnarray*}
E(F(x_{t+1})) & \leq & E\left(F(x_t) - \frac{c_t}{\vf_t}\left\langle \nabla F(x_t), \nabla_M f(x_t,\omega_t) \right\rangle_{x_t} + \frac{C_1 G_1^2}{2\Phi_{\min}^2} c_t^2\right) \\
& = & E(F(x_t)) - E(\frac{c_t}{\vf_t}\left\langle \nabla F(x_t), \nabla_M f(x_t,\omega_t) \right\rangle_{x_t})+ \frac{C_1 G_1^2}{2\Phi_{\min}^2} c_t^2.
\end{eqnarray*}
Note that 
\[
E(\frac{c_t}{\vf_t}\left\langle \nabla F(x_t), \nabla_M f(x_t,\omega_t) \right\rangle_{x_t}) = E\left(\frac{c_t}{\vf_t}E(\left\langle \nabla F(x_t), \nabla_M f(x_t,\omega_t) \right\rangle_{x_t}~\big{|}~x_t)\right) \geq \frac{c_t}{\Phi_{\max}}E(\|\nabla F(x_t)\|_{x_t}^2 ).
\]
Thus,
\[
E(F(x_{t+1}))  \leq E(F(x_t)) - \frac{c_t}{\Phi_{\max}}E(\|\nabla F(x_t)\|_{x_t}^2 ) + \frac{C_1 G_1^2}{2\Phi_{\min}^2} c_t^2.
\]
Summing the above inequality from $0$ to $t$, we get that
\[
E(F(x_{t+1})) \leq E(F(x_0)) -\sum_{\tau=0}^t\frac{c_\tau}{\Phi_{\max}}E(\|\nabla F(x_\tau)\|_{x_\tau}^2 ) + \frac{C_1 G_1^2}{2\Phi_{\min}^2} \sum_{\tau=0}^t c_\tau^2.
\]
Therefore,
\[
\frac{1}{\Phi_{\max}}\sum_{\tau=0}^t c_\tau E(\|\nabla F(x_\tau)\|_{x_\tau}^2 ) \leq  E(F(x_0)) - E(F(x_{t+1})) + \frac{C_1 G_1^2}{2\Phi_{\min}^2} \sum_{\tau=0}^t c_\tau^2 \leq E(F(x_0)) - E(F(x_{t+1})) + \frac{C_1 G_1^2}{2\Phi_{\min}^2} \sigma.
\]
By Lemma \ref{lemma-x-t-confined}, $\{x_t\}$ is contained in a compact subset of $M$. So the right hand side of the above inequality is bounded above. This implies that $\sum_{t=0}^\infty c_t E(\|\nabla F(x_t)\|_{x_t}^2)$ converges. It then follows that $\sum_{t=0}^\infty c_t \|\nabla F(x_t)\|_{x_t}^2$ is finite with probability $1$, that is, $\sum_{t=0}^\infty c_t \|\nabla F(x_t)\|_{x_t}^2$ converges almost surely. But ${0 \leq \frac{c_t}{\vf_t} \|\nabla F(x_t)\|_{x_t}^2 \leq \frac{c_t}{\Phi_{\min}} \|\nabla F(x_t)\|_{x_t}^2}$. By the $M$-Test, we know that $\sum_{t=0}^\infty \frac{c_t}{\vf_t} \|\nabla F(x_t)\|_{x_t}^2$ converges when $\sum_{t=0}^\infty c_t \|\nabla F(x_t)\|_{x_t}^2$ converges. Thus, $\sum_{t=0}^\infty \frac{c_t}{\vf_t} \|\nabla F(x_t)\|_{x_t}^2$ also converges almost surely.
\end{proof}

\begin{lemma}\label{lemma-F-x-t-converge}
Both $\sum_{t=0}^\infty \frac{c_t}{\vf_t}\left\langle \nabla F(x_t), \nabla_M f(x_t,\omega_t)\right\rangle_{x_t}$ and $\lim_{t\rightarrow \infty} F(x_t)$ converge almost surely.
\end{lemma}

\begin{proof}
Define 
\begin{equation}\label{eq-lemma-F-x-t-converge-1}
u_t = \left\langle \nabla F(x_t), \nabla_M f(x_t,\omega_t) - \nabla F(x_t)\right\rangle_{x_t}
\end{equation} 
and $z_t = \sum_{\tau=0}^t \frac{c_\tau}{\vf_\tau} u_\tau$. By Lemmas \ref{lemma-x-t-confined} and \ref{lemma-F-locally-Lipschitz}, the sequence $\{u_t\}_{t=0}^\infty$ is bounded. Choose a $U>0$ such that $|u_t|\leq U$ for all $t\geq 0$. Note that $x_t$ and $\vf_t$ depend only on $\omega_0,\dots,\omega_{t-1}$ and are independent of $\omega_t$. So 
\[
E(\frac{c_t}{\vf_t}u_t~\big{|}~\omega_0,\dots,\omega_{t-1}) = \frac{c_t}{\vf_t}E(u_t~\big{|}~\omega_0,\dots,\omega_{t-1})= \frac{c_t}{\vf_t}\left\langle \nabla F(x_t), \nabla F(x_t) -\nabla F(x_t)\right\rangle_{x_t} = 0
\] 
and therefore $E(z_t~\big{|}~\omega_0,\dots,\omega_{t-1}) = z_{t-1}$. This shows that $\{z_t\}_{t=0}^\infty$ is a martingale relative to $\{\omega_t\}_{t=0}^\infty$. For any $t\geq 0$, we have that
\[
E(z_t) = E(z_{t-1}) + E(\frac{c_t}{\vf_t} u_t) = E(z_{t-1}) +  E(E(\frac{c_t}{\vf_t} u_t~\big{|}~\omega_0,\dots,\omega_{t-1})) = E(z_{t-1}) = \cdots =E(z_0)=0.
\]
Note that $z_{t-1}$ is also determined by $\omega_0,\dots,\omega_{t-1}$, So the variance of $z_t$ satisfies
\begin{eqnarray*}
&& Var(z_t)  =  E(z_t^2) =E((\frac{c_t}{\vf_t} u_t+z_{t-1})^2) = E(\frac{c_t^2}{\vf_t^2} u_t^2 + z_{t-1}^2 + 2\frac{c_t}{\vf_t} u_t z_{t-1}) \\
& = &   Var(z_{t-1}) + E(\frac{c_t^2}{\vf_t^2}u_t^2) + 2 E(\frac{c_t}{\vf_t}u_t z_{t-1}) \leq  Var(z_{t-1}) + \frac{c_t^2 U^2 }{\Phi_{\min}^2}+ 2E(\frac{c_t}{\vf_t}u_t z_{t-1}) \\
& = & Var(z_{t-1}) + \frac{c_t^2 U^2 }{\Phi_{\min}^2}+ 2E(z_{t-1}E(\frac{c_t}{\vf_t}u_t ~\big{|}~\omega_0,\dots,\omega_{t-1})) = Var(z_{t-1}) + \frac{c_t^2 U^2 }{\Phi_{\min}^2}.
\end{eqnarray*}
Summing the above inequality from $1$ to $t$, we get that 
\[
Var(z_t) \leq Var(z_0) + \frac{U^2 }{\Phi_{\min}^2} \sum_{\tau=1}^t c_\tau^2\leq Var(z_0) + \frac{U^2 }{\Phi_{\min}^2} \sigma.
\] 
This shows that $\{Var(z_t)\}_{t=0}^\infty$ is bounded. Thus, by the Martingale Convergence Theorem, ${\lim_{t\rightarrow \infty}z_t =  \sum_{t=0}^\infty \frac{c_t}{\vf_t} u_t}$ converges almost surely. By Lemma \ref{lemma-expectation-converges}, $\sum_{t=0}^\infty \frac{c_t}{\vf_t}  \|\nabla F(x_t)\|_{x_t}^2$ also converges almost surely. Thus, 
\[
\sum_{t=0}^\infty \frac{c_t}{\vf_t} \left\langle \nabla F(x_t), \nabla_M f(x_t,\omega_t)\right\rangle_{x_t} = \sum_{t=0}^\infty \frac{c_t}{\vf_t}  u_t + \sum_{t=0}^\infty \frac{c_t}{\vf_t}  \|\nabla F(x_t)\|_{x_t}^2
\]
converges almost surely.

Next consider $\{F(x_t)\}_{t=0}^\infty$. Assume that $\sum_{t=0}^\infty \frac{c_t}{\vf_t} \left\langle \nabla F(x_t), \nabla_M f(x_t,\omega_t)\right\rangle_{x_t}$ converges. Define 
\[
v_t = F(x_t) -\sum_{\tau=t}^\infty \frac{c_\tau}{\vf_\tau} \left\langle \nabla F(x_\tau), \nabla_M f(x_\tau,\omega_\tau)\right\rangle_{x_\tau} + \frac{C_1 G_1^2}{2\Phi_{\min}^2} \sum_{\tau=t}^\infty c_\tau^2,
\]
where $G_1$ is defined in Equation \eqref{eq-def-G-1}. By Inequality \eqref{eq-lemma-expectation-converges-1}, $\{v_t\}_{t=0}^\infty$ is a decreasing sequence. By Lemmas \ref{lemma-x-t-confined} and \ref{lemma-F-locally-Lipschitz}, $\{F(x_t)\}_{t=0}^\infty$ is a bounded sequence. Since $\sum_{t=0}^\infty c_t^2$ and $\sum_{t=0}^\infty \frac{c_t}{\vf_t}\left\langle \nabla F(x_t), \nabla_M f(x_t,\omega_t)\right\rangle_{x_t}$ both converge, the sequences $\{\sum_{\tau=t}^\infty c_\tau^2\}_{t=0}^\infty$ and $\{\sum_{\tau=t}^\infty \frac{c_\tau}{\vf_\tau} \left\langle \nabla F(x_\tau), \nabla_M f(x_\tau,\omega_\tau)\right\rangle_{x_\tau}\}_{t=0}^\infty$ are also bounded. So $\{v_t\}_{t=0}^\infty$ is bounded too. Thus, $\lim_{t \rightarrow \infty} v_t$ converges. Note that 
\[
\lim_{t \rightarrow \infty} \sum_{\tau=t}^\infty c_\tau^2 = \lim_{t \rightarrow \infty} \sum_{\tau=t}^\infty \frac{c_\tau}{\vf_\tau} \left\langle \nabla F(x_\tau), \nabla_M f(x_\tau,\omega_\tau)\right\rangle_{x_\tau} =0.
\]
This shows that $\lim_{t\rightarrow \infty} F(x_t) = \lim_{t \rightarrow \infty} v_t$ converges when $\sum_{t=0}^\infty \frac{c_t}{\vf_t} \left\langle \nabla F(x_t), \nabla_M f(x_t,\omega_t)\right\rangle_{x_t}$ converges. But we have shown that the latter sum converges almost surely. Hence, $\lim_{t\rightarrow \infty} F(x_t)$ also converges almost surely.
\end{proof}

The following lemma is the only place in our proof where we need to compare vectors in different tangent spaces.

\begin{lemma}\label{lemma-gradient-difference}
There are positive constants $C_2$ and $r_2$ such that 
\[
\left|\|\nabla F (R_x(\mathbf{v}))\|_{R_x(\mathbf{v})} - \|\nabla(F\circ R_x)(\mathbf{v})\|_x\right| \leq C_2 \|\mathbf{v}\|_x
\] 
for every $x \in M$ satisfying $\rho(x)\leq \rho_1$ and every $\mathbf{v} \in T_x M$ satisfying $\|\mathbf{v}\|_x \leq r_2$.
\end{lemma}

\begin{proof}
For every $x \in M$, there is a $C^\infty$ coordinate chart $(\W_x,(y_1,\dots,y_m))$ such that $x\in \W_x$ and the closure $\overline{\W_x}$ is compact. Choose an open subset $\U_x$ of $\W_x$ such that $x \in \U_x \subset \overline{\U_x} \subset \W_x$. Note that the closure $\overline{\U_x}$ of $\U_x$ is also compact. 

For $y\in \W_x$, apply the Gram-Schmidt Process to the basis $[\frac{\partial }{\partial y_1},\dots,\frac{\partial }{\partial y_m}]$ of $T_x M$ and then normalize the length. This gives an orthonormal basis $[\mathbf{v}_1(y),\dots,\mathbf{v}_m(y)]$ of $T_y M$ and a $C^\infty$ $m\times m$ invertible upper-triangular matrix $P(y)$ such that $[\frac{\partial }{\partial y_1},\dots,\frac{\partial }{\partial y_m}]=[\mathbf{v}_1(y),\dots,\mathbf{v}_m(y)]P(y)$. For the tangent bundle $T\W_x$ of $\W_x$, we use the coordinates $(y_1,\dots,y_m,z_1,\dots,z_m)$ where, for any $y\in\W_x$ and $\mathbf{v} \in T_y\W_x$, we have $\mathbf{v} = \sum_{i=1}^m z_i \mathbf{v}_i$. Since the frame $[\mathbf{v}_1(y),\dots,\mathbf{v}_m(y)]$ is $C^\infty$ on $\W_x$, the coordinates $(y_1,\dots,y_m,z_1,\dots,z_m)$ are $C^\infty$ on $T\W_x$. Let $[dR_{i,j}]_{m \times m}=[dR_{i,j}(y_1,\dots,y_m,z_1,\dots,z_m)]_{m \times m}$ be the $m\times m$ matrix representing $dR_y|_{\mathbf{v}}$ on $T\W_x$ in the coordinates $(y_1,\dots,y_m,z_1,\dots,z_m)$. Since the retraction $R$ is $C^2$, we know that $[dR_{i,j}]$ is $C^1$. 

By the compactness of $\overline{\U_x}$, there is an $r_x>0$ such that $R_y(\mathbf{v})\in \W_x$ whenever $y \in \overline{\U_x}$ and $\mathbf{v}\in T_y M$ satisfying $\|\mathbf{v}\|_y\leq r_x$. For any $y \in \overline{\U_x}$, $\mathbf{v}\in T_y M$ satisfying $\|\mathbf{v}\|_y\leq r_x$ and $\mathbf{u}\in T_y M$, we have that 
\begin{eqnarray*}
&&\left\langle \nabla (F\circ R_y)(\mathbf{v}), \mathbf{u}\right\rangle_y = d(F\circ R_y)|_{\mathbf{v}}(\mathbf{u}) = ((dF|_{R_y(\mathbf{v})})\circ (dR_y|_{\mathbf{v}}))(\mathbf{u})  \\
& = & \left\langle (\nabla F)(R_y(\mathbf{v})), (dR_y|_{\mathbf{v}})(\mathbf{u})\right\rangle_{R_y(\mathbf{v})} = \left\langle \mathrm{adj}(dR_y|_{\mathbf{v}})((\nabla F)(R_y(\mathbf{v}))), \mathbf{u}\right\rangle_{y},
\end{eqnarray*}
where $\mathrm{adj}(dR_y|_{\mathbf{v}})$ is the adjoint of $dR_y|_{\mathbf{v}}$ with respect to the inner products $\left\langle \ast,\ast\right\rangle_y$ and $\left\langle \ast,\ast\right\rangle_{R_y(\mathbf{v})}$. This shows that 
\begin{equation}\label{eq-lemma-gradient-difference-1}
\nabla (F\circ R_y)(\mathbf{v})=\mathrm{adj}(dR_y|_{\mathbf{v}})((\nabla F)(R_y(\mathbf{v}))).
\end{equation} 
Recall that $[\mathbf{v}_1,\dots,\mathbf{v}_m]$ is orthonormal in $T\W_x$. So $\mathrm{adj}(dR_y|_{\mathbf{v}})$ is represented by the matrix $[dR_{i,j}]^T$ under the coordinates $(y_1,\dots,y_m,z_1,\dots,z_m)$.

Recall that each entry of the matrix $[dR_{i,j}]$ is a $C^1$ function on $T\W_x$. In particular, the gradient $\nabla_z dR_{i,j} = \sum_{p=1}^m \frac{\partial dR_{i,j}}{\partial z_p}\mathbf{v}_p\in T_yM$ is bounded on the compact set $\{\mathbf{v}\in T_y M~\big{|}~y\in \U_x,~\|\mathbf{v}\|_y \leq R_x\}$. Define $B_x = \max\{\|\nabla_z dR_{i,j}(y,\mathbf{v})\|_y ~\big{|}~1\leq i,j\leq m,~y\in \U_x,~\mathbf{v}\in T_y M,~\|\mathbf{v}\|_y \leq r_x\}$. Let $y\in \U_x$ be the point given by the coordinates $(y_1,\dots,y_m)$ and $\mathbf{v}=\sum_{p=1}^m z_p \mathbf{v}_p(y)$ satisfying $\|\mathbf{v}\|_y\leq r_x$. By the Mean Value Theorem, there is an $s^\star \in [0,1]$ such that 
\begin{eqnarray*}
&& dR_{i,j}(y_1,\dots,y_m,z_1,\dots,z_m) - dR_{i,j}(y_1,\dots,y_m,0,\dots,0) = \left\langle \nabla_z dR_{i,j}(y,s^\star\mathbf{v}),\mathbf{v}\right\rangle_y \\
& \leq & \|\nabla_z dR_{i,j}(y,s^\star\mathbf{v})\|_y  \|\mathbf{v}\|_y\leq B_x \|\mathbf{v}\|_y.
\end{eqnarray*}
Recall that $dR_y|_{\mathbf{0}_y} =\id_{T_y M}$. So $[dR_{i,j}(y_1,\dots,y_m,0,\dots,0)]=I_m$, that is,
\[
dR_{i,j}(y_1,\dots,y_m,0,\dots,0) =\delta_{i,j}=
\begin{cases}
1 &\text{if }i=j, \\ 
0 &\text{if }i\neq j.
\end{cases}
\]
Combining the above, we get that
\begin{equation}\label{eq-lemma-gradient-difference-2}
dR_{i,j}(y_1,\dots,y_m,z_1,\dots,z_m) - \delta_{i,j} \leq B_x \|\mathbf{v}\|_y.
\end{equation}

Since $K:=\{x\in M~\big{|}~ \rho(x)\leq \rho_1\}$ is compact, there is a finite subset $\{x_1,\dots,x_n\}\subset K$ such that $K\subset \bigcup_{q=1}^n \U_{x_q}$. Define $r_2 = \min\{ r_{x_1},\dots,r_{x_n}\}$, $B = \max\{ B_{x_1},\dots,B_{x_n}\}$ and ${G_2 = \sup\{ \|\nabla F(x)\|_x~\big{|}~ x \in \bigcup_{q=1}^n \W_{x_q}\}}$. Given $x\in K$, pick an $x_q$ such that $x\in \U_{x_q}$. Denote 
\begin{itemize}
	\item by $(y_1,\dots,y_m,z_1,\dots,z_m)$ the coordinate on $T\W_{x_p}$ given in the first half of the proof,
	\item by $[\mathbf{v}_1(y),\dots,\mathbf{v}_m(y)]$ the orthonormal frame of $T\W_{x_p}$ given in the first half of the proof,
	\item for $y\in \W_{x_q}$, by $\mathbf{w}(y)=[w_1(y),\dots,w_m(y)]^T \in \R^m$ the vector satisfying that $\nabla F(y) = \sum_{p=1}^m w_p(y)\mathbf{v}_p(y)$.
\end{itemize}
Furthermore, denote by $\|\ast\|$ the standard norm on $R^m$ and by $\|\ast\|_F$ the Frobenius norm on $\R^{m \times m}$. Then, for $x\in \U_{x_q}$ and $\mathbf{v} \in T_x M$ satisfying $\|\mathbf{v}\|_x \leq r_2$, we have
\begin{eqnarray*}
&& \left|\|\nabla F (R_x(\mathbf{v}))\|_{R_x(\mathbf{v})} - \|\nabla(F\circ R_x)(\mathbf{v})\|_x\right| = \left|\|\nabla F (R_x(\mathbf{v}))\|_{R_x(\mathbf{v})} - \|\mathrm{adj}(dR_x|_\mathbf{v})((\nabla F) (R_x(\mathbf{v})))\|_x\right| \\
& = & \left| \|[\mathbf{v}_1(R_x(\mathbf{v})),\dots, \mathbf{v}_m(R_x(\mathbf{v}))]\mathbf{w}(R_x(\mathbf{v}))\|_{R_x(\mathbf{v})} - \|[\mathbf{v}_1(x),\dots, \mathbf{v}_m(x)][dR_{i,j}(x,\mathbf{v})]^T\mathbf{w}(R_x(\mathbf{v}))\|_{x} \right| \\
& = & \left| \|\mathbf{w}(R_x(\mathbf{v}))\| - \|[dR_{i,j}(x,\mathbf{v})]^T\mathbf{w}(R_x(\mathbf{v}))\| \right| \leq \|[dR_{i,j}(x,\mathbf{v})]^T\mathbf{w}(R_x(\mathbf{v})) - \mathbf{w}(R_x(\mathbf{v}))\| \\
& = & \|([dR_{i,j}(x,\mathbf{v})]^T-I_m)\mathbf{w}(R_x(\mathbf{v})) \| = \|([dR_{i,j}(x,\mathbf{v})-\delta_{i,j}]^T)\mathbf{w}(R_x(\mathbf{v})) \| \\
& = & \sqrt{\sum_{j=1}^m (\sum_{i=1}^m  ( dR_{i,j}(x,\mathbf{v})-\delta_{i,j}) w_i(R_x(\mathbf{v})))^2} \leq \sqrt{m\sum_{j=1}^m \sum_{i=1}^m  ( dR_{i,j}(x,\mathbf{v})-\delta_{i,j})^2 (w_i(R_x(\mathbf{v})))^2} \\
& \leq & \sqrt{m\sum_{j=1}^m \sum_{i=1}^m  B^2\|\mathbf{v}\|_x^2 (w_i(R_x(\mathbf{v})))^2} = \sqrt{m^2 B^2\|\mathbf{v}\|_x^2 \sum_{i=1}^m (w_i(R_x(\mathbf{v})))^2} \\
& = & mB\|\mathbf{v}\|_x\|\mathbf{w}(R_x(\mathbf{v}))\| = m B \|\mathbf{v}\|_x \|\nabla F (R_x(\mathbf{v}))\|_{R_x(\mathbf{v})} \leq m B \|\mathbf{v}\|_x G_2.
\end{eqnarray*}
Thus, if we choose $C_2=mBG_2$, then $\left|\|\nabla F (R_x(\mathbf{v}))\|_{R_x(\mathbf{v})} - \|\nabla(F\circ R_x)(\mathbf{v})\|_x\right| \leq C_2 \|\mathbf{v}\|_x$
for every $x \in K$ and every $\mathbf{v} \in T_x M$ satisfying $\|\mathbf{v}\|_x \leq r_2$. This completes the proof.
\end{proof}

\begin{lemma}\label{lemma-nabla-F-x-t-converge}
$\lim_{t \rightarrow \infty} \|\nabla F(x_t)\|_{x_t} =0$ almost surely.
\end{lemma}

\begin{proof}
By Lemmas \ref{lemma-expectation-converges} and \ref{lemma-F-x-t-converge}, $\sum_{t=0}^\infty \frac{c_t}{\vf_t} \|\nabla F(x_t)\|_{x_t}^2$, $\sum_{t=0}^\infty \frac{c_t}{\vf_t}\left\langle \nabla F(x_t), \nabla_M f(x_t,\omega_t)\right\rangle_{x_t}$ and $\lim_{t\rightarrow \infty} F(x_t)$ all converge almost surely. So, to prove the current lemma, we only need to show that $\lim_{t \rightarrow \infty} \|\nabla F(x_t)\|_{x_t} =0$ if all these three are convergent, which we will assume throughout this proof. 

First, we claim that $\liminf_{t \rightarrow \infty} \|\nabla F(x_t)\|_{x_t} =0$. Otherwise, there would be an $\epsilon>0$ and a $\tau>0$ such that $\|\nabla F(x_t)\|_{x_t}>\epsilon$ if $t\geq \tau$. Then $\sum_{t=0}^\infty \frac{c_t}{\vf_t} \|\nabla F(x_t)\|_{x_t}^2 \geq \sum_{t=\tau}^\infty \frac{c_t}{\vf_t} \|\nabla F(x_t)\|_{x_t}^2 \geq \frac{\epsilon^2}{\Phi_{\max}}\sum_{t=\tau}^\infty c_t =\infty$, which contradicts our assumption that $\sum_{t=0}^\infty \frac{c_t}{\vf_t} \|\nabla F(x_t)\|_{x_t}^2$ converges. This shows that $\liminf_{t \rightarrow \infty} \|\nabla F(x_t)\|_{x_t} =0$.

It remains to show that, under our assumptions, $\limsup_{t \rightarrow \infty} \|\nabla F(x_t)\|_{x_t} =0$, too. Let us prove this by contradiction again. Assume $\limsup_{t \rightarrow \infty} \|\nabla F(x_t)\|_{x_t} =s>0$. Let $C_1$, $C_2$, $G_1$ and $r_2$ be the positive constants given in Lemmas \ref{lemma-F-locally-Lipschitz} and \ref{lemma-gradient-difference}. Since $\lim_{t\rightarrow \infty} c_t=0$, there is a $T>0$ such that $\frac{c_t (C_1+C_2)G_1}{\Phi_{\min}}\leq \frac{s}{8}$ and  $\frac{c_t  G_1}{\Phi_{\min}}\leq r_2$ if $t>T$. Since $\liminf_{t \rightarrow \infty} \|\nabla F(x_t)\|_{x_t} =0$ and $\limsup_{t \rightarrow \infty} \|\nabla F(x_t)\|_{x_t} =s>0$, there are two infinite sequences $\{p_i\}$ and $\{q_i\}$ of positive integers such that
\begin{itemize}
	\item $T<p_1<q_1<p_2 <q_2<\cdots < p_i<q_i<p_{i+1}<\cdots$,
	\item $\|\nabla F (x_{p_i})\|_{x_{p_i}} < \frac{s}{4}$, $\|\nabla F (x_{q_i})\|_{x_{q_i}} > \frac{s}{2}$ and $\frac{s}{4}\leq \|\nabla F (x_t)\|_{x_{t}} \leq \frac{s}{2}$ if $p_i<t<q_i$ for $i=1,2\dots$. 
\end{itemize}
Then, by Lemmas \ref{lemma-F-locally-Lipschitz} and \ref{lemma-gradient-difference}, we have that, for $i=1,2\dots$,
\begin{eqnarray*}
&& \frac{s}{4} < \|\nabla F (x_{q_i})\|_{x_{q_i}} - \|\nabla F (x_{p_i})\|_{x_{p_i}} \leq \sum_{t=p_i}^{q_i-1} \left|\|\nabla F (x_{t+1})\|_{x_{t+1}} - \|\nabla F (x_{t})\|_{x_{t}}\right| \\
& = & \sum_{t=p_i}^{q_i-1} \left|\|\nabla F (R_{x_t}(-\frac{c_t}{\vf_t}\nabla_M f(x_t,\omega_t)))\|_{R_{x_t}(-\frac{c_t}{\vf_t}\nabla_M f(x_t,\omega_t))} - \|\nabla F (x_{t})\|_{x_{t}}\right| \\
& \leq &  \sum_{t=p_i}^{q_i-1} \left|\|\nabla F (R_{x_t}(-\frac{c_t}{\vf_t}\nabla_M f(x_t,\omega_t)))\|_{R_{x_t}(-\frac{c_t}{\vf_t}\nabla_M f(x_t,\omega_t))} - \|\nabla (F\circ R_{x_t}) (-\frac{c_t}{\vf_t}\nabla_M f(x_t,\omega_t))\|_{x_{t}}\right| \\
&&  + \sum_{t=p_i}^{q_i-1} \left|\|\nabla (F\circ R_{x_t}) (-\frac{c_t}{\vf_t}\nabla_M f(x_t,\omega_t))\|_{x_{t}} - \| \nabla F (x_{t})\|_{x_{t}}\right| \\
& \leq & \sum_{t=p_i}^{q_i-1} C_2\|-\frac{c_t}{\vf_t}\nabla_M f(x_t,\omega_t)\|_{x_t}+ \sum_{t=p_i}^{q_i-1} C_1\|-\frac{c_t}{\vf_t}\nabla_M f(x_t,\omega_t)\|_{x_t}= (C_1+C_2)\sum_{t=p_i}^{q_i-1} \frac{c_t}{\vf_t}\|\nabla_M f(x_t,\omega_t)\|_{x_t} \\
& \leq & \frac{(C_1+C_2)G_1}{\Phi_{\min}}\sum_{t=p_i}^{q_i-1} c_t.
\end{eqnarray*}
Thus, 
\begin{equation}\label{eq-lemma-nabla-F-x-t-converge-1}
\sum_{t=p_i}^{q_i-1} c_t > \frac{s\Phi_{\min}}{4(C_1+C_2)G_1}>0 \text{ for } i=1,2\dots.
\end{equation}

Using Lemmas \ref{lemma-F-locally-Lipschitz} and \ref{lemma-gradient-difference} again, we get
\begin{eqnarray*}
&& \frac{s}{4}-\|\nabla F (x_{p_i})\|_{x_{p_i}} \leq \|\nabla F (x_{{p_i}+1})\|_{x_{{p_i}+1}} -\|\nabla F (x_{p_i})\|_{x_{p_i}} \\
& = & \|\nabla F (R_{x_{p_i}}(-\frac{c_{p_i}}{\vf_{p_i}}\nabla_M f(x_{p_i},\omega_{p_i})))\|_{R_{x_{p_i}}(-\frac{c_{p_i}}{\vf_{p_i}}\nabla_M f(x_{p_i},\omega_{p_i}))} -\|\nabla F (x_{p_i})\|_{x_{p_i}} \\
& \leq & \left| \|\nabla F (R_{x_{p_i}}(-\frac{c_{p_i}}{\vf_{p_i}}\nabla_M f(x_{p_i},\omega_{p_i})))\|_{R_{x_{p_i}}(-\frac{c_{p_i}}{\vf_{p_i}}\nabla_M f(x_{p_i},\omega_{p_i}))}-\|\nabla (F \circ R_{x_{p_i}})(-\frac{c_{p_i}}{\vf_{p_i}}\nabla_M f(x_{p_i},\omega_{p_i}))\|_{x_{p_i}}\right| \\
&& + \|\nabla (F \circ R_{x_{p_i}})(-\frac{c_{p_i}}{\vf_{p_i}}\nabla_M f(x_{p_i},\omega_{p_i}))-\nabla F (x_{p_i})\|_{x_{p_i}} \\
& \leq & C_2 \|-\frac{c_{p_i}}{\vf_{p_i}}\nabla_M f(x_{p_i},\omega_{p_i})\|_{x_{p_i}} + C_1 \|-\frac{c_{p_i}}{\vf_{p_i}}\nabla_M f(x_{p_i},\omega_{p_i})\|_{x_{p_i}} \\
& = & \frac{c_{p_i} (C_1+C_2)}{\Phi_{\min}}\|\nabla_M f(x_{p_i},\omega_{p_i})\|_{x_{p_i}} \leq \frac{c_{p_i} (C_1+C_2)G_1}{\Phi_{\min}} \leq \frac{s}{8}.
\end{eqnarray*}
This shows that 
\begin{equation}\label{eq-lemma-nabla-F-x-t-converge-2}
\|\nabla F (x_{p_i})\|_{x_{p_i}} \geq \frac{s}{8}.
\end{equation}

By Inequalities \eqref{eq-lemma-expectation-converges-1}, \eqref{eq-lemma-nabla-F-x-t-converge-2} and the definitions of $p_i$, $q_i$, we have 
\begin{eqnarray*}
F(x_{q_i}) & \leq & F(x_{p_i}) -\sum_{t=p_i}^{q_i-1}\frac{c_t}{\vf_t}\left\langle \nabla F(x_t), \nabla_M f(x_t,\omega_t) \right\rangle_{x_t} + \frac{C_1 G_1^2}{2\Phi_{\min}^2} \sum_{t=p_i}^{q_i-1}c_t^2 \\
& = & F(x_{p_i}) - \sum_{t=p_i}^{q_i-1} \frac{c_t}{\vf_t} \left\langle \nabla F(x_t), \nabla_M f(x_t,\omega_t)-\nabla F(x_t)\right\rangle_{x_t} - \sum_{t=p_i}^{q_i-1}\frac{c_t}{\vf_t}\|\nabla F(x_t)\|_{x_t}^2 + \frac{C_1 G_1^2}{2\Phi_{\min}^2} \sum_{t=p_i}^{q_i-1}c_t^2 \\
& \leq & F(x_{p_i}) - \sum_{t=p_i}^{q_i-1}  \frac{c_t}{\vf_t}\left\langle \nabla F(x_t), \nabla_M f(x_t,\omega_t)-\nabla F(x_t)\right\rangle_{x_t} -\frac{s^2}{64\Phi_{\max}}\sum_{t=p_i}^{q_i-1}c_t + \frac{C_1 G_1^2}{2\Phi_{\min}^2} \sum_{t=p_i}^{q_i-1}c_t^2 \\
& = & F(x_{p_i}) -\sum_{t=p_i}^{q_i-1}  \frac{c_t}{\vf_t} u_t -\frac{s^2}{64\Phi_{\max}}\sum_{t=p_i}^{q_i-1}c_t + \frac{C_1 G_1^2}{2\Phi_{\min}^2} \sum_{t=p_i}^{q_i-1}c_t^2,
\end{eqnarray*}
where $G_1$ is defined in Equation \eqref{eq-def-G-1} and $u_t$ is defined in \eqref{eq-lemma-F-x-t-converge-1}.
Combining this with Inequality \eqref{eq-lemma-nabla-F-x-t-converge-1}, we get that 
\begin{equation}\label{eq-lemma-nabla-F-x-t-converge-3}
\frac{s\Phi_{\min}}{4(C_1+C_2)G_1} <\sum_{t=p_i}^{q_i-1} c_t \leq \frac{64\Phi_{\max}}{s^2}\left(F(x_{p_i})-F(x_{q_i}) - \sum_{t=p_i}^{q_i-1} \frac{c_t}{\vf_t} u_t + \frac{C_1 G_1^2}{2\Phi_{\min}^2} \sum_{t=p_i}^{q_i-1}c_t^2\right).
\end{equation}
But $\sum_{t=0}^{\infty}c_t^2$ converges. So, when $\sum_{t=0}^\infty \frac{c_t}{\vf_t} \|\nabla F(x_t)\|_{x_t}^2$, $\sum_{t=0}^\infty \frac{c_t}{\vf_t} \left\langle \nabla F(x_t), \nabla_M f(x_t,\omega_t)\right\rangle_{x_t}$ and $\lim_{t\rightarrow \infty} F(x_t)$ all converge,  the right hand side of Inequality \eqref{eq-lemma-nabla-F-x-t-converge-3} converges to $0$ as $i\rightarrow \infty$. Thus, under our convergence assumptions, we get $0< \frac{s\Phi_{\min}}{4(C_1+C_2)G_1}\leq 0$ by taking the limit of Inequality \eqref{eq-lemma-nabla-F-x-t-converge-3} as $i\rightarrow \infty$. This is a contradiction. Therefore, $\limsup_{t \rightarrow \infty} \|\nabla F(x_t)\|_{x_t} =0$ under our convergence assumptions. This shows that $\lim_{t \rightarrow \infty} \|\nabla F(x_t)\|_{x_t} =0$ when $\sum_{t=0}^\infty c_t \|\nabla F(x_t)\|_{x_t}^2$, $\sum_{t=0}^\infty c_t\left\langle \nabla F(x_t), \nabla_M f(x_t,\omega_t)\right\rangle_{x_t}$ and $\lim_{t\rightarrow \infty} F(x_t)$ all converge. Hence, $\lim_{t \rightarrow \infty} \|\nabla F(x_t)\|_{x_t} =0$ almost surely.
\end{proof}

\begin{proof}[Proof of Theorem \ref{thm-confined-SGD}]
It is clear that Theorem \ref{thm-confined-SGD} follows from Lemmas \ref{lemma-x-t-confined}, \ref{lemma-F-x-t-converge} and \ref{lemma-nabla-F-x-t-converge}.
\end{proof}

\begin{proof}[Proof of Corollary \ref{cor-confined-SGD}]
Corollary \ref{cor-confined-SGD} is a special case of Theorem \ref{thm-confined-SGD} when $\vf_t=\vf$ for all $t \geq 0$. First, as a fixed constant, $\vf$ is independent of $\{\omega_t\}_{t=0}^\infty$. Note that $\rho(x_t) \leq \rho_1$ for all $t \geq 0$ and that $\left\langle \nabla \rho (x), \nabla f_\omega(x) \right\rangle_x \geq 0$ for every $\omega \in \Omega$ and every $x \in M$ satisfying $\rho(x)\geq \rho_0$. So $A \geq A_t$ and $B\geq B_t$ for all $t \geq 0$. And $c \geq c_t$ for all $t \geq 0$ by the definition of $c$. Thus, the sequence $\{\vf_t=\vf\}_{t=0}^\infty$ satisfies all requirements in Theorem \ref{thm-confined-SGD}.
\end{proof}

\section{Proofs for Results in Section \ref{sec-WLRA-confined-SGD}}\label{sec-proof-sec-3}
\begin{proof}[Proof of Lemma \ref{lemma-g-eta-gamma-RWLRA-1-expectation}]
\[
E_{(\eta,\gamma) \sim \mu}(\hat{f}(P; \eta,\gamma)) = \sum_{i=1}^m\sum_{j=1}^n w_{i,j} \hat{f}(P; i,j) 
= \sum_{i=1}^m\sum_{j=1}^n w_{i,j}(a_{i,j}-p_{i,j})^2  = \hat{F}(P).
\]
\begin{eqnarray*}
E_{(\eta,\gamma) \sim \mu}(f(P; \eta,\gamma)) & = & \sum_{i=1}^m\sum_{j=1}^n w_{i,j} f(P; i,j) = \sum_{i=1}^m\sum_{j=1}^n w_{i,j}\left((a_{i,j}-p_{i,j})^2 + \lambda \|P\|_F^2\right) \\
& = & \sum_{i=1}^m\sum_{j=1}^n w_{i,j}(a_{i,j}-p_{i,j})^2 + \sum_{i=1}^m\sum_{j=1}^n w_{i,j} \lambda \|P\|_F^2 \\
& = & \hat{F}(P) +\lambda \|P\|_F^2 = F(P).
\end{eqnarray*}
And
\[
E_{(\eta,\gamma) \sim \mu}(g(U,\mathbf{x},V; \eta,\gamma)) = E_{(\eta,\gamma) \sim \mu}(f(U D^{k\times k}_k(\mathbf{x}) V^T; \eta,\gamma)) = F(U D^{k\times k}_k(\mathbf{x}) V^T) = G(U,\mathbf{x},V).
\]
\end{proof}

\begin{proof}[Proof of Corollary \ref{cor-random-g-gradient-RWLRA-1}]
The matrices for $\nabla_U \hat{f}_{\eta,\gamma}$, $\nabla_V \hat{f}_{\eta,\gamma}$ and $\nabla_{\mathbf{x}} \hat{f}_{\eta,\gamma}$ follow the Equations \eqref{eq-partial-dev-u}, \eqref{eq-partial-dev-v} and \eqref{eq-partial-dev-x}. Further, the gradient of $g_{\eta,\gamma}$ given by Equation \eqref{eq-g-gradient-RWLRA-1} follows from Lemmas \ref{lemma-stiefel} and \ref{lemma-gradient-submfd}.
\end{proof}

\begin{proof}[Proof of Lemma \ref{lemma-rho-RWLRA-1-confinement}]
For any $r>0$, the set
\[
\{(U,\mathbf{x},V)\in V_k(\R^m)\times \R^k \times V_k(\R^n)~\big{|}~\rho(U,\mathbf{x},V) \leq r^2\} = V_k(\R^m)\times \{\mathbf{x}\in \R^k ~\big{|}~\|\mathbf{x}\| \leq r\}\times V_k(\R^n)
\] 
is compact. Note that $\rho$ is independent of $U$ and $V$. So the gradient of $\rho$ is 
\begin{equation}\label{eq-nabla-rho-RWLRA-1}
\nabla \rho(U,\mathbf{x},V) = (0,2\mathbf{x},0) \in T_U V_k(\R^m)\times \R^k \times T_VV_k(\R^n)
\end{equation} 
for any $(U,\mathbf{x},V)\in V_k(\R^m)\times \R^k \times V_k(\R^n)$. By Equations \eqref{eq-norm-equal} and \eqref{eq-g-gradient-RWLRA-1}, we have that
\begin{eqnarray*}
&&\left\langle \nabla \rho(U,\mathbf{x},V), \nabla g_{\eta,\gamma}(U,\mathbf{x},V)\right\rangle \\
& = & 4\mathbf{x}^T \left[\begin{array}{c}
- (a_{\eta,\gamma}-p_{\eta,\gamma})u_{\eta,1} v_{\gamma,1} + \lambda x_1\\
- (a_{\eta,\gamma}-p_{\eta,\gamma})u_{\eta,2} v_{\gamma,2} + \lambda x_2\\
\vdots \\
- (a_{\eta,\gamma}-p_{\eta,\gamma})u_{\eta,k} v_{\gamma,k} + \lambda x_k
\end{array}
\right] \\
& = & -4\sum_{l=1}^k (a_{\eta,\gamma}-p_{\eta,\gamma})u_{\eta,l} v_{\gamma,l}x_l + 4\lambda\sum_{l=1}^k x_l^2
= -4(a_{\eta,\gamma}-p_{\eta,\gamma})p_{\eta,\gamma} + 4\lambda \|\mathbf{x}\|^2 \\
& = & -4(a_{\eta,\gamma}-p_{\eta,\gamma})p_{\eta,\gamma} + 4\lambda \rho(U,\mathbf{x},V) \geq -a_{\eta,\gamma}^2 + 4\lambda \rho(U,\mathbf{x},V).
\end{eqnarray*}
So $\left\langle \nabla \rho(U,\mathbf{x},V), \nabla g_{\eta,\gamma}(U,\mathbf{x},V)\right\rangle \geq 0$ if $\rho(U,\mathbf{x},V)\geq \rho_0 \geq \frac{\alpha}{4\lambda}$. Then $\left\langle \nabla \rho(U,\mathbf{x},V), \nabla g_{\eta,\gamma}(U,\mathbf{x},V)\right\rangle \geq 0$ if $\rho(U,\mathbf{x},V)\geq \rho_0$.
\end{proof}

Note that, from the proof of Lemma \ref{lemma-rho-RWLRA-1-confinement}, we have
\begin{equation}\label{eq-rho-g-product-RWLRA-1}
\left\langle \nabla \rho(U,\mathbf{x},V), \nabla g_{\eta,\gamma}(U,\mathbf{x},V)\right\rangle = -4(a_{\eta,\gamma}-p_{\eta,\gamma})p_{\eta,\gamma} + 4\lambda \|\mathbf{x}\|^2.
\end{equation}
Next, we give the Hessian of the map $\rho \circ R$.

\begin{lemma}\label{lemma-hessian-RWLRA-1}
Let $\Theta$ be any positive scalar, $\rho$ the confinement function in Definition \ref{def-confinement-RWLRA-1} and $R$ the retraction in Definition \ref{def-retraction-GS-prod}. For any $(U,\mathbf{x},V)  \in V_k(\R^m)\times \R^k \times V_k(\R^n)$ and any $\theta \in [-\Theta,\Theta]$, 
\begin{equation}\label{eq-hessian-RWLRA-1}
\Hess(\rho \circ R_{(U,\mathbf{x},V)})|_{\theta\nabla g_{\eta,\gamma}(U,\mathbf{x},V)}(\nabla g_{\eta,\gamma}(U,\mathbf{x},V) ,\nabla g_{\eta,\gamma}(U,\mathbf{x},V))
= 8\sum_{l=1}^k (-(a_{\eta,\gamma}-p_{\eta,\gamma})u_{\eta,l} v_{\gamma,l} + \lambda x_l)^2.
\end{equation}
\end{lemma}

\begin{proof}
Note that $(\rho \circ R_{(U,\mathbf{x},V)})(Y,\hat{\mathbf{x}},Z)=\|\mathbf{x}+\hat{\mathbf{x}}\|^2$ for any $(U,\mathbf{x},V)  \in V_k(\R^m)\times \R^k \times V_k(\R^n)$ and any $(Y,\hat{\mathbf{x}},Z) \in  T_{U}V_k(\R^m)\times \R^k \times T_{V}V_k(\R^n)$. So $\Hess(\rho \circ R_{(U,\mathbf{x},V)})|_{(Y',\hat{\mathbf{x}}',Z')}((Y,\hat{\mathbf{x}},Z),(Y,\hat{\mathbf{x}},Z))=2\|\hat{\mathbf{x}}\|^2$ for any $(U,\mathbf{x},V)  \in V_k(\R^m)\times \R^k \times V_k(\R^n)$ and any $(Y,\hat{\mathbf{x}},Z), ~(Y',\hat{\mathbf{x}}',Z') \in  T_{U}V_k(\R^m)\times \R^k \times T_{V}V_k(\R^n)$. In particular, it is independent of $(Y',\hat{\mathbf{x}}',Z')$. Thus, for any $(U,\mathbf{x},V)  \in V_k(\R^m)\times \R^k \times V_k(\R^n)$ and any $\theta \in [-\Theta,\Theta]$, we know that $\Hess(\rho \circ R_{(U,\mathbf{x},V)})|_{\theta\nabla g_{\eta,\gamma}(U,\mathbf{x},V)}(\nabla g_{\eta,\gamma}(U,\mathbf{x},V) ,\nabla g_{\eta,\gamma}(U,\mathbf{x},V))=2\|\nabla_{\mathbf{x}} \hat{f}_{\eta,\gamma} + 2\lambda\mathbf{x}\|^2$. Combining this with Corollary \ref{cor-random-g-gradient-RWLRA-1}, we get Equation \eqref{eq-hessian-RWLRA-1}.
\end{proof}

\begin{algorithm}[A More General Stochastic Gradient Descent for Problem \ref{prob-RWLRA-1-reform}]\label{alg-generalized-confined-SGD-RWLRA-1-direct} \

\noindent\makebox[\linewidth]{\rule{\textwidth}{1pt}}

\textbf{Input:} 
\begin{itemize}
    \item[-] the random function $g$ given in Definition \ref{def-random-functions-eta-gamma-RWLRA-1}, 
    \item[-] the retraction $R$ given in Definition \ref{def-retraction-GS-prod}, 
    \item[-] the positive integers $m, n$ and $k$ given in Problem \ref{prob-RWLRA-1-reform}, 
    \item[-] the positive scalar $\lambda$ given in Problem \ref{prob-RWLRA-1-reform}, 
    \item[-] the matrices $A$ and $W$ given in Problem \ref{prob-RWLRA-1-reform}, 
    \item[-] positive scalars $a$, $b$, $\Theta$ and $\Phi_{\min}$, 
    \item[-] a sequence $\{c_t\}_{t=0}^\infty$ of positive numbers satisfying $\sum_{t=0}^\infty c_t =\infty$ and $\sum_{t=0}^\infty c_t^2 <\infty$, 
    \item[-] an initial iterate $(U_0,\mathbf{x}_0,V_0) \in V_k(\R^m)\times \R^k \times V_k(\R^n)$.
\end{itemize}

\textbf{Output:} A sequence of iterates $\{(U_t,\mathbf{x}_t,V_t)\}_{t=0}^\infty  \subset V_k(\R^m)\times \R^k \times V_k(\R^n)$.
\begin{itemize}
	\item \emph{for $t=0,1,2\dots$ do}
	\begin{enumerate}[1.]
        \item Select a random element $(\eta_t, \gamma_t)$ from $\Omega_{m, n}$ with the probability measure $\mu$ independent of $\{(\eta_{\tau}, \gamma_{\tau})\}_{\tau=1}^{t-1}$.
        \item Define $A_t$ and $B_t$ by
        \begin{align}
        A_t: &= \frac{1}{a}\max\left\{\max\{0, ~4(a_{\eta,\gamma}-p_{t,\eta,\gamma})p_{t,\eta,\gamma} - 4\lambda \|\mathbf{x}_t\|^2\}~\big{|}~(\eta,\gamma) \in \Omega_{m\times n}\right\}, \label{eq-vf-bound-a-t-RWLRA-1-direct} \\
        B_t: &= \frac{1}{b} \max\left\{ \sqrt{8\sum_{l=1}^k (-(a_{\eta,\gamma}-p_{t,\eta,\gamma})u_{t,\eta,l} v_{t,\gamma,l} + \lambda x_{t,l})^2}~\big{|}~ (\eta,\gamma) \in \Omega_{m\times n}\right\}, \label{eq-vf-bound-b-t-RWLRA-1-direct}
        \end{align}
        where $U_t = [u_{t,i,j}], ~V_t = [v_{t,i,j}], ~\mathbf{x}_t = [x_{t,l}]$ and $P_t = [p_{t,i,j}]:=  U_t D^{k\times k}_k(\mathbf{x}_t) V_t^T \in \R^{m\times n}$.
        \item Choose random positive number $\vf_t$ so that
        \begin{equation}\label{eq-def-vf-t-RWLRA-1-direct-generalized}
        \vf_t := \max\{A_t,B_t,\frac{c_t}{\Theta}, \Phi_{\min}\}.
        \end{equation}
	\item Set
	\begin{equation}\label{eq-confined-SGD-RWLRA-1-recursion-generalized}
	(U_{t+1},\mathbf{x}_{t+1},V_{t+1})= R_{(U_t,\mathbf{x}_t,V_t)}(-\frac{c_t}{\vf_t}\nabla g_{\eta_t,\gamma_t}(U_t,\mathbf{x}_t,V_t)),
	\end{equation}
        where $\nabla g_{\eta_t,\gamma_t}(U_t,\mathbf{x}_t,V_t)$ is given in Corollary \ref{cor-random-g-gradient-RWLRA-1}.
	\end{enumerate}
	\item \emph{end for}
\end{itemize}
\noindent\makebox[\linewidth]{\rule{\textwidth}{1pt}}
\end{algorithm}

\begin{proposition}\label{prop-SGD-mfd-RWLRA-1-direct-general}
Let $G$ be the function given in Definition \ref{eq-def-G-RWLRA-1} and $\{(U_t,\mathbf{x}_t,V_t)\}_{t=0}^\infty \subset V_k(\R^m)\times \R^k \times V_k(\R^n)$ be the sequence from the Algorithm \ref{alg-generalized-confined-SGD-RWLRA-1-direct}. Then:
\begin{enumerate}
    \item $\{(U_t,\mathbf{x}_t,V_t)\}_{t=0}^\infty$ is contained in the  compact subset $\{(U,\mathbf{x},V)\in V_k(\R^m)\times \R^k \times V_k(\R^n) ~\big{|}~\|\mathbf{x}\|^2\leq \rho_0+ ca + \frac{b^2\sigma}{2}\}$ of $V_k(\R^m)\times \R^k \times V_k(\R^n)$ where $\rho_0$ is given in Equation \eqref{eq-def-rho-0-RWLRA-1}, $a$ and $b$ are given in Algorithm \ref{alg-generalized-confined-SGD-RWLRA-1-direct}, and $c= \max \{c_t~\big{|}~t\geq 0\}$ and $\sigma=\sum_{t=0}^\infty c_t^2$ with the sequence $\{c_t\}_{t=0}^\infty$ is given in Algorithm \ref{alg-generalized-confined-SGD-RWLRA-1-direct}. Therefore, $\{(U_t,\mathbf{x}_t,V_t)\}_{t=0}^\infty$ has convergent subsequences;
    \item $\{G(U_t,\mathbf{x}_t,V_t)\}_{t=0}^\infty$ converges almost surely to a finite number;
    \item $\{\|\nabla G(U_t,\mathbf{x}_t,V_t)\|\}_{t=0}^\infty$ converges almost surely to $0$;
    \item any limit point of $\{(U_t,\mathbf{x}_t,V_t)\}_{t=0}^\infty$ is almost surely a stationary point of $G$.
\end{enumerate}
\end{proposition}

\begin{proof}
Based on the preceding discussion in Section \ref{sec-WLRA-confined-SGD} and this appendix, this proposition is a direct consequence of Theorem \ref{thm-confined-SGD}. The only thing not immediately clear is that the sequence $\{\vf_t\}_{t=0}^\infty$ is also bounded above. To see this, note that
    \begin{eqnarray*}
   && A_t \leq  \frac{1}{a}\sup\left\{\max\{0, ~4(a_{\eta,\gamma}-p_{\eta,\gamma})p_{\eta,\gamma} - 4\lambda \|\mathbf{x}\|^2\}~\big{|}~(U,\mathbf{x},V)\in K_1,~(\eta,\gamma) \in \Omega_{m\times n}\right\}, \\
   && B_t \leq  \frac{1}{b} \sup\left\{ \sqrt{8\sum_{l=1}^k (-(a_{\eta,\gamma}-p_{\eta,\gamma})u_{\eta,l} v_{\gamma,l} + \lambda x_l)^2}~\big{|}~ (U,\mathbf{x},V)\in K_1, ~(\eta,\gamma) \in \Omega_{m\times n}\right\}, \\
    && \frac{c_t}{\Theta} \leq \frac{c}{\Theta},
    \end{eqnarray*}
    where $K_1 = \{(U,\mathbf{x},V)\in V_k(\R^m)\times \R^k \times V_k(\R^n) ~\big{|}~ \|\mathbf{x}\|^2\leq \rho_1\}$ and $U = [u_{i,j}]$, $V = [v_{i,j}]$, $\mathbf{x} = [x_l]$, $P = [p_{i,j}]  :=  U D^{k\times k}(\mathbf{x}) V^T$. Together, these imply that $\{\vf_t\}_{t=0}^\infty$ is bounded above.
\end{proof}

To use Algorithm \ref{alg-generalized-confined-SGD-RWLRA-1-direct} directly, one needs to compute the $\vf_t$ given by Equation \eqref{eq-def-vf-t-RWLRA-1-direct-generalized} in each iteration of the algorithm. When $mn$ is large, such computations may be costly. So we provide manual estimations for these $\{\vf_t\}$, which speed up the computation of each iteration of the algorithm at the expense of having somewhat smaller step sizes. Our estimations of $\vf_t$ also provide suggestions for the values of the constants $a$ and $b$ that should be used in the algorithm.

\begin{lemma}\label{lemma-A-t-B-t-bound-RWLRA-1}
For $t\geq 0$, the $A_t$ and $B_t$ given in Algorithm \ref{alg-generalized-confined-SGD-RWLRA-1-direct} are upper bounded by $\tilde{A}_t$ and $\tilde{B}_t$ given below:
\begin{equation}\label{eq-A-t-bound-RWLRA-1}
A_t\leq \tilde{A}_t := \begin{cases}
0 & \text{if } \|\mathbf{x}_t\|^2 \geq \frac{\alpha}{4\lambda}, \\ 
\frac{4(\sqrt{\alpha} + \|\mathbf{x}_t\|)\|\mathbf{x}_t\|+4\lambda \|\mathbf{x}_t\|^2}{a} < \frac{\lambda +2\sqrt{\lambda}+1}{\lambda a}\alpha & \text{if } \|\mathbf{x}_t\|^2 < \frac{\alpha}{4\lambda},
\end{cases}
\end{equation}

\begin{equation}\label{eq-B-t-bound-RWLRA-1}
B_t\leq \tilde{B}_t := \frac{1}{b}\sqrt{16k (2 \alpha + (2+ \lambda^2) \|\mathbf{x}_t\|^2)} \leq \frac{1}{b}\sqrt{16k (2 \alpha + (2+ \lambda^2) \rho_1)}.
\end{equation}
\end{lemma}

\begin{proof}
Note that, for $(U,\mathbf{x},V)=([u_{i,j}],[x_1,\dots,x_k]^T,[v_{i,j}])\in V_k(\R^m)\times \R^k \times V_k(\R^n)$ and ${P = [p_{i,j}]  :=  U D^{k\times k}(\mathbf{x}) V^T}$, we have $|p_{\eta,\gamma}|\leq \|P\|_F=\|\mathbf{x}\|=\sqrt{\rho(U,\mathbf{x},V)}$ and $|a_{\eta,\gamma}| \leq \sqrt{\alpha}$. Moreover, ${4(a_{\eta,\gamma}-p_{t,\eta,\gamma})p_{t,\eta,\gamma} - 4\lambda \|\mathbf{x}_t\|^2\leq 0}$ if $\|\mathbf{x}_t\|^2\geq \frac{\alpha}{4\lambda}$. So, for $t\geq 0$, we have that
\[
\max\{0, ~4(a_{\eta,\gamma}-p_{t,\eta,\gamma})p_{t,\eta,\gamma} - 4\lambda \|\mathbf{x}_t\|^2\} \begin{cases}
=0 & \text{if } \|\mathbf{x}_t\|^2 \geq \frac{\alpha}{4\lambda}, \\ 
\leq 4(\sqrt{\alpha} + \|\mathbf{x}_t\|)\|\mathbf{x}_t\|+4\lambda \|\mathbf{x}_t\|^2 < \frac{\lambda +2\sqrt{\lambda}+1}{\lambda}\alpha & \text{if } \|\mathbf{x}_t\|^2 < \frac{\alpha}{4\lambda}.
\end{cases}
\]
Thus, for the $A_t$ given in Algorithm \ref{alg-generalized-confined-SGD-RWLRA-1-direct}, we have Equation \eqref{eq-A-t-bound-RWLRA-1}

By Equation \eqref{eq-hessian-RWLRA-1}, for any $(U,\mathbf{x},V)=([u_{i,j}],[x_1,\dots,x_k]^T,[v_{i,j}])\in V_k(\R^m)\times \R^k \times V_k(\R^n)$, ${P = [p_{i,j}]  :=  U D^{k\times k}(\mathbf{x}) V^T}$ and any $\theta \in [-\Theta,\Theta]$, we have
\begin{eqnarray*}
0 & \leq & \Hess(\rho \circ R_{(U,\mathbf{x},V)})|_{\theta\nabla g_{\eta,\gamma}(U,\mathbf{x},V)}(\nabla g_{\eta,\gamma}(U,\mathbf{x},V) ,\nabla g_{\eta,\gamma}(U,\mathbf{x},V)) \\
& = & 8\sum_{l=1}^k (- (a_{\eta,\gamma}-p_{\eta,\gamma})u_{\eta,l} v_{\gamma,l} + \lambda x_l)^2 \\
& \leq & 16 \sum_{l=1}^k \left(( (a_{\eta,\gamma}-p_{\eta,\gamma})u_{\eta,l} v_{\gamma,l})^2 + \lambda^2 x_l^2\right).
\end{eqnarray*}
Since $(U,\mathbf{x},V)  \in V_k(\R^m)\times \R^k \times V_k(\R^n)$, we have that $\sum_{i=1}^m u_{i,l}^2 = \sum_{j=1}^n v_{j,l}^2 =1$ for $l=1,\dots,k$. In particular, $|u_{i,l}|,~|v_{j,l}|\leq 1$ for $l=1,\dots,k$. So, for $l = 1, \dots, k$, 
\[
((a_{\eta,\gamma}-p_{\eta,\gamma})u_{\eta,l} v_{\gamma,l})^2 \leq (a_{\eta,\gamma}-p_{\eta,\gamma})^2 \leq 2(a_{\eta,\gamma}^2+p_{\eta,\gamma}^2)\leq 2 \alpha + 2\|P\|^2 = 2 \alpha + 2\|\mathbf{x}\|^2.
\]
Note that $x_l^2 \leq \|\mathbf{x}\|^2$. Combining the above, we get
\begin{eqnarray*}
0&\leq & \Hess(\rho \circ R_{(U,\mathbf{x},V)})|_{\theta\nabla g_{\eta,\gamma}(U,\mathbf{x},V)}(\nabla g_{\eta,\gamma}(U,\mathbf{x},V) ,\nabla g_{\eta,\gamma}(U,\mathbf{x},V)) \\
& \leq & 16k (2 \alpha + (2+ \lambda^2) \|\mathbf{x}\|^2) 
\leq 16k (2 \alpha + (2+ \lambda^2)\rho_1)
\end{eqnarray*}
for $(U,\mathbf{x},V)\in V_k(\R^m)\times \R^k \times V_k(\R^n)$ satisfying $\rho(U,\mathbf{x},V) \leq \rho_1$. This shows that, for the $B_t$ given in Algorithm \ref{alg-generalized-confined-SGD-RWLRA-1-direct}, we have Equation \eqref{eq-B-t-bound-RWLRA-1}
\end{proof}

With the estimations in Lemma \ref{lemma-A-t-B-t-bound-RWLRA-1}, we have the following corollary.

\begin{corollary}\label{cor-SGD-mfd-RWLRA-1-simplified-vf-t}
Proposition \ref{prop-SGD-mfd-RWLRA-1-direct-general} remains true if we replace the $\vf_t$ in Algorithm \ref{alg-generalized-confined-SGD-RWLRA-1-direct} by 
\begin{equation}\label{eq-def-vf-t-RWLRA-1-simplified-vf-t}
\tilde{\vf}_t := \max\{\tilde{A}_t,\tilde{B}_t,\frac{c_t}{\Theta}, \Phi_{\min}\},
\end{equation}
where $\tilde{A}_t$ and $\tilde{B}_t$ are given in Equations \eqref{eq-A-t-bound-RWLRA-1} and \eqref{eq-B-t-bound-RWLRA-1}.
\end{corollary}

Of course, $\tilde{\vf}_t$ given by Equation \eqref{eq-def-vf-t-RWLRA-1-simplified-vf-t} is larger than $\vf_t$ given by Equation \eqref{eq-def-vf-t-RWLRA-1-direct-generalized}, which leads to a smaller step size $\frac{c_t}{\tilde{\vf}_t}$ in Algorithm \ref{alg-generalized-confined-SGD-RWLRA-1-direct}. 
But, for very large $mn$, one can compute $\tilde{\vf}_t$ much quicker than $\vf_t$.

\begin{corollary}\label{cor-a-b-choices-RWLRA-1}
For the choices of $a$ and $b$ given in Equation \eqref{eq-a-b-choices-RWRLA-1}, we have
\[
\vf_t \leq \tilde{\vf}_t
 \leq  \max\left\{(\lambda +2\sqrt{\lambda}+1)\alpha, \sqrt{32k\alpha\lambda+ 8k(2+ \lambda^2)(2\lambda\rho_0+2c+\sigma)}, \frac{c}{\Theta}, \Phi_{\min}\right\} = O(1) \text{ as } \lambda \rightarrow 0^+.
\]
\end{corollary}

\begin{proof}
By Inequalities \eqref{eq-A-t-bound-RWLRA-1} and \eqref{eq-B-t-bound-RWLRA-1}, we have that 
\begin{eqnarray*}
\vf_t \leq \tilde{\vf}_t & \leq & \max\left\{ \frac{\lambda +2\sqrt{\lambda}+1}{\lambda a}\alpha, \frac{\sqrt{16k (2 \alpha + (2+ \lambda^2) \rho_1)}}{b}, \frac{c}{\Theta}, \Phi_{\min}\right\} \\
& = & \max\left\{ \frac{\lambda +2\sqrt{\lambda}+1}{\lambda a}\alpha, \frac{\sqrt{16k (2\alpha+ (2+ \lambda^2)(\rho_0 + ca + \frac{b^2\sigma}{2}))}}{b}, \frac{c}{\Theta}, \Phi_{\min}\right\},
\end{eqnarray*}

where $\alpha$ is given in Equation \eqref{eq-def-alpha-RWLRA-1} and $\rho_0$ is given in Equation \eqref{eq-def-rho-0-RWLRA-1}. Recall that
\[
a = \frac{1}{\lambda} \text{ and } b = \frac{1}{\sqrt{\lambda}}.
\]
By Equation \eqref{eq-a-b-choices-RWRLA-1}, we have
\begin{eqnarray*}
 && \frac{\lambda +2\sqrt{\lambda}+1}{\lambda a}\alpha = (\lambda +2\sqrt{\lambda}+1)\alpha, \\
 && \frac{\sqrt{16k (2\alpha+ (2+ \lambda^2)(\rho_0 + ca + \frac{b^2\sigma}{2}))}}{b} = \sqrt{32k\alpha\lambda+ 8k(2+ \lambda^2)(2\lambda\rho_0+2c+\sigma)}. 
\end{eqnarray*}
Thus, we have the corollary.
\end{proof}

Corollary \ref{cor-a-b-choices-RWLRA-1} means that the choices of $a$ and $b$ given by Equation \eqref{eq-a-b-choices-RWRLA-1} provide a finite upper bound for $\{\vf_t\}$ and $\{\tilde{\vf}_t\}$ as $\lambda \rightarrow 0^+$. So, for very small $\lambda>0$, we are not shrinking the preferred step size $c_t$ by too much to get the actual step sizes $\frac{c_t}{\vf_t}$ and $\frac{c_t}{\tilde{\vf}_t}$ that guarantee convergence.

\begin{corollary}\label{cor-SGD-mfd-RWLRA-1-single-vf-t}
In Algorithm \ref{alg-generalized-confined-SGD-RWLRA-1-direct}, choose $a$ and $b$ be as in Equation \eqref{eq-a-b-choices-RWRLA-1}, and $\Theta$ and $\Phi_{\min}$ so that
\[
\Phi_{\min} \geq \max\left\{ (\lambda +2\sqrt{\lambda}+1)\alpha, \sqrt{32k\alpha\lambda+ 8k(2+ \lambda^2)(2\lambda\rho_0+2c+\sigma)}\right\} \ \text{and} \
\Theta = \frac{c}{\Phi_{\min}}.
\]
For these values of $a$, $b$, $\Theta$ and $\Phi_{\min}$, the sequence of $\{\vf_t\}_{t=0}^{\infty}$ in Algorithm \ref{alg-generalized-confined-SGD-RWLRA-1-direct} is given by the constant value
\[
\vf_t = \Phi_{\min}.
\]
\end{corollary}

\begin{proof}
Corollary \ref{cor-SGD-mfd-RWLRA-1-single-vf-t} follows from Corollary \ref{cor-confined-SGD}. 
\end{proof}

\begin{proof}[Proof of Proposition \ref{prop-SGD-mfd-RWLRA-1-direct}]
Let $\Phi_{\min}$ be as in Equation \eqref{eq-phi-min-choice-RWRLA-1} and $\Theta$ in Equation \eqref{eq-theta-choice-RWRLA-1}. Proposition \ref{prop-SGD-mfd-RWLRA-1-direct} follows from Corollary \ref{cor-SGD-mfd-RWLRA-1-single-vf-t} as a special case.
\end{proof}

Remark \ref{rk-SGD-mfd-RWLRA-1-single-vf-t} also follows from Corollary \ref{cor-SGD-mfd-RWLRA-1-single-vf-t}.

\section{An Accelerated Line Search Algorithm for Problem \ref{prob-RWLRA-1-reform}}\label{sec-proof-sec-4}
We first briefly review the accelerated line search on manifolds (Algorithm \ref{alg-ALS}) from Absil, Mahony, and Sepulchre (2009). 

\begin{definition}(\cite[Definition 4.2.1]{AMS}\label{def-gradient-related})
Let $f$ be a differentiable function on a Riemannian manifold $M$ and $\{x_t\}$ a sequence in $M$. A sequence of tangent vectors $\{\eta_t \in T_{x_t}M\}$ is called gradient related to $f$ if, for any subsequence $\{x_{t_p}\}$ that converges to a non-critical point of $f$, the corresponding subsequence $\{\eta_{t_p} \in T_{x_{t_p}}M\}$ is bounded and satisfies $\limsup_{p \rightarrow \infty} \left\langle \nabla f(x_{t_p}), \eta_{t_p}\right\rangle_{x_{t_p}} <0$.
\end{definition}

Obviously, the sequence $\{-\nabla f(x_t)\}$ is gradient-related to $f$.
\begin{definition}(\cite[Definition 4.2.2]{AMS}\label{def-Armijo})
Assume that
\begin{itemize}
	\item $M$ is a Riemannian manifold,
	\item $f$ is a differentiable function on $M$,
	\item $x \in M$, $\eta \in T_xM$,
	\item $\overline{\alpha}>0$ and $\beta,\iota \in (0,1)$.
\end{itemize}
The Armijo point $\eta^A=\tau^A\eta=\beta^m \overline{\alpha}\eta$, where $m$ is the smallest non-negative integer satisfying 
\[
f(x)-f(R_x(\beta^m \overline{\alpha}\eta)) \geq -\iota \left\langle \nabla f(x),\beta^m \overline{\alpha}\eta\right\rangle_x.
\]  
$\tau^A$ is called the Armijo step size.
\end{definition}

\begin{algorithm}(\cite[Page 63, Algorithm 1]{AMS}\label{alg-ALS})\

\noindent\makebox[\linewidth]{\rule{\textwidth}{1pt}}

\textbf{Input:} 
\begin{itemize}
    \item[-] a Riemannian manifold $M$, 
    \item[-] a function $f$ continuously differentiable on $M$, 
    \item[-] a retraction $R$ on $M$ given in Definition \ref{def-retraction},  
    \item[-] scalars $\overline{\alpha}>0$ and $\beta,\iota \in (0,1)$, 
    \item[-] an initial iterate $x_0 \in M$.
\end{itemize}

\textbf{Output:} A sequence of iterates $\{x_t\}_{t=0}^\infty \subset M$.
\begin{itemize}
	\item \emph{for $t=0,1,2\dots$ do}
	\begin{enumerate}[1.]
		\item Pick $\eta_t \in  T_{x_t}M$ such that the sequence $\{\eta_t\}$ is gradient related to $f$.
		\item Select $x_{t+1}$ so that
		\begin{equation}\label{eq-ALS-recursion}
		f(x_t)-f(x_{t+1}) \geq f(x_t)-f(R_{x_t}(\tau_t^A\eta_t)),
		\end{equation}
		where $\tau_t^A$ is the Armijo step size for the given $\overline{\alpha},\beta,\iota$ and $\eta_t$ (see Definition \ref{def-Armijo}). 
	\end{enumerate}
	\item \emph{end for}
\end{itemize}
\noindent\makebox[\linewidth]{\rule{\textwidth}{1pt}}
\end{algorithm}

For the convergence of Algorithm \ref{alg-ALS}, we have the following theorem.

\begin{theorem}(\cite[Theorem 4.3.1 and Corollary 4.3.2]{AMS}\label{thm-ALS-converges})
Under the assumptions in Algorithm \ref{alg-ALS}. Every convergent subsequence of $\{x_l\}$ converges to a critical point of $f$. If the set $\{x\in M~|~f(x)\leq f(x_0)\}$ is compact, then $\lim_{l \rightarrow \infty} \|\nabla f(x_l)\|_{x_l}=0$ and $\{x_l\}$ has convergent subsequences.
\end{theorem}

Next, we apply Algorithm \ref{alg-ALS} to design a convergent accelerated line search algorithm for Problem \ref{prob-RWLRA-1-reform}. To design our accelerated line search for Problem \ref{prob-RWLRA-1-reform}, we first compute the gradient of the function $G:V_k(\R^m)\times \R^k \times V_k(\R^n)\rightarrow \R$ given in Definition \ref{def-func-G}. To do that, we need some preparations first.

\begin{definition}\label{def-Delta}
Define $\Delta: = \{(i, j) \in \Omega_{m,n} : w_{i,j} > 0\}$, where $\Omega_{m,n}$ is given in Definition \ref{def-Omega-m-n-measure} and ${W = [w_{ij}]}$ is given in Problem \ref{prob-RWLRA-1-reform}.
For $i \in \{1, 2, \ldots, m\}$, define $\Delta_{i, \ast}: = \{j \in \{1, 2, \ldots, n\}: w_{i, j} > 0\}$.
For $j \in \{1, 2, \ldots, n\}$, define {$\Delta_{\ast, j}: = \{i \in \{1, 2, \ldots, m\}: w_{i, j} > 0\}$}.
\end{definition}

\begin{remark}
For $(\eta,\gamma)\in \Omega_{m,n}$, we view the function $\hat{f}_{\eta,\gamma}$ given in Definition \ref{def-random-functions-eta-gamma-RWLRA-1} as a composite function of $(U,\mathbf{x},V)$ via $\hat{f}_{\eta,\gamma} = \hat{f}_{\eta,\gamma}(P)$ and $P = U D^{k\times k}_k(\mathbf{x}) V^T$, where $P=[p_{i,j}] \in \R^{m\times n}$, $U=[u_{i,j}] \in \R^{m\times k}$, $V=[v_{i,j}] \in \R^{n\times k}$ and $\mathbf{x}=[x_1,\dots,x_k]^T\in \R^k$.
Note that, for $(\eta,\gamma)\in \Omega_{m,n}$,
\[
\frac{\partial \hat{f}_{\eta,\gamma}}{\partial p_{\eta,\gamma}} = -2(a_{\eta,\gamma} - p_{\eta,\gamma}).
\]
Also, for $(\eta,\gamma)\in \Omega_{m,n}$,
\begin{equation}\label{eq-p-uxv}
p_{\eta,\gamma} = \sum_{\zeta=1}^k u_{\eta,\zeta}x_{\zeta} v_{\gamma,\zeta},
\end{equation} 
and thus, for $l=1,\dots,k$, $i=1,\dots,m$ and $j=1,\dots,n$, 
\begin{equation}\label{eq-partial-p-uxv}
\frac{\partial p_{\eta,\gamma}}{\partial u_{i,l}} = \delta_{\eta,i}x_l v_{\gamma,l}, \hspace{1pc}
\frac{\partial p_{\eta,\gamma}}{\partial v_{j,l}} = \delta_{\gamma,j}x_l u_{\eta,l} \hspace{1pc}
\text{and}\hspace{1pc}
\frac{\partial p_{\eta,\gamma}}{\partial x_l} = u_{\eta,l}v_{\gamma,l}.
\end{equation}
By the Chain Rule, we have that, for $(\eta,\gamma)\in \Omega_{m,n}$, $l=1,\dots,k$, $i=1,\dots,m$ and $j=1,\dots,n$, 
\begin{eqnarray}
\frac{\partial \hat{f}_{\eta,\gamma}}{\partial u_{i,l}} & = & -2\delta_{\eta,i} (a_{\eta,\gamma}-p_{\eta,\gamma})x_l v_{\gamma,l}, \label{eq-partial-dev-u}\\
\frac{\partial \hat{f}_{\eta,\gamma}}{\partial v_{j,l}} & = & -2\delta_{\gamma,j} (a_{\eta,\gamma}-p_{\eta,\gamma})x_l u_{\eta,l}, \label{eq-partial-dev-v}\\
\frac{\partial \hat{f}_{\eta,\gamma}}{\partial x_{l}} & = & -2(a_{\eta,\gamma}-p_{\eta,\gamma})u_{\eta,l} v_{\gamma,l}. \label{eq-partial-dev-x}
\end{eqnarray}
\end{remark}

With the preparations above, we are ready to give the gradient of $G$ in Corollary \ref{cor-G-gradient-RWLRA-1} below by Lemmas \ref{lemma-gradient-submfd} and \ref{lemma-stiefel-orthogonal-proj}.

\begin{corollary}\label{cor-G-gradient-RWLRA-1}
Define the matrices
\begin{eqnarray*}
\nabla_U \hat{F} &=& \left[\frac{\partial \hat{F}}{\partial u_{i,l}}\right]_{m\times k}
= \left[\sum_{j \in \Delta_{i, \ast}}-2w_{i,j}(a_{i,j} - p_{i,j})x_l v_{j,l}\right]_{m\times k}, \\
\nabla_V \hat{F} &=& \left[\frac{\partial \hat{F}}{\partial v_{j,l}}\right]_{n\times k}
= \left[\sum_{i \in \Delta_{\ast, j}}-2w_{i,j}(a_{i,j} - p_{i,j})x_l u_{i,l}\right]_{n\times k}, \\
\nabla_{\mathbf{x}} \hat{F} &=& \left[\frac{\partial \hat{F}}{\partial x_l}\right]_k
= \left[\sum_{(i, j) \in \Delta} -2w_{i,j}(a_{i,j} - p_{i,j})u_{i, l}v_{j, l}\right]_k.
\end{eqnarray*}
Then the gradient of $G$ at any $(U,\mathbf{x},V) \in V_k(\R^m)\times \R^k \times V_k(\R^n)$ is 
\begin{equation}\label{eq-G-gradient-RWLRA-1}
\nabla G(U,\mathbf{x},V) = (\Pi_U(\nabla_U \hat{F}), \nabla_{\mathbf{x}} \hat{F} + 2\lambda \mathbf{x}, \Pi_V(\nabla_V \hat{F})) \in T_U V_k(\R^m)\times \R^k \times T_VV_k(\R^n),
\end{equation}
where $\Pi_U$ and $\Pi_V$ are the orthogonal projections given in Lemma \ref{lemma-stiefel-orthogonal-proj}.
\end{corollary}

\begin{proof}
Note that, $\hat{F} = \sum_{\eta=1}^m\sum_{\gamma=1}^n w_{\eta,\gamma}\hat{f}_{\eta,\gamma}$. Therefore, by Equations \eqref{eq-partial-dev-u}, \eqref{eq-partial-dev-v} and \eqref{eq-partial-dev-x},
\begin{eqnarray*}
\frac{\partial \hat{F}}{\partial u_{i,l}} & = &
\sum_{\eta=1}^m\sum_{\gamma=1}^n w_{\eta,\gamma}\frac{\partial \hat{f}_{\eta,\gamma}}{\partial u_{i,l}}
= \sum_{j \in \Delta_{i, \ast}}-2w_{i,j}(a_{i,j} - p_{i,j})x_l v_{j,l}, \\
\frac{\partial \hat{F}}{\partial v_{j,l}} & = &
\sum_{\eta=1}^m\sum_{\gamma=1}^n w_{\eta,\gamma}\frac{\partial \hat{f}_{\eta,\gamma}}{\partial v_{j,l}}
= \sum_{i \in \Delta_{\ast, j}}-2w_{i,j}(a_{i,j} - p_{i,j})x_l u_{i,l}, \\
\frac{\partial \hat{F}}{\partial x_l} & = &
\sum_{\eta=1}^m\sum_{\gamma=1}^n w_{\eta,\gamma}\frac{\partial \hat{f}_{\eta,\gamma}}{\partial x_l}
= \sum_{(i, j) \in \Delta} -2w_{i,j}(a_{i,j} - p_{i,j})u_{i, l}v_{j, l}.
\end{eqnarray*}
Then, by Lemmas \ref{lemma-stiefel} and \ref{lemma-gradient-submfd}, the gradient of $G$ is given by Equation \eqref{eq-G-gradient-RWLRA-1}.
\end{proof}

Now, we are ready to give our accelerated line search for Problem \ref{prob-RWLRA-1-reform}.

\begin{algorithm}[An Accelerated Line Search for Problem \ref{prob-RWLRA-1-reform}]\label{alg-ALS-RWLRA-1} \

\noindent\makebox[\linewidth]{\rule{\textwidth}{1pt}}

\textbf{Input:} 
\begin{itemize}
    \item[-] the function $G$ given in Definition \ref{def-func-G}, 
    \item[-] the retraction $R$ given in Definition \ref{def-retraction-GS-prod}, 
    \item[-] the positive integers $m, n$ and $k$ given in Problem \ref{prob-RWLRA-1-reform}, 
    \item[-] the positive scalar $\lambda$ given in Problem \ref{prob-RWLRA-1-reform},
    \item[-] the matrices $A$ and $W$ given in Problem \ref{prob-RWLRA-1-reform}, 
    \item[-] scalars $\overline{\alpha}>0$ and $\beta,\iota \in (0,1)$, 
    \item[-] an initial iterate $(U_0,\mathbf{x}_0,V_0) \in V_k(\R^m)\times \R^k \times V_k(\R^n)$.
\end{itemize}

\textbf{Output:} A sequence of iterates $\{(U_t,\mathbf{x}_t,V_t)\}_{t=0}^\infty \subset V_k(\R^m)\times \R^k \times V_k(\R^n)$.
\begin{itemize}
	\item \emph{for $t=0,1,2\dots$ do}
	\begin{enumerate}[1.]
		\item Select $(U_{t+1},\mathbf{x}_{t+1},V_{t+1})$ so that
		\begin{equation}\label{eq-RWLRA-1-ALS-recursion}
		G(U_t,\mathbf{x}_t,V_t)-G(U_{t+1},\mathbf{x}_{t+1},V_{t+1}) \geq G(U_t,\mathbf{x}_t,V_t)-G(R_{(U_t,\mathbf{x}_t,V_t)}(-\tau_t^A \nabla G(U_t,\mathbf{x}_t,V_t))),
		\end{equation}
		where $\nabla G(U_t,\mathbf{x}_t,V_t)$ is given in Corollary \ref{cor-G-gradient-RWLRA-1} and $\tau_t^A$ is the Armijo step size for the given $\overline{\alpha},\beta$, $\iota$ and $-\nabla G(U_t,\mathbf{x}_t,V_t)$ (see Definition \ref{def-Armijo}). 
	\end{enumerate}
	\item \emph{end for}
\end{itemize}
\noindent\makebox[\linewidth]{\rule{\textwidth}{1pt}}
\end{algorithm}

\begin{proposition}\label{prop-ALS-RWLRA-1}
Let $\{(U_t,\mathbf{x}_t,V_t)\}_{t=0}^\infty$ be the sequence from Algorithm \ref{alg-ALS-RWLRA-1}. Then ${\lim_{t \rightarrow \infty} \|\nabla G(U_t,\mathbf{x}_t,V_t)\|=0}$ and $\{(U_t,\mathbf{x}_t,V_t)\}_{t=0}^\infty$ has a subsequence converging to a critical point of $G$.
\end{proposition}

\begin{proof}
Note that, with $(U_0,\mathbf{x}_0,V_0)$ given in Algorithm \ref{alg-ALS-RWLRA-1},
\[
G(U_0,\mathbf{x}_0,V_0) \geq G(U,\mathbf{x},V)
= \hat{F}(U D^{k\times k}_k(\mathbf{x}) V^T) + \lambda \|\mathbf{x}\|^2
\geq \lambda \|\mathbf{x}\|^2.
\]
This implies that $\{(U,\mathbf{x},V)  \in V_k(\R^m)\times \R^k \times V_k(\R^n) ~|~ G(U,\mathbf{x},V) \leq G(U_0,\mathbf{x}_0,V_0)\}$ is a compact subset of $V_k(\R^m)\times \R^k \times V_k(\R^n)$. Since the sequence $\{-\nabla G(U_t,\mathbf{x}_t,V_t)\}_{t=0}^\infty$ is gradient related to $G$, we conclude that Theorem \ref{prop-ALS-RWLRA-1} follows from Theorem \ref{thm-ALS-converges}.
\end{proof}

\section{Regularized Weighted Low-Rank Approximation Problem \ref{prob-RWLRA-2}}\label{sec-RWLRA-2}
Note that, for $P\in \R^{m\times n}$, $\rank P \leq k$ if and only if $P=XY^T$ for some $(X,Y) \in \R^{m\times k} \times \R^{n\times k}$. This observation leads to the reformulation of Problem \ref{prob-WLRA} into Problem \ref{prob-RWLRA-2}, which is studied in the previous literature (for example, \cite{Ban-Woodruff-Zhang:2019}). To estimate the solution to Problem \ref{prob-RWLRA-2}, we work over $\R^{m\times k} \times \R^{n\times k}$ in this Appendix.

The retraction we use on $\R^{m\times k} \times \R^{n\times k}$ is simply the matrix addition. More precisely, the retraction $R:T(\R^{m\times k} \times \R^{n\times k})\rightarrow \R^{m\times k} \times \R^{n\times k}$ is given by 
\begin{equation}\label{eq-retraction-RWLRA-2}
R_{(X,Y)}(\hat{X},\hat{Y})=(X+\hat{X},Y+\hat{Y})
\end{equation}
for $(X,Y) \in \R^{m\times k} \times \R^{n\times k}$ and $(\hat{X},\hat{Y})\in T_{(X,Y)}(\R^{m\times k} \times \R^{n\times k}) = \R^{m\times k} \times \R^{n\times k}$. 

\subsection{A Stochastic Gradient Descent for Problem \ref{prob-RWLRA-2}}\label{subsec-RWLRA-2-confined-SGD}

It turns out that Theorem \ref{thm-confined-SGD} also provides a convergent stochastic gradient descent for Problem \ref{prob-RWLRA-2}. To design the stochastic gradient descent, we still work on the probability space $\Omega_{m,n}$ with the probability distribution $\mu$ given in Definition \ref{def-Omega-m-n-measure}. We first define a random function satisfying Assumption (\ref{assumption-random-function}) of Theorem \ref{thm-confined-SGD}.

\begin{definition}\label{def-random-f-eta-gamma-RWLRA-2}
Define $h: \R^{m\times k}\times \R^{n\times k}\times \Omega_{m,n} \rightarrow \R$ by
\begin{equation}\label{eq-def-f-eta-gamma-RWLRA-2}
h(X,Y;\eta,\gamma) = (a_{\eta,\gamma}-\sum_{l=1}^k x_{\eta,l}y_{\gamma,l})^2+ \lambda (\|X\|_F^2+\|Y\|_F^2),
\end{equation}
for all $X=[x_{i,j}] \in \R^{m\times k}$, $Y=[y_{i,j}] \in \R^{n\times k}$ and $(\eta,\gamma) \in \Omega_{m,n}$. For notational convenience, we define the function $h_{\eta,\gamma}: \R^{m\times k}\times \R^{n\times k} \rightarrow \R$ for $(\eta,\gamma) \in \Omega_{m,n}$ by
\begin{equation}\label{eq-def-h-eta-gamma-RWLRA-2}
h_{\eta,\gamma} = h(X,Y;\eta,\gamma)
\end{equation}
for $(X, Y)\in \R^{m\times k}\times \R^{n\times k}$.
\end{definition}

\begin{lemma}\label{lemma-f-expectation-RWLRA-2}
Let function $h$ be given in Definition \ref{def-random-f-eta-gamma-RWLRA-2}. For all $(X,Y) \in \R^{m\times k} \times \R^{n\times k}$, we have
\[
E_{(\eta,\gamma)\sim \mu}(h(X,Y; \eta,\gamma)) = H(X,Y)
\]
where the expectation is taken over the probability space $\Omega_{m,n}$ with the probability distribution $\mu$ given in Definition \ref{def-Omega-m-n-measure}. Therefore, the function $h$ given in Definition \ref{def-random-f-eta-gamma-RWLRA-2} is the random function with expectation $H$ given in Equation \eqref{def-function-H-RWLRA-2}.
\end{lemma}

\begin{proof}
$E_{(\eta,\gamma)\sim \mu}(h(X,Y; \eta,\gamma)) = \sum_{i=1}^m \sum_{j=1}^n w_{i,j} h(X,Y; i,j) = H(X,Y)$.
\end{proof}

Next, we compute the gradient of the function $h_{\eta,\gamma}: \R^{m\times k}\times \R^{n\times k} \rightarrow \R$.
\begin{corollary}\label{cor-random-f-gradient-RWLRA-2}
For all  $X=[x_{i,j}] \in \R^{m\times k}$ and $Y=[y_{i,j}] \in \R^{n\times k}$, write $p_{i,j} = \sum_{l=1}^k x_{i,l}y_{j,l}$. For $h_{\eta,\gamma}: \R^{m\times k} \times \R^{n\times k} \rightarrow \R$ given in Definition \ref{def-random-f-eta-gamma-RWLRA-2}, we have
\begin{eqnarray*}
\frac{\partial h_{\eta,\gamma}(X,Y)}{\partial x_{i,l}} & = & -2\delta_{\eta,i} (a_{\eta,\gamma}-p_{\eta,\gamma})y_{\gamma,l}+2\lambda x_{i,l}, \\
\frac{\partial h_{\eta,\gamma}(X,Y)}{\partial y_{j,l}} & = & -2\delta_{\gamma,j} (a_{\eta,\gamma}-p_{\eta,\gamma})x_{\eta,l}+2\lambda y_{j,l}.
\end{eqnarray*}
The gradient of $h_{\eta,\gamma}$ is 
\begin{equation}\label{eq-random-f-gradient-RWLRA-2}
\nabla h_{\eta,\gamma}(X,Y) = (\nabla_X h_{\eta,\gamma}(X,Y), \nabla_Y h_{\eta,\gamma}(X,Y)) = \left(\left[\frac{\partial h_{\eta,\gamma}(X,Y)}{\partial x_{i,l}}\right], \left[\frac{\partial h_{\eta,\gamma}(X,Y)}{\partial y_{j,l}}\right]\right) \in \R^{m\times k} \times \R^{n\times k}.
\end{equation}
\end{corollary}

\begin{proof}
This is just a simple computation of partial derivatives.
\end{proof}

Next we define a confinement function for the random function $h$ given in Equation \eqref{eq-def-f-eta-gamma-RWLRA-2}.

\begin{definition}\label{def-confinement-RWLRA-2}
Define $\rho:\R^{m\times k} \times \R^{n\times k} \rightarrow \R$ by
\begin{equation}\label{eq-def-confinement-RWLRA-2}
\rho(X,Y)= \|X\|_F^2+\|Y\|_F^2
\end{equation}
for all $(X,Y) \in \R^{m\times k} \times \R^{n\times k}$. 
\end{definition}

\begin{lemma}\label{lemma-confinement-RWLRA-2}
Define 
\begin{equation}\label{eq-def-rho-0-RWLRA-2}
\rho_0 = \max\{\|X_0\|_F^2+\|Y_0\|_F^2, \frac{\alpha}{2\lambda}\}, 
\end{equation}
where 
$\alpha$ is given in Equation \eqref{eq-def-alpha-RWLRA-1} and $(X_0, Y_0)$ is the initial iterate of Algorithm \ref{alg-confined-SGD-RWLRA-2-direct}. With the $\rho_0$ given in Equation \eqref{eq-def-rho-0-RWLRA-2}, the function $\rho$ given in Definition \ref{def-confinement-RWLRA-2} and the random function $h$ in Definition \ref{def-random-f-eta-gamma-RWLRA-2} satisfy that
\begin{itemize}
    \item $\rho(X_0,Y_0) \leq \rho_0$,
    \item $\langle \nabla \rho(X,Y), \nabla h_{\eta, \gamma}(X,Y) \rangle \geq 0$ for $(\eta,\gamma) \in \Omega_{m,n}$ and $(X,Y)\in \R^{m\times k}\times \R^{n\times k}$ satisfying $\rho(X,Y) \geq \rho_0$.
\end{itemize}
That is, the function $\rho$ is a confinement for the random function $h$ in Definition \ref{def-random-f-eta-gamma-RWLRA-2} on the Euclidean space $\R^{m\times k} \times \R^{n\times k}$, and $\rho_0$ satisfies Assumptions (\ref{assumption-confinement-function}) and (\ref{assumption-initial}) of Theorem \ref{thm-confined-SGD}.
\end{lemma}

\begin{proof}
Note that $\nabla \rho(X,Y) = (2X,2Y)$. So 
\begin{eqnarray*}
&& \left\langle \nabla \rho(X,Y),  \nabla h_{\eta,\gamma}(X,Y)\right\rangle = \left\langle (2X,2Y),  ([\frac{\partial h_{\eta,\gamma}(X,Y)}{\partial x_{i,l}}], [\frac{\partial h_{\eta,\gamma}(X,Y)}{\partial y_{j,l}}])\right\rangle \\
& = & 2\Tr(X^T\left[\frac{\partial h_{\eta,\gamma}(X,Y)}{\partial x_{i,l}}\right]) + 2\Tr(Y^T \left[\frac{\partial h_{\eta,\gamma}(X,Y)}{\partial y_{j,l}}\right]) \\
& = & \sum_{i=1}^m \sum_{l=1}^k 2x_{i,l}\frac{\partial h_{\eta,\gamma}(X,Y)}{\partial x_{i,l}} + \sum_{j=1}^n \sum_{l=1}^k 2y_{j,l}\frac{\partial h_{\eta,\gamma}(X,Y)}{\partial y_{j,l}} \\
& = & \sum_{i=1}^m \sum_{l=1}^k 2x_{i,l}\left(-2\delta_{\eta,i} (a_{\eta,\gamma}-p_{\eta,\gamma})y_{\gamma,l}+2\lambda x_{i,l}\right)+ \sum_{j=1}^n \sum_{l=1}^k 2y_{j,l}\left(-2\delta_{\gamma,j} (a_{\eta,\gamma}-p_{\eta,\gamma})x_{\eta,l}+2\lambda y_{j,l}\right) \\
& = & -4(a_{\eta,\gamma}-p_{\eta,\gamma})p_{\eta,\gamma} + 4 \lambda \|X\|_F^2 -  4(a_{\eta,\gamma}-p_{\eta,\gamma})p_{\eta,\gamma} + 4 \lambda \|Y\|_F^2 \\
& = & -8(a_{\eta,\gamma}-p_{\eta,\gamma})p_{\eta,\gamma} + 4 \lambda\rho(X,Y) \geq -2a_{\eta,\gamma}^2+4 \lambda\rho(X,Y) \geq -2\alpha + 4 \lambda\rho(X,Y) \geq 0
\end{eqnarray*}
if $\rho(X,Y) \geq \rho_0\geq \frac{\alpha}{2\lambda}$.
Then $\left\langle \nabla \rho(X,Y), \nabla h_{\eta,\gamma}(X,Y)\right\rangle \geq 0$ 
if $\rho(X,Y) \geq \rho_0$.
\end{proof}

From the proof of Lemma \ref{lemma-confinement-RWLRA-2}, we know that 
\begin{equation}\label{eq-rho-f-product-RWLRA-2}
\left\langle \nabla \rho(X,Y), \nabla h_{\eta,\gamma}(X,Y)\right\rangle = -8(a_{\eta,\gamma}-p_{\eta,\gamma})p_{\eta,\gamma} + 4 \lambda(\|X\|_F^2+\|Y\|_F^2).
\end{equation}

Next, we give the Hessian of the map $\rho \circ R$.

\begin{lemma}\label{lemma-hessian-RWLRA-2}
Let $\Theta$ be any positive scalar, 
$\rho$ be the confinement function in Definition \ref{def-confinement-RWLRA-1} and $R$ be the retraction in Definition \ref{def-retraction-GS-prod}. For any $(X, Y) \in \R^{m\times k} \times \R^{n\times k}$ and any $\theta \in [-\Theta,\Theta]$, 
\begin{eqnarray}
\label{eq-hessian-gradient-RWLRA-2} && \Hess(\rho \circ R_{(X,Y)})_{\theta \nabla h_{\eta,\gamma}(X,Y)}(\nabla h_{\eta,\gamma}(X,Y),\nabla h_{\eta,\gamma}(X,Y)) \\
\nonumber & = & 4 \left( (a_{\eta,\gamma}-p_{\eta,\gamma})^2 \sum_{l=1}^k(x_{\eta,l}^2 + y_{\gamma,l}^2) + 4 \lambda(a_{\eta,\gamma}-p_{\eta,\gamma})p_{\eta,\gamma} + \lambda^2 (\|X\|_F^2 + \|Y\|_F^2)\right)\geq 0,
\end{eqnarray}
where $[p_{i,j}]=XY^T \in \R^{ m \times n}$.
\end{lemma}

\begin{proof}
Recall that $R_{(X,Y)}(\hat{X},\hat{Y})=(X+\hat{X},Y+\hat{Y})$ for $(X,Y) \in \R^{m\times k} \times \R^{n\times k}$ and $(\hat{X},\hat{Y})\in T_{(X,Y)}(\R^{m\times k} \times \R^{n\times k}) = \R^{m\times k} \times \R^{n\times k}$. Therefore $(\rho \circ R_{(X,Y)})(\hat{X},\hat{Y}) = \|X+\hat{X}\|_F^2+\|Y+\hat{Y}\|_F^2$. Thus, 
\begin{equation}\label{eq-hessian-RWLRA-2}
\Hess(\rho \circ R_{(X,Y)})_{(\theta \hat{X},\theta\hat{Y})}((\hat{X},\hat{Y}),(\hat{X},\hat{Y})) = 2(\|\hat{X}\|_F^2+\|\hat{Y}\|_F^2) \geq 0.
\end{equation}
In particular, by Lemma \ref{cor-random-f-gradient-RWLRA-2}, we get Equation \eqref{eq-hessian-gradient-RWLRA-2}.
\end{proof}

\begin{algorithm}[A Stochastic Gradient Descent for Problem \ref{prob-RWLRA-2}]\label{alg-confined-SGD-RWLRA-2-direct}\

\noindent\makebox[\linewidth]{\rule{\textwidth}{1pt}}

\textbf{Input:} 
\begin{itemize}
    \item[-] the function $h$ given in Definition \ref{def-random-f-eta-gamma-RWLRA-2}, 
    \item[-] the retraction $R$ given in Equation \eqref{eq-retraction-RWLRA-2},
    \item[-] the positive integers $m, n$ and $k$ given in Problem \ref{prob-RWLRA-2}, 
    \item[-] the positive scalar $\lambda$ given in Problem \ref{prob-RWLRA-2},
    \item[-] the matrices $A$ and $K$ given in Problem \ref{prob-RWLRA-2},
    \item[-] positive scalars $a$, $b$, $\Theta$ and $\Phi_{\min}$, 
    \item[-] a sequence $\{c_t\}_{t=0}^\infty$ of positive numbers satisfying $\sum_{t=0}^\infty c_t =\infty$ and $\sum_{t=0}^\infty c_t^2 <\infty$, 
    \item[-] an initial iterate $(X_0,Y_0)\in \R^{m\times k}\times \R^{n\times k}$. 
\end{itemize}

\textbf{Output:} A sequence of iterates $\{(X_t,Y_t)\}_{t=0}^\infty \subset \R^{m\times k}\times \R^{n\times k}$.
\begin{itemize}
	\item \emph{for $t=0,1,2\dots$ do}
	\begin{enumerate}[1.]
        \item Select a random element $(\eta_t, \gamma_t)$ from $\Omega_{m, n}$ with the probability measure $\mu$ independent of $\{(\eta_{\tau}, \gamma_{\tau})\}_{\tau=1}^{t-1}$.
        \item Define $A_t$ and $B_t$ by
        \begin{align}
        A_t: &= \frac{1}{a}\max\left\{\max\{0, ~8(a_{\eta,\gamma}-p_{t,\eta,\gamma})p_{t,\eta,\gamma} - 4 \lambda(\|X_t\|_F^2+\|Y_t\|_F^2)\}~\big{|}~(\eta,\gamma) \in \Omega_{m\times n}\right\}, \label{eq-vf-bound-a-t-RWLRA-2-direct} \\
        B_t: &= \frac{1}{b} \max\left\{ \sqrt{4 \left( (a_{\eta,\gamma}-p_{t,\eta,\gamma})^2 \sum_{l=1}^k(x_{t,\eta,l}^2 +y_{t,\gamma,l}^2) + 4 \lambda(a_{\eta,\gamma}-p_{t,\eta,\gamma})p_{t,\eta,\gamma} + \lambda^2 (\|X_t\|_F^2 + \|Y_t\|_F^2)\right)}~\big{|}~ (\eta,\gamma) \in \Omega_{m\times n}\right\}, \label{eq-vf-bound-b-t-RWLRA-2-direct}
        \end{align}
        where $X_t = [x_{t,i,j}],~ Y_t = [y_{t,i,j}]$ and $P_t = [p_{t,i,j}]:=  X_t Y_t^T \in \R^{m\times n}$.
        \item Define random positive number $\vf_t$ by
        \begin{equation}\label{eq-def-vf-t-RWLRA-2-direct}
        \vf_t := \max\{A_t,B_t,\frac{c_t}{\Theta}, \Phi_{\min}\}.
        \end{equation}
	\item Set
	\begin{equation}\label{eq-confined-SGD-recursion-RWLRA-2}
	(X_{t+1},Y_{t+1}) = (X_t,Y_t) - \frac{c_t}{\vf_t} \nabla h_{\eta_t,\gamma_t}(X_t,Y_t),
	\end{equation}
        where $\nabla h_{\eta_t,\gamma_t}(X_t, Y_t)$ is given in Lemma \ref{cor-random-f-gradient-RWLRA-2}.
	\end{enumerate}
	\item \emph{end for}
\end{itemize}
\noindent\makebox[\linewidth]{\rule{\textwidth}{1pt}}
\end{algorithm}

\begin{proposition}\label{prop-confined-SGD-RWLRA-2}
Let H be the function given in Equation \eqref{def-function-H-RWLRA-2} and $\{(X_t,Y_t)\}_{t=0}^\infty \subset \R^{m\times k} \times \R^{n\times k}$ be the sequence from the Algorithm \ref{alg-confined-SGD-RWLRA-2-direct}. Then:
\begin{enumerate}
    \item $\{(X_t,Y_t)\}_{t=0}^\infty$ is contained in the compact subset ${\{(X,Y)\in \R^{m\times k} \times \R^{n\times k}~\big{|}~\|X\|_F^2+\|Y\|_F^2\leq \rho_0+ ca + \frac{b^2\sigma}{2}\}}$ of $\R^{m\times k} \times \R^{n\times k}$, where $\rho_0$ given in Equation \eqref{eq-def-rho-0-RWLRA-2}, $a$ and $b$ are given in Algorithm \ref{alg-confined-SGD-RWLRA-2-direct}, and $c= \max \{c_t~\big{|}~t\geq 0\}$ and $\sigma=\sum_{t=0}^\infty c_t^2$ with the sequence $\{c_t\}_{t=0}^\infty$ given in Algorithm \ref{alg-confined-SGD-RWLRA-2-direct}. Therefore, $\{(X_t,Y_t)\}_{t=0}^\infty$ has convergent subsequences;
    \item $\{H(X_t,Y_t)\}_{t=0}^\infty$ converges almost surely to a finite number;
    \item $\{\|\nabla H(X_t,Y_t)\|\}_{t=0}^\infty$ converges almost surely to $0$;
    \item any limit point of $\{(X_t,Y_t)\}_{t=0}^\infty$ is almost surely a stationary point of $H$.
\end{enumerate}
\end{proposition}

\begin{proof}
This Proposition follows from Theorem \ref{thm-confined-SGD} and the preceding discussions in this Appendix.
\end{proof}

To use Algorithm \ref{alg-confined-SGD-RWLRA-2-direct} directly, one needs to compute the $\vf_t$ given by Equation \eqref{eq-def-vf-t-RWLRA-2-direct} in each iteration of the algorithm. When $mn$ is large, such computations may be costly. So we provide manual estimations for these $\{\vf_t\}$, which speed up the computation of each iteration of the algorithm at the expense of having somewhat smaller step sizes. Our estimates of $\vf_t$ also provide suggestions for the values of the constants $a$ and $b$ that should be used in the algorithm.

\begin{lemma}\label{lemma-A-t-B-t-bound-RWLRA-2}
For $t\geq 0$, the $A_t$ and $B_t$ given in Algorithm \ref{alg-confined-SGD-RWLRA-2-direct} are upper bounded by $\tilde{A}_t$ and $\tilde{B}_t$ given below:
\begin{equation}\label{eq-def-tilde-A-t-RWLRA-2}
A_t \leq \tilde{A}_t := \begin{cases}
0 & \text{if } \|X_t\|_F^2+\|Y_t\|_F^2  \geq \frac{\alpha}{2\lambda}, \\
\frac{4}{a}\left((\sqrt{\alpha}+\frac{(\|X_t\|_F^2+\|Y_t\|_F^2)}{2})(\|X_t\|_F^2+\|Y_t\|_F^2) + \lambda(\|X_t\|_F^2+\|Y_t\|_F^2)\right) & \text{if } \|X_t\|_F^2+\|Y_t\|_F^2 < \frac{\alpha}{2\lambda},
\end{cases}
\end{equation}
\begin{equation}\label{eq-def-tilde-B-t-RWLRA-2} 
B_t \leq \tilde{B}_t := \frac{1}{b} \sqrt{8 \left(\sqrt{\alpha}+\frac{(\|X_t\|_F^2+\|Y_t\|_F^2)}{2}\right)^2 (\|X_t\|_F^2+\|Y_t\|_F^2) + 8\lambda^2 (\|X_t\|_F^2+\|Y_t\|_F^2)}. 
\end{equation}
\end{lemma}

\begin{proof}
First, we note that, for $p_{\eta,\gamma}=\sum_{l=1}^k x_{\eta,l}y_{\gamma,l}$, 
\begin{equation}
|p_{\eta,\gamma}| = \left|\sum_{l=1}^k x_{\eta,l}y_{\gamma,l}\right| \leq \sqrt{\sum_{l=1}^k x_{\eta,l}^2}\sqrt{\sum_{l=1}^k y_{\gamma,l}^2} \leq \frac{\sum_{l=1}^k x_{\eta,l}^2 + \sum_{l=1}^k y_{\gamma,l}^2}{2} \leq \frac{(\|X\|_F^2+\|Y\|_F^2)}{2}.
\end{equation}

From Equation \eqref{eq-rho-f-product-RWLRA-2}, we know that 
\begin{eqnarray*}
&& \max\{0, -\left\langle \nabla \rho(X,Y),  \nabla h_{\eta,\gamma}(X,Y)\right\rangle\} \leq |-8(a_{\eta,\gamma}-p_{\eta,\gamma})p_{\eta,\gamma} + 4 \lambda(\|X\|_F^2+\|Y\|_F^2)| \\
& \leq & 8(|a_{\eta,\gamma}|+|p_{\eta,\gamma}|)|p_{\eta,\gamma}| + 4 \lambda(\|X\|_F^2+\|Y\|_F^2) \\
& \leq & 4(\sqrt{\alpha}+\frac{(\|X\|_F^2+\|Y\|_F^2)}{2})(\|X\|_F^2+\|Y\|_F^2) + 4 \lambda(\|X\|_F^2+\|Y\|_F^2).
\end{eqnarray*}
Thus,
\begin{eqnarray*}
&& \max\{0, -\left\langle \nabla \rho(X,Y),  \nabla h_{\eta,\gamma}(X,Y)\right\rangle\} \\
&& \begin{cases}
=0 &\text{if } \|X\|_F^2+\|Y\|_F^2 \geq \frac{\alpha}{2\lambda}, \\
\leq 4(\sqrt{\alpha}+\frac{(\|X\|_F^2+\|Y\|_F^2)}{2})(\|X\|_F^2+\|Y\|_F^2) + 4 \lambda(\|X\|_F^2+\|Y\|_F^2) & \text{if } \|X\|_F^2+\|Y\|_F^2  < \frac{\alpha}{2\lambda}.
\end{cases}
\end{eqnarray*}
Therefore, for the $A_t$ defined in Algorithm \ref{alg-confined-SGD-RWLRA-2-direct}, we have Inequality \eqref{eq-def-tilde-A-t-RWLRA-2}.

Note that 
\begin{equation}\label{eq-gradient-bound-RWLRA-2}
\tilde{A}_t < \frac{4}{a}\left(\left(\sqrt{\alpha}+\frac{\alpha}{4\lambda}\right)\frac{\alpha}{2\lambda} +  \frac{\alpha}{2} \right).
\end{equation}

By Equation \eqref{eq-hessian-RWLRA-2}, for any $\theta \in [-\Theta,\Theta]$ and $(X,Y) \in  \R^{m\times k} \times \R^{n\times k}$,
\begin{eqnarray*}
0 & \leq & \Hess(\rho \circ R_{(X,Y)})_{\theta \nabla h_{\eta,\gamma}(X,Y)}(\nabla h_{\eta,\gamma}(X,Y),\nabla h_{\eta,\gamma}(X,Y)) =\|\nabla_X h_{\eta,\gamma}(X,Y)\|_F^2 + \|\nabla_Y h_{\eta,\gamma}(X,Y)\|_F^2 \\
& = & \sum_{i=1}^m \sum_{l=1}^k (-2\delta_{\eta,i} (a_{\eta,\gamma}-p_{\eta,\gamma})y_{\gamma,l}+2\lambda x_{i,l})^2 + \sum_{j=1}^n \sum_{l=1}^k(-2\delta_{\gamma,j} (a_{\eta,\gamma}-p_{\eta,\gamma})x_{\eta,l}+2\lambda y_{j,l})^2.
\end{eqnarray*}
But
\begin{eqnarray*}
&& \sum_{i=1}^m \sum_{l=1}^k (-2\delta_{\eta,i} (a_{\eta,\gamma}-p_{\eta,\gamma})y_{\gamma,l}+2\lambda x_{i,l})^2 \leq 2\sum_{i=1}^m \sum_{l=1}^k \left((-2\delta_{\eta,i} (a_{\eta,\gamma}-p_{\eta,\gamma})y_{\gamma,l})^2+(2\lambda x_{i,l})^2 \right)\\
& = & 8 \sum_{i=1}^m \sum_{l=1}^k \left(\delta_{\eta,i} (a_{\eta,\gamma}-p_{\eta,\gamma})^2y_{\gamma,l}^2 + \lambda^2 x_{i,l}^2 \right) = 8 (a_{\eta,\gamma}-p_{\eta,\gamma})^2 \sum_{l=1}^k y_{\gamma,l}^2 + 8\lambda^2 \|X\|_F^2 \\
& \leq & 8 (|a_{\eta,\gamma}|+|p_{\eta,\gamma}|)^2 \|Y\|_F^2 + 8\lambda^2 \|X\|_F^2 \leq  8 \left(\sqrt{\alpha}+\frac{\|X\|_F^2+\|Y\|_F^2 }{2}\right)^2 \|Y\|_F^2 + 8\lambda^2 \|X\|_F^2
\end{eqnarray*}
and, similarly,
\[
\sum_{j=1}^n \sum_{l=1}^k(-2\delta_{\gamma,j} (a_{\eta,\gamma}-p_{\eta,\gamma})x_{\eta,l}+2\lambda y_{j,l})^2 \leq 8 \left(\sqrt{\alpha}+\frac{\|X\|_F^2+\|Y\|_F^2 }{2}\right)^2 \|X\|_F^2 + 8\lambda^2 \|Y\|_F^2.
\]
Altogether, we get that
\begin{eqnarray*}
0 & \leq & \Hess(\rho \circ R_{(X,Y)})_{\theta\nabla h_{\eta,\gamma}(X,Y)}(\nabla h_{\eta,\gamma}(X,Y),\nabla h_{\eta,\gamma}(X,Y))\\
& \leq & 8 \left(\sqrt{\alpha}+\frac{\|X\|_F^2+\|Y\|_F^2 }{2}\right)^2 (\|X\|_F^2+\|Y\|_F^2 ) + 8\lambda^2 (\|X\|_F^2+\|Y\|_F^2 ).
\end{eqnarray*}
Thus, for the $B_t$ defined in Algorithm \ref{alg-confined-SGD-RWLRA-2-direct}, we have Inequality \eqref{eq-def-tilde-B-t-RWLRA-2}.
\end{proof}

With the estimations in Lemma \ref{lemma-A-t-B-t-bound-RWLRA-2}, we have the following corollary.

\begin{corollary}\label{cor-SGD-mfd-RWLRA-2-simplified-vf-t}
Proposition \ref{prop-confined-SGD-RWLRA-2} remains true if we replace the $\vf_t$ in it by 
\begin{equation}\label{eq-def-vf-t-RWLRA-2-simplified-vf-t}
\tilde{\vf}_t := \max\{\tilde{A}_t,\tilde{B}_t,\frac{c_t}{\Theta}, \Phi_{\min}\},
\end{equation}
where $\tilde{A}_t$ and $\tilde{B}_t$ are given in Equations \eqref{eq-def-tilde-A-t-RWLRA-2} and \eqref{eq-def-tilde-B-t-RWLRA-2}.
\end{corollary}

Again, as in Section \ref{sec-WLRA-confined-SGD}, $\tilde{\vf}_t$ is larger than $\vf_t$, but much easier to compute when $mn$ is very large.

Note that
\begin{equation}\label{eq-hessian-bound-RWLRA-2}
\tilde{B}_t \leq \frac{1}{b} \sqrt{8 \left(\sqrt{\alpha}+\frac{\rho_1}{2}\right)^2 \rho_1 + 8\lambda^2 \rho_1} \text{ if } \rho(X_t,Y_t) \leq \rho_1.
\end{equation}

Combining Inequalities \eqref{eq-gradient-bound-RWLRA-2} and \eqref{eq-hessian-bound-RWLRA-2}, one can see that 
\begin{eqnarray*}
 \vf_t \leq \tilde{\vf}_t & \leq & \max \left\{ \frac{4\left(\left(\sqrt{\alpha}+\frac{\alpha}{4\lambda}\right)\frac{\alpha}{2\lambda} +  \frac{\alpha}{2} \right)}{a}, \frac{\sqrt{8 \left(\sqrt{\alpha}+\frac{\rho_1}{2}\right)^2 \rho_1 + 8\lambda^2 \rho_1}}{b}, \frac{c}{\Theta}, \Phi_{\min} \right\} \\
& = & \max \left\{ \frac{4\left(\left(\sqrt{\alpha}+\frac{\alpha}{4\lambda}\right)\frac{\alpha}{2\lambda} +  \frac{\alpha}{2} \right)}{a}, \frac{\sqrt{8 \left(\left(\sqrt{\alpha}+\frac{\rho_0}{2}+ \frac{ca}{2} + \frac{b^2\sigma}{4}\right)^2+ \lambda^2 \right) (\rho_0+ ca + \frac{b^2\sigma}{2})}}{b}, \frac{c}{\Theta}, \Phi_{\min} \right\}.
\end{eqnarray*}
We pick 
\begin{equation}\label{eq-a-b-choices-RWRLA-2}
a = \frac{1}{\lambda} \text{ and } b = \frac{1}{\sqrt{\lambda}},
\end{equation}
and then,
\begin{eqnarray*}
&& \frac{4\left(\left(\sqrt{\alpha}+\frac{\alpha}{4\lambda}\right)\frac{\alpha}{2\lambda} +  \frac{\alpha}{2} \right)}{a} = 2\alpha\sqrt{\alpha} + \frac{\alpha^2}{2\lambda} + 2\lambda \alpha , \\
&& \frac{\sqrt{8 \left(\left(\sqrt{\alpha} + \frac{\rho_0}{2} + \frac{ca}{2} + \frac{b^2\sigma}{4}\right)^2+ \lambda^2 \right) (\rho_0 + ca + \frac{b^2\sigma}{2})}}{b} =  \sqrt{ \left(\left(2\sqrt{\alpha}+\rho_0+\frac{2c + \sigma}{2\lambda}\right)^2+ 4\lambda^2 \right) (2\lambda\rho_0+ 2c + \sigma)}.
\end{eqnarray*}

\begin{corollary}\label{cor-a-b-choices-RWLRA-2}
For the choices of $a$ and $b$ given in Equation \eqref{eq-a-b-choices-RWRLA-2}, we have
\begin{align*}
\vf_t \leq \tilde{\vf}_t
 &\leq  \max \left\{ 2\alpha\sqrt{\alpha} + \frac{\alpha^2}{2\lambda} + 2\lambda \alpha, \sqrt{\left(\left(2\sqrt{\alpha}+\rho_0+\frac{2c + \sigma}{2\lambda}\right)^2+ 4\lambda^2 \right) (2\lambda\rho_0+ 2c + \sigma)}, \frac{c}{\Theta}, \Phi_{\min} \right\} \\
 &= O(\frac{1}{\lambda}) \text{ as } \lambda \rightarrow 0^+.
\end{align*}
\end{corollary}

It seems that $O(\frac{1}{\lambda})$ is the slowest growth of $\vf_t$ and $\tilde{\vf}_t$ as $\lambda \rightarrow 0^+$ that we can guarantee. This implies that, for very small $\lambda>0$, $\vf_t$ and $\tilde{\vf}_t$ may be very large. And, consequently, the actual step sizes $\frac{c_t}{\vf_t}$ and $\frac{c_t}{\tilde{\vf}_t}$ used in the gradient descent may be very small comparing to the preferred step size $c_t$ when $\lambda>0$ is very small.

Again, as in Section \ref{sec-WLRA-confined-SGD}, we can use a constant value for every $\vf_t$.

\begin{corollary}\label{cor-SGD-mfd-RWLRA-2-single-vf-t}
In Algorithm \ref{alg-confined-SGD-RWLRA-2-direct}, choose $a$ and $b$ as given in Equation \eqref{eq-a-b-choices-RWRLA-2}, and $\Theta$ and $\Phi_{\min}$ given by 
\begin{align*}
\Phi_{\min} \geq &\max \left \{ 2\alpha\sqrt{\alpha} + \frac{\alpha^2}{2\lambda} + 2\lambda \alpha, \sqrt{\left(\left(2\sqrt{\alpha}+\rho_0+\frac{2c + \sigma}{2\lambda}\right)^2+ 4\lambda^2 \right) (2\lambda\rho_0+ 2c + \sigma)}\right\}
\end{align*}
and
\[
\Theta = \frac{c}{\Phi_{\min}}.
\]
For these values of $a$, $b$, $\Theta$ and $\Phi_{\min}$, the sequence $\{\vf_t\}_{t=0}^{\infty}$ in Algorithm \ref{alg-confined-SGD-RWLRA-2-direct} is given by the constant value
\begin{equation}\label{eq-def-vf-t-RWLRA-2-single-vf-t}
\vf_t= \Phi_{\min}.
\end{equation}
\end{corollary}

In Section \ref{sec-numerical-results}, we use a special case of Algorithm \ref{alg-confined-SGD-RWLRA-2-direct} as a benchmark. In this case, we fix the sequence $\{c_t\}_{t=0}^\infty$ with $c_t = \frac{1}{t+1}$. Thus, $c = \max \{c_t~\big{|}~t\geq 0\} = 1$ and $\sigma = \sum_{t=0}^\infty c_t^2 = \frac{\pi^2}{6}$. Further, we choose a positive scalar $\Phi_{\min}$ satisfying
\begin{equation}\label{eq-phi-min-choice-RWRLA-2-specific}
\Phi_{\min} = K \max\left \{ 2\alpha\sqrt{\alpha} + \frac{\alpha^2}{2\lambda} + 2\lambda \alpha, \sqrt{\left(\left(2\sqrt{\alpha}+\rho_0+\frac{12 + \pi^2}{12\lambda}\right)^2+ 4\lambda^2 \right) (2\lambda\rho_0+ \frac{12 + \pi^2}{6})}\right\}
\end{equation}
where $\alpha$ is given in Equation \eqref{eq-def-alpha-RWLRA-1}, $\rho_0$ is given in Equation \eqref{eq-def-rho-0-RWLRA-2}, 
and $K \geq 1$ is one of the constant inputs for Algorithm \ref{alg-confined-SGD-RWLRA-1-direct} below. Fix positive scalar $\Theta$ in Algorithm \ref{alg-confined-SGD-RWLRA-2-direct} as
\begin{equation}\label{eq-theta-choice-RWRLA-2-specific}
\Theta = \frac{1}{\Phi_{\min}}.
\end{equation}

Next, we are ready to give the special case of Algorithm \ref{alg-confined-SGD-RWLRA-2-direct}.

\begin{algorithm}[A Special Case of Algorithm \ref{alg-confined-SGD-RWLRA-2-direct}]\label{alg-confined-SGD-RWLRA-2-direct-specific}\

\noindent\makebox[\linewidth]{\rule{\textwidth}{1pt}}

\textbf{Input:} 

\begin{itemize}
    \item[-] the random function $h$ given in Definition \ref{def-random-f-eta-gamma-RWLRA-2}, 
    \item[-] the retraction $R$ given in Equation \eqref{eq-retraction-RWLRA-2},
    \item[-] the positive integers $m, n$ and $k$ given in Problem \ref{prob-RWLRA-2}, 
    \item[-] the positive scalar $\lambda$ given in Problem \ref{prob-RWLRA-2}, 
    \item[-] the matrices $A$ and $W$ given in Problem \ref{prob-RWLRA-2},
    \item[-] a scalar $K \geq 1$, 
    \item[-] an initial iterate $(X_0,Y_0)\in \R^{m\times k}\times \R^{n\times k}$. 
\end{itemize}

\textbf{Output:} A sequence of iterates $\{(X_t,Y_t)\}_{t=0}^\infty  \subset \R^{m\times k} \times \R^{n\times k}$.
\begin{itemize}
	\item \emph{for $t=0,1,2\dots$ do}
	\begin{enumerate}[1.]
        \item Select a random element $(\eta_t, \gamma_t)$ from $\Omega_{m, n}$ with the probability distribution $\mu$ independent of $\{(\eta_{\tau}, \gamma_{\tau})\}_{\tau=0}^{t-1}$.
	\item Set
	\begin{equation}\label{eq-confined-SGD-RWLRA-2-recursion-specific}
	(X_{t+1},Y_{t+1})= R_{(X_t,Y_t)}\left(-\frac{1}{(t+1)\Phi_{\min}}\nabla h_{\eta_t,\gamma_t}(X_t,Y_t)\right),
	\end{equation}
        where $\nabla h_{\eta_t,\gamma_t}(X_t,Y_t)$ is given in Corollary \ref{cor-random-f-gradient-RWLRA-2} and $\Phi_{\min}$ is given in Equation \eqref{eq-phi-min-choice-RWRLA-2-specific}.
	\end{enumerate}
	\item \emph{end for}
\end{itemize}
\noindent\makebox[\linewidth]{\rule{\textwidth}{1pt}}
\end{algorithm}

\begin{proposition}\label{prop-SGD-mfd-RWLRA-2-direct-specific}
Let H be the function given in Equation \eqref{def-function-H-RWLRA-2} and $\{(X_t,Y_t)\}_{t=0}^\infty \subset \R^{m\times k} \times \R^{n\times k}$ be the sequence from the Algorithm \ref{alg-confined-SGD-RWLRA-2-direct-specific}. Then:
\begin{enumerate}
    \item $\{(X_t,Y_t)\}_{t=0}^\infty$ is contained in the compact subset ${\{(X,Y)\in \R^{m\times k} \times \R^{n\times k}~\big{|}~\|X\|_F^2+\|Y\|_F^2\leq \rho_0+\frac{\pi^2 + 12}{12\lambda}\}}$ of $\R^{m\times k} \times \R^{n\times k}$, where $\rho_0$ given in Equation \eqref{eq-def-rho-0-RWLRA-2} and $\lambda$ in Problem \ref{prob-RWLRA-2}. Therefore, $\{(X_t,Y_t)\}_{t=0}^\infty$ has convergent subsequences;
    \item $\{H(X_t,Y_t)\}_{t=0}^\infty$ converges almost surely to a finite number;
    \item $\{\|\nabla H(X_t,Y_t)\|\}_{t=0}^\infty$ converges almost surely to $0$;
    \item any limit point of $\{(X_t,Y_t)\}_{t=0}^\infty$ is almost surely a stationary point of $F$.
\end{enumerate}
\end{proposition}

\begin{proof}
Let $\Phi_{\min}$ be as in Equation \eqref{eq-phi-min-choice-RWRLA-2-specific} and $\Theta$ in Equation \eqref{eq-theta-choice-RWRLA-2-specific}. Proposition \ref{prop-SGD-mfd-RWLRA-2-direct-specific} follows from Corollary \ref{cor-SGD-mfd-RWLRA-2-single-vf-t} as a special case.
\end{proof}

\begin{remark}\label{rk-SGD-mfd-RWLRA-2-single-vf-t}
Proposition \ref{prop-SGD-mfd-RWLRA-2-direct-specific} follow from Corollary \ref{cor-confined-SGD} with
\[
\vf = \Phi_{\min} = K \max\left\{2\alpha\sqrt{\alpha} + \frac{\alpha^2}{2\lambda} + 2\lambda \alpha, \sqrt{\left(\left(2\sqrt{\alpha}+\rho_0+\frac{12 + \pi^2}{12\lambda}\right)^2+ 4\lambda^2 \right) (2\lambda\rho_0+ \frac{12 + \pi^2}{6})}\right\}.
\]
Theoretically, we can hold $K$ constant for different $\lambda$ and still have convergent stochastic gradient descents. Thus, it is possible to make 
\begin{equation}\label{eq-def-vf-t-RWLRA-2-single-vf-t-specific}
\vf  = O(\frac{1}{\lambda}) \text{ as } \lambda \rightarrow 0^+.
\end{equation}
In practice, we choose different $K$ for different $\lambda$ to optimize the performance of Algorithm \ref{alg-confined-SGD-RWLRA-2-direct-specific}. Comparing our estimates of $\vf_t$ in this appendix to those in Section \ref{sec-WLRA-confined-SGD}, it seems that, for small $\lambda>0$, Algorithm \ref{alg-confined-SGD-RWLRA-1-direct} should converge faster than Algorithm \ref{alg-confined-SGD-RWLRA-2-direct} when we hold $K$ constant.
\end{remark}

\subsection{An Accelerated Line Search Algorithm for Problem \ref{prob-RWLRA-2}}\label{subsec-RWLRA-2-ALS}
We apply Algorithm \ref{alg-ALS} to design an accelerated line search algorithm for Problem \ref{prob-RWLRA-2}. To design the algorithm, we first compute the gradient of $H:\R^{m\times k} \times \R^{n\times k} \rightarrow \R$ given in Problem \ref{prob-RWLRA-2}. 

\begin{corollary}\label{cor-F-gradient-RWLRA-2}
Given $H:\R^{m\times k} \times \R^{n\times k} \rightarrow \R$ be as in Problem \ref{prob-RWLRA-2}, define the matrices:
\begin{eqnarray*}
\nabla_X H(X,Y) &=&
\left[\frac{\partial H(X,Y)}{\partial x_{i,l}}\right]_{m \times k}
= -2(W\odot(A - P))Y + 2\lambda X, \\
\nabla_Y H(X,Y) &=&
\left[\frac{\partial H(X,Y)}{\partial y_{j,l}}\right]_{n \times k} = -2((W\odot(A - P))^TX + 2\lambda Y,
\end{eqnarray*}
for $X=[x_{i,j}] \in \R^{m\times k}$ and $Y=[y_{i,j}] \in \R^{n\times k}$, where $\odot$ is the Hadamard product of matrices. The gradient of $H$ at any $(X,Y)\in \R^{m\times k} \times \R^{n\times k}$ is 
\begin{equation}\label{eq-F-gradient-RWLRA-2}
\nabla H(X,Y) = (\nabla_X H(X,Y), \nabla_Y H(X,Y)) \in \R^{m\times k} \times \R^{n\times k}.
\end{equation}
\end{corollary}

\begin{proof}
Note that, $H(X,Y) = \sum_{\eta=1}^m \sum_{\gamma=1}^n w_{\eta,\gamma} h_{\eta,\gamma} (X,Y)$. Therefore,
\begin{eqnarray*}
\frac{\partial H(X,Y)}{\partial x_{i,l}} & = &
\sum_{\eta=1}^m\sum_{\gamma=1}^n w_{\eta,\gamma}\frac{\partial h_{\eta,\gamma}(X,Y)}{\partial x_{i,l}}
= \sum_{j = 1}^n -2w_{i, j}(a_{i, j} - p_{i, j})y_{j, l} + 2\lambda x_{i, l}, \\
\frac{\partial H(X,Y)}{\partial y_{j,l}} & = &
\sum_{\eta=1}^m\sum_{\gamma=1}^n w_{\eta,\gamma}\frac{\partial h_{\eta,\gamma}(X,Y)}{\partial y_{j,l}}
= \sum_{i = 1}^m -2w_{i, j}(a_{i, j} - p_{i, j})x_{i, l} + 2\lambda y_{j, l}.
\end{eqnarray*}
Also, note that,
\begin{eqnarray*}
\left[\frac{\partial H(X,Y)}{\partial x_{i,l}}\right]_{m \times k}
&=& \left[\sum_{j = 1}^n -2w_{i, j}(a_{i, j} - p_{i, j})y_{j, l} + 2\lambda x_{i, l}\right]_{m\times k}
= -2(W\odot(A - P))Y + 2\lambda X, \\
\left[\frac{\partial H(X,Y)}{\partial y_{j,l}}\right]_{n \times k}
&=& \left[\sum_{i = 1}^m -2w_{i, j}(a_{i, j} - p_{i, j})x_{i, l} + 2\lambda y_{j, l}\right]_{n\times k}
=-2((W\odot(A - P))^TX + 2\lambda Y.
\end{eqnarray*}
Then, the gradient of $H$ is given by Equation \eqref{eq-F-gradient-RWLRA-2}.
\end{proof}

Next, the accelerated line search algorithm for Problem \ref{prob-RWLRA-2} is given below.
\begin{algorithm}[An Accelerated Line Search for Problem \ref{prob-RWLRA-2}]\label{alg-ALS-RWLRA-2}\

\noindent\makebox[\linewidth]{\rule{\textwidth}{1pt}}

\textbf{Input:} 
    \begin{itemize}
    \item[-] the function $H$ given in Equation \eqref{def-function-H-RWLRA-2},
    \item[-] the retraction $R$ in Equation \eqref{eq-retraction-RWLRA-2},
    \item[-] the positive integers $m, n$ and $k$ given in Problem \ref{prob-RWLRA-2}, 
    \item[-] the positive scalar $\lambda$ given in Problem \ref{prob-RWLRA-2}, 
    \item[-] the matrices $A$ and $W$ given in Problem \ref{prob-RWLRA-2},
    \item[-] scalars $\overline{\alpha}>0$ and $\beta,\iota \in (0,1)$,
    \item[-] an initial iterate $(X_0,Y_0)  \in \R^{m\times k} \times \R^{n\times k}$.
    \end{itemize}

\textbf{Output:} A sequence of iterates $\{(X_t,Y_t)\}_{t=0}^\infty \subset \R^{m\times k} \times \R^{n\times k}$.
\begin{itemize}
	\item \emph{for $t=0,1,2\dots$ do}
	\begin{enumerate}[1.]
	\item Select $(X_{t+1},Y_{t+1})$ so that
	\begin{equation}\label{eq-ALS-RWLRA-2-recursion}
	H((X_t,Y_t))-H(X_{t+1},Y_{t+1}) \geq H(X_t,Y_t)-H(R_{(X_t,Y_t)}(-\tau_t^A\nabla H(X_t, Y_t))),
	\end{equation}
        where $\nabla H(X_t, Y_t)$ is given in Corollary \ref{cor-F-gradient-RWLRA-2} and $\tau_t^A$ is the Armijo step size for the given $\overline{\alpha},\beta,\iota$ and $-\nabla H(X_t, Y_t)$ (See Definition \ref{def-Armijo}).
	\end{enumerate}
	\item \emph{end for}
\end{itemize}
\noindent\makebox[\linewidth]{\rule{\textwidth}{1pt}}
\end{algorithm}

\begin{proposition}\label{prop-ALS-RWLRA-2}
Let $\{(X_t,Y_t)\}_{t=0}^\infty$ be the sequence constructed by Algorithm \ref{alg-ALS-RWLRA-2}. Then ${\lim_{t \rightarrow \infty} \|\nabla H(X_t,Y_t)\|=0}$ and $\{(X_t,Y_t)\}_{t=0}^\infty$ has a subsequence converging to a critical point of $H$.
\end{proposition}

\begin{proof}
Note that, with $(X_0,Y_0)$ given in Algorithm \ref{alg-ALS-RWLRA-2},
\[
H(X_0,Y_0) \geq H(X,Y)
\geq \lambda (\|X\|_F^2 + \|Y\|_F^2),
\]
which implies that $\{(X,Y) \in \R^{m\times k} \times \R^{n\times k} ~|~ H(X,Y) \leq H(X_0,Y_0)\}$ is a compact subset of \\ $\Big\{X \in \R^{m\times k} ~|~ \|X\|_F \leq \sqrt{\frac{H(X_0,Y_0)}{\lambda}}\Big\} \times \Big\{Y \in \R^{n\times k} ~|~ \|Y\|_F \leq \sqrt{\frac{H(X_0,Y_0)}{\lambda}}\Big\}$. Since the sequence $\{-\nabla H(X_t,Y_t)\}_{t=0}^\infty$ is gradient related to $H$, we conclude that Proposition \ref{prop-ALS-RWLRA-2} follows from Theorem \ref{thm-ALS-converges}.
\end{proof}

\section{Weighted Low-Rank Approximation with Positive Weights}\label{sec-WLRA-ALS}
If all entries of $W$ in Problem \ref{prob-WLRA} are positive, we say the problem has positive weights. In this case, the confinement exists even without any regularization. Problem \ref{prob-WLRA-PW} below is this special case of Problem \ref{prob-WLRA}.

\begin{problem}[The Weighted Low-Rank Approximation Problem with Positive Weights]\label{prob-WLRA-PW}
Under Assumptions (\ref{problem-assumption-1})-(\ref{problem-assumption-3}) in Problem \ref{prob-WLRA}, further assume that $W=[w_{i,j}] \in \R^{m\times n}$ satisfies $w_{i,j} > 0$ for every $(i,j) \in \{1,2,\dots,m\}\times \{1,2,\dots,n\}$. Let $\hat{F}:\R^{m\times n} \rightarrow \R$ be as in Problem \ref{prob-WLRA}. Solve for
\[
\mathrm{argmin}\{\hat{F}(P) ~|~ P\in \R^{m\times n},~  \rank P \leq k\}.
\]
\end{problem}

Similar to Problem \ref{prob-RWLRA-1}, we reformulate Problem \ref{prob-WLRA-PW} to the unconstrained Problem \ref{prob-WLRA-PW-reform} on the manifold $V_k(\R^m)\times \R^k \times V_k(\R^n)$ by Lemma \ref{lemma-reduced-svd}. Lemma \ref{lemma-reduced-svd} also indicates that any solution to Problem \ref{prob-WLRA-PW} is equivalent to a solution to Problem \ref{prob-WLRA-PW-reform}, and vice versa. 

\begin{problem}[A Reformulated Weighted Low-Rank Approximation Problem with Positive Weights]\label{prob-WLRA-PW-reform}
Under Assumptions (\ref{problem-assumption-1})-(\ref{problem-assumption-3}) in Problem \ref{prob-WLRA}, further assume that $w_{i,j} > 0$ for every $(i,j) \in \{1,2,\dots,m\}\times \{1,2,\dots,n\}$. Let $\hat{F}:\R^{m\times n} \rightarrow \R$ be as in Problem \ref{prob-WLRA}. Solve for
\[
\mathrm{argmin}\{\hat{F}(U D^{k\times k}_k(\mathbf{x}) V^T) ~|~ (U,\mathbf{x},V)  \in V_k(\R^m)\times \R^k \times V_k(\R^n) \}.
\]
\end{problem}

\subsection{A Stochastic Gradient Descent for Problem \ref{prob-WLRA-PW-reform}}\label{subsec-WLRA-PW-confined-SGD}
Similar to Section \ref{sec-WLRA-confined-SGD}, to design a convergent stochastic gradient descent for Problem \ref{prob-WLRA-PW-reform}, we still work on the probability space $\Omega_{m,n}$ with the probability distribution $\mu$ given in Definition \ref{def-Omega-m-n-measure}. We first define a random function satisfying Assumption (\ref{assumption-random-function}) of Theorem \ref{thm-confined-SGD}. 

\begin{definition}\label{def-g-eta-gamma-WLRA-PW}
For matrix of weights $W=[w_{i,j}] \in \R^{m\times n}$ given in Problem \ref{prob-WLRA-PW-reform}, define
\begin{equation}\label{eq-def-w-0}
w_0 := \min\{w_{i,j}~\big{|}~ (i,j) \in \Omega_{m,n}\}>0.
\end{equation}
Let $\lambda$ be a constant satisfying
\begin{equation}\label{eq-def-lambda-PW}
0 < \lambda < w_0.
\end{equation}
We will fix the value of $\lambda$ later. With $\lambda$ satisfying Inequality \eqref{eq-def-lambda-PW}, define $\bar{f}:\R^{m\times n}\times \Omega_{m,n} \rightarrow \R$ and $\tilde{f}:\R^{m\times n}\times \Omega_{m,n} \rightarrow \R$ by 
\begin{eqnarray}
\label{eq-def-bar-f-eta-gamma-WLRA-PW} \bar{f} (P; \eta,\gamma): &=& (a_{\eta,\gamma}-p_{\eta,\gamma})^2 -\frac{\lambda}{w_{\eta,\gamma}}p_{\eta,\gamma}^2 \hspace{1pc} \text{and} \hspace{1pc} \\
\label{eq-def-f-eta-gamma-WLRA-PW} \tilde{f}(P; \eta,\gamma): &=& (a_{\eta,\gamma}-p_{\eta,\gamma})^2 -\frac{\lambda}{w_{\eta,\gamma}}p_{\eta,\gamma}^2 + \lambda \|P\|_F^2, 
\end{eqnarray}
for all $P=[p_{i,j}]\in  \R^{m\times n}$ and $(\eta,\gamma)\in \Omega_{m,n}$. Define $\tilde{g}:V_k(\R^m)\times \R^k \times V_k(\R^n)\times \Omega_{m,n}\rightarrow \R$
\begin{equation}
\label{eq-def-g-eta-gamma-WLRA-PW} \tilde{g}(U,\mathbf{x},V; \eta,\gamma) = \tilde{f}_{\eta,\gamma} (U D^{k\times k}_k(\mathbf{x}) V^T)
= \bar{f}_{\eta,\gamma} (U D^{k\times k}_k(\mathbf{x}) V^T) + \lambda \|\mathbf{x}\|^2
\end{equation}
for all $(U,\mathbf{x},V) \in V_k(\R^m)\times \R^k \times V_k(\R^n)$ and $(\eta,\gamma)\in \Omega_{m,n}$, where $D^{k\times k}_k(\mathbf{x})$ is defined in Equation \eqref{eq-def-D-k-by-k}. For notational convenience, we define the functions $\bar{f}_{\eta, \gamma}:\R^{m\times n}\rightarrow \R$, $\tilde{f}_{\eta, \gamma}:\R^{m\times n}\rightarrow \R$ and ${\tilde{g}_{\eta, \gamma}:V_k(\R^m)\times \R^k \times V_k(\R^n) \rightarrow \R}$ for $(\eta, \gamma)\in \Omega_{m,n}$ by
\[
\bar{f}_{\eta, \gamma}(P) = \bar{f}(P; \eta, \gamma), \
\tilde{f}_{\eta, \gamma}(P) = \tilde{f}(P; \eta, \gamma) \ \text{and }
\tilde{g}_{\eta, \gamma}(U,\mathbf{x},V) = \tilde{g}(U,\mathbf{x},V; \eta, \gamma),
\]
for all $P=[p_{i,j}]\in \R^{m\times n}$ and $(U,\mathbf{x},V) \in V_k(\R^m)\times \R^k \times V_k(\R^n)$.
\end{definition}

\begin{definition}\label{def-func-hat-G}
For $\hat{F}:\R^{m\times n} \rightarrow \R$ given in Problem \ref{prob-WLRA-PW-reform}, define $\hat{G}:V_k(\R^m)\times \R^k \times V_k(\R^n)\rightarrow \R$ by
\begin{equation}\label{eq-def-g-WLRA-PW}
\hat{G}(U,\mathbf{x},V) = \hat{F}(U D^{k\times k}_k(\mathbf{x}) V^T)
\end{equation}
for all $(U,\mathbf{x},V) \in V_k(\R^m)\times \R^k \times V_k(\R^n)$. 
\end{definition}

\begin{lemma}\label{lemma-g-eta-gamma-WLRA-PW-expectation}
Let $\hat{F}: \R^{m\times n} \rightarrow \R$ be as in Problem \ref{prob-WLRA-PW-reform} and $\hat{G}: V_k(\R^m)\times \R^k \times V_k(\R^n) \rightarrow \R$ in Definition \ref{def-func-hat-G}. For $P\in \R^{m\times n}$ and $(U,\mathbf{x},V) \in V_k(\R^m)\times \R^k \times V_k(\R^n)$, we have
\begin{equation*}
E_{(\eta,\gamma) \sim \mu}(\tilde{f}(P; \eta,\gamma)) = \hat{F}(P) \text{ and }
E_{(\eta,\gamma) \sim \mu}(\tilde{g}(U,\mathbf{x},V; \eta,\gamma)) = \hat{G} (U,\mathbf{x},V),
\end{equation*}
where the expectations are taken over the probability space $\Omega_{m,n}$ with respect to the probability distribution $\mu$ given in Definition \ref{def-Omega-m-n-measure}. Therefore, $\tilde{f}$ and $\tilde{g}$ are the random functions with expectations $\hat{F}$ and $\hat{G}$.
\end{lemma}
\begin{proof}
\begin{eqnarray*}
E_{(\eta,\gamma) \sim \mu}(\tilde{f}(P; \eta,\gamma)) & = & \sum_{i=1}^m\sum_{j=1}^n w_{i,j} \tilde{f}(P; i,j) = \sum_{i=1}^m\sum_{j=1}^n w_{i,j}\left((a_{i,j}-p_{i,j})^2 -\frac{\lambda}{w_{i,j}}p_{i,j}^2 + \lambda \|P\|_F^2\right) \\
& = & \sum_{i=1}^m\sum_{j=1}^n w_{i,j}(a_{i,j}-p_{i,j})^2 - \sum_{i=1}^m\sum_{j=1}^n \lambda p_{i,j}^2 + \sum_{i=1}^m\sum_{j=1}^n w_{i,j} \lambda \|P\|_F^2 \\
& = & \hat{F}(P) -\lambda \|P\|_F^2 +\lambda \|P\|_F^2 = \hat{F}(P).
\end{eqnarray*}
And
\[
E_{(\eta,\gamma) \sim \mu}(\tilde{g}(U,\mathbf{x},V; \eta,\gamma)) = E_{(\eta,\gamma) \sim \mu}(\tilde{f}(U D^{k\times k}_k(\mathbf{x}) V^T; \eta,\gamma)) = \hat{F}(U D^{k\times k}_k(\mathbf{x}) V^T) = \hat{G}(U,\mathbf{x},V).
\]
\end{proof}

Next, we compute the gradient of the function $\tilde{g}_{\eta,\gamma}: V_k(\R^m)\times \R^k \times V_k(\R^n)\rightarrow \R$. To to that, we need some preparations below.

\begin{remark}
For $(\eta,\gamma)\in \Omega_{m,n}$, view the function $\bar{f}_{\eta,\gamma}$ given in Definition \ref{def-g-eta-gamma-WLRA-PW} as a composite function of $(U,\mathbf{x},V)$ via $\bar{f}_{\eta,\gamma} = \bar{f}_{\eta,\gamma}(P)$ and $P = U D^{k\times k}_k(\mathbf{x}) V^T$, where $P=[p_{i,j}] \in \R^{m\times n}$, $U=[u_{i,j}] \in \R^{m\times k}$, $V=[v_{i,j}] \in \R^{n\times k}$ and $\mathbf{x}=[x_1,\dots,x_k]^T\in \R^k$.

Note that, for $(\eta,\gamma)\in \Omega_{m,n}$,
\[
\frac{\partial \bar{f}_{\eta,\gamma}}{\partial p_{\eta,\gamma}} = -2(a_{\eta,\gamma} - \left(1-\frac{\lambda}{w_{\eta,\gamma}})p_{\eta,\gamma}\right).
\]
By Equations \eqref{eq-p-uxv}, \eqref{eq-partial-p-uxv} and the Chain Rule, we have that, for $(\eta,\gamma)\in \Omega_{m,n}$, $l=1,\dots,k$, $i=1,\dots,m$ and $j=1,\dots,n$, 
\begin{eqnarray}
\frac{\partial \bar{f}_{\eta,\gamma}}{\partial u_{i,l}}& = & -2\delta_{\eta,i} \left(a_{\eta,\gamma}-(1-\frac{\lambda}{w_{\eta,\gamma}})p_{\eta,\gamma}\right)x_l v_{\gamma,l}, \label{eq-partial-dev-u-PW} \\
\frac{\partial \bar{f}_{\eta,\gamma}}{\partial v_{j,l}} & = & -2\delta_{\gamma,j} \left(a_{\eta,\gamma}-(1-\frac{\lambda}{w_{\eta,\gamma}})p_{\eta,\gamma}\right)x_l u_{\eta,l}, \label{eq-partial-dev-v-PW} \\
\frac{\partial \bar{f}_{\eta,\gamma}}{\partial x_{l}} & = & -2 \left(a_{\eta,\gamma}-(1-\frac{\lambda}{w_{\eta,\gamma}})p_{\eta,\gamma}\right)u_{\eta,l} v_{\gamma,l}. \label{eq-partial-dev-x-PW} \\
\end{eqnarray}
\end{remark}

With the preparations above, we are ready to give the gradient of $\tilde{g}$ in Corollary \ref{cor-random-hat-g-gradient-WLRA-PW} below by Lemmas \ref{lemma-stiefel} and \ref{lemma-gradient-submfd}.

\begin{corollary}\label{cor-random-hat-g-gradient-WLRA-PW}
Given $(\eta,\gamma)\in \Omega_{m,n}$, define the matrices
\begin{eqnarray*}
 \nabla_U \bar{f}_{\eta,\gamma} & = & \left[\frac{\partial \bar{f}_{\eta,\gamma}}{\partial u_{i,l}}\right]_{m\times k}= \left[-2\delta_{\eta,i} \left(a_{\eta,\gamma}-(1-\frac{\lambda}{w_{\eta,\gamma}})p_{\eta,\gamma}\right)x_l v_{\gamma,l}\right]_{m\times k}, \\
 \nabla_V \bar{f}_{\eta,\gamma} & = & \left[\frac{\partial \bar{f}_{\eta,\gamma}}{\partial v_{j,l}}\right]_{n\times k} =  \left[-2\delta_{\gamma,j} \left(a_{\eta,\gamma}-(1-\frac{\lambda}{w_{\eta,\gamma}})p_{\eta,\gamma}\right)x_l u_{\eta,l} \right]_{n\times k}, \\
 \nabla_{\mathbf{x}} \bar{f}_{\eta,\gamma} & = & \left[\begin{array}{c}
 \frac{\partial \bar{f}_{\eta,\gamma}}{\partial x_{1}} \\
 \frac{\partial \bar{f}_{\eta,\gamma}}{\partial x_{2}} \\
 \vdots \\
 \frac{\partial \bar{f}_{\eta,\gamma}}{\partial x_{k}}
 \end{array}\right]
 = \left[\begin{array}{c}
 -2 \left(a_{\eta,\gamma}-(1-\frac{\lambda}{w_{\eta,\gamma}})p_{\eta,\gamma}\right)u_{\eta,1} v_{\gamma,1} \\
 -2 \left(a_{\eta,\gamma}-(1-\frac{\lambda}{w_{\eta,\gamma}})p_{\eta,\gamma}\right)u_{\eta,2} v_{\gamma,2} \\
 \vdots \\
 -2 \left(a_{\eta,\gamma}-(1-\frac{\lambda}{w_{\eta,\gamma}})p_{\eta,\gamma}\right)u_{\eta,k} v_{\gamma,k}
 \end{array}\right].
\end{eqnarray*}
Then the gradient of $\tilde{g}_{\eta,\gamma}$ at any $(U,\mathbf{x},V) \in V_k(\R^m)\times \R^k \times V_k(\R^n)$ is 
\begin{equation}\label{eq-g-gradient-WLRA-PW}
\nabla \tilde{g}_{\eta,\gamma}(U,\mathbf{x},V) = (\Pi_U(\nabla_U \bar{f}_{\eta,\gamma}), \nabla_{\mathbf{x}} \bar{f}_{\eta,\gamma} + 2\lambda \mathbf{x}, \Pi_V(\nabla_V \bar{f}_{\eta,\gamma})) \in T_U V_k(\R^m)\times \R^k \times T_VV_k(\R^n),
\end{equation}
where $\Pi_U$ and $\Pi_V$ are the orthogonal projections given in Lemma \ref{lemma-stiefel-orthogonal-proj}.
\end{corollary}
\begin{proof}
The matrices for $\nabla_U \bar{f}_{\eta,\gamma}$, $\nabla_V \bar{f}_{\eta,\gamma}$ and $\nabla_{\mathbf{x}} \bar{f}_{\eta,\gamma}$ follow the Equations \eqref{eq-partial-dev-u-PW}, \eqref{eq-partial-dev-v-PW} and \eqref{eq-partial-dev-x-PW}. Then, by Lemmas \ref{lemma-stiefel} and \ref{lemma-gradient-submfd}, the gradient of $\tilde{g}_{\eta,\gamma}$ is given by Equation \eqref{eq-g-gradient-WLRA-PW}.
\end{proof}

The function $\rho : V_k(\R^m)\times \R^k \times V_k(\R^n)\rightarrow \R$ given in Definition \ref{def-confinement-RWLRA-1} is also a confinement function for the random function $\tilde{g}$ given by Equation \eqref{eq-def-g-eta-gamma-WLRA-PW}.

\begin{lemma}\label{lemma-rho-WLRA-PW-confinement}
Define
\begin{equation}\label{eq-rho-0-WLRA-PW}
\rho_0 = \max\{\|\mathbf{x}_0\|^2, \frac{\alpha}{4\lambda(1-\frac{\lambda}{w_0})}\}
\end{equation}
where $\alpha$ is given in Equation \eqref{eq-def-alpha-RWLRA-1}, $\lambda$ satisfies Inequality \eqref{eq-def-lambda-PW}, $w_0$ is given in Equation \eqref{eq-def-w-0}, and $\mathbf{x}_0$ is given by the initial iterate $(U_0,\mathbf{x}_0,V_0)$ in Algorithm \ref{alg-confined-SGD-WLRA-PW}. With $\rho_0$ given in Equation \eqref{eq-rho-0-WLRA-PW}, the function $\rho$ in Definition \ref{def-confinement-RWLRA-1} and the function $\tilde{g}_{\eta, \gamma}$ in Definition \ref{def-g-eta-gamma-WLRA-PW} satisfy that
\begin{itemize}
    \item $\rho(U_0,\mathbf{x}_0,V_0) \leq \rho_0$,
    \item $\langle \nabla \rho(U,\mathbf{x},V), \nabla \tilde{g}_{\eta, \gamma}(U,\mathbf{x},V) \rangle \geq 0$ for $(\eta,\gamma) \in \Omega_{m,n}$ and $(U,\mathbf{x},V)\in V_k(\R^m)\times \R^k \times V_k(\R^n)$ satisfying $\rho(U,\mathbf{x},V) \geq \rho_0$. 
\end{itemize}
That is, the function $\rho : V_k(\R^m)\times \R^k \times V_k(\R^n)\rightarrow \R$ in Definition \ref{def-confinement-RWLRA-1} is a confinement of the random function $\tilde{g}$ in Definition \ref{def-g-eta-gamma-WLRA-PW} on the manifold $V_k(\R^m)\times \R^k \times V_k(\R^n)$, and $\rho_0$ satisfies Assumptions (\ref{assumption-confinement-function}) and (\ref{assumption-initial}) of Theorem \ref{thm-confined-SGD}.
\end{lemma}

\begin{proof}
By Equation \eqref{eq-nabla-rho-RWLRA-1}, the gradient of $\rho(U,\mathbf{x},V)=\|\mathbf{x}\|^2$ is $\nabla \rho = (0,2\mathbf{x},0) \in T_U V_k(\R^m)\times \R^k \times T_VV_k(\R^n)$. By Equations \eqref{eq-norm-equal} and \eqref{eq-g-gradient-WLRA-PW}, we have that
\begin{eqnarray*}
&&\left\langle \nabla \rho(U,\mathbf{x},V), \nabla \tilde{g}_{\eta,\gamma}(U,\mathbf{x},V)\right\rangle \\
& = & 2\mathbf{x}^T  \left[\begin{array}{c}
-2 \left(a_{\eta,\gamma}-(1-\frac{\lambda}{w_{\eta,\gamma}})p_{\eta,\gamma}\right)u_{\eta,1} v_{\gamma,1} + 2\lambda x_1 \\
-2 \left(a_{\eta,\gamma}-(1-\frac{\lambda}{w_{\eta,\gamma}})p_{\eta,\gamma}\right)u_{\eta,2} v_{\gamma,2} + 2\lambda x_2 \\
\vdots \\
-2 \left(a_{\eta,\gamma}-(1-\frac{\lambda}{w_{\eta,\gamma}})p_{\eta,\gamma}\right)u_{\eta,k} v_{\gamma,k} + 2\lambda x_k
\end{array}\right] \\
& = & -4\sum_{l=1}^k \left(a_{\eta,\gamma}-(1-\frac{\lambda}{w_{\eta,\gamma}})p_{\eta,\gamma}\right)u_{\eta,l} v_{\gamma,l}x_l + 4\lambda\sum_{l=1}^k x_l^2 \\
& = & -4(a_{\eta,\gamma}-(1-\frac{\lambda}{w_{\eta,\gamma}})p_{\eta,\gamma})p_{\eta,\gamma} + 4\lambda \|\mathbf{x}\|^2 \\
& \geq & -\frac{a_{\eta,\gamma}^2}{1-\frac{\lambda}{w_{\eta,\gamma}}} + 4\lambda \|\mathbf{x}\|^2 \geq -\frac{a_{\eta,\gamma}^2}{1-\frac{\lambda}{w_0}} + 4\lambda \|\mathbf{x}\|^2 \geq -\frac{\alpha}{1-\frac{\lambda}{w_0}} + 4\lambda \|\mathbf{x}\|^2.
\end{eqnarray*}
So $\left\langle \nabla \rho(U,\mathbf{x},V), \nabla \tilde{g}_{\eta,\gamma}(U,\mathbf{x},V)\right\rangle \geq 0$ if $\|\mathbf{x}\|^2\geq \rho_0$. Therefore $\left\langle \nabla \rho(U,\mathbf{x},V), \nabla \tilde{g}_{\eta,\gamma}(U,\mathbf{x},V)\right\rangle \geq 0$ if \\ ${\rho(U,\mathbf{x},V)\geq \rho_0 \geq \frac{\alpha}{4\lambda(1-\frac{\lambda}{w_0})}}$.
\end{proof}

Note that, from the proof of Lemma \ref{lemma-rho-WLRA-PW-confinement}, we have that
\begin{equation}\label{eq-rho-g-product-WLRA-PW}
\left\langle \nabla \rho(U,\mathbf{x},V), \nabla \tilde{g}_{\eta,\gamma}(U,\mathbf{x},V)\right\rangle = -4(a_{\eta,\gamma}-(1-\frac{\lambda}{w_{\eta,\gamma}})p_{\eta,\gamma})p_{\eta,\gamma} + 4\lambda \|\mathbf{x}\|^2 .
\end{equation}
Next, we give the Hessian of the map $\rho \circ R$.

\begin{lemma}\label{lemma-hessian-WLRA-PW}
Let $\Theta$ be any positive scalar, $\rho$ be the confinement function in Definition \ref{def-confinement-RWLRA-1} and $R$ be the retraction in Definition \ref{def-retraction-GS-prod}. For any $(U,\mathbf{x},V)  \in V_k(\R^m)\times \R^k \times V_k(\R^n)$ and any $\theta \in [-\Theta,\Theta]$, 
\begin{eqnarray}
\label{eq-hessian-WLRA-PW} && \Hess(\rho \circ R_{(U,\mathbf{x},V)})|_{\theta\nabla \tilde{g}_{\eta,\gamma}(U,\mathbf{x},V)}(\nabla \tilde{g}_{\eta,\gamma}(U,\mathbf{x},V) ,\nabla \tilde{g}_{\eta,\gamma}(U,\mathbf{x},V)) \\
\nonumber & = & 8\sum_{l=1}^k (- (a_{\eta,\gamma}-(1-\frac{\lambda}{w_{\eta,\gamma}})p_{\eta,\gamma})u_{\eta,l} v_{\gamma,l} + \lambda x_l)^2.
\end{eqnarray}
\end{lemma}

\begin{proof}
Note that $(\rho \circ R_{(U,\mathbf{x},V)})(Y,\hat{\mathbf{x}},Z)=\|\mathbf{x}+\hat{\mathbf{x}}\|^2$ for any $(U,\mathbf{x},V)  \in V_k(\R^m)\times \R^k \times V_k(\R^n)$ and any $(Y,\hat{\mathbf{x}},Z) \in  T_{U}V_k(\R^m)\times \R^k \times T_{V}V_k(\R^n)$. So $\Hess(\rho \circ R_{(U,\mathbf{x},V)})|_{(Y',\hat{\mathbf{x}}',Z')}((Y,\hat{\mathbf{x}},Z),(Y,\hat{\mathbf{x}},Z))=2\|\hat{\mathbf{x}}\|^2$ for any $(U,\mathbf{x},V)  \in V_k(\R^m)\times \R^k \times V_k(\R^n)$ and any $(Y,\hat{\mathbf{x}},Z), ~(Y',\hat{\mathbf{x}}',Z') \in  T_{U}V_k(\R^m)\times \R^k \times T_{V}V_k(\R^n)$. In particular, it is independent of $(Y',\hat{\mathbf{x}}',Z')$. Thus, for any $(U,\mathbf{x},V)  \in V_k(\R^m)\times \R^k \times V_k(\R^n)$ and any $\theta \in [-\Theta,\Theta]$, we have that $\Hess(\rho \circ R_{(U,\mathbf{x},V)})|_{\theta\nabla \tilde{g}_{\eta,\gamma}(U,\mathbf{x},V)}(\nabla \tilde{g}_{\eta,\gamma}(U,\mathbf{x},V) ,\nabla \tilde{g}_{\eta,\gamma}(U,\mathbf{x},V))=2\|\nabla_{\mathbf{x}} \tilde{f}_{\eta,\gamma} + 2\lambda\mathbf{x}\|^2$. Combining this with Corollary \ref{cor-random-hat-g-gradient-WLRA-PW}, we get Equation \eqref{eq-hessian-WLRA-PW}.
\end{proof}

\begin{algorithm}[A Stochastic Gradient Descent for Problem \ref{prob-WLRA-PW-reform}]\label{alg-confined-SGD-WLRA-PW}\

\noindent\makebox[\linewidth]{\rule{\textwidth}{1pt}}

\textbf{Input:} 
\begin{itemize}
    \item[-] the random function $\tilde{g}$ given in Definition \ref{def-g-eta-gamma-WLRA-PW}, 
    \item[-] the retraction $R$ given in Definition \ref{def-retraction-GS-prod}, 
    \item[-] the positive integers $m, n$ and $k$ given in Problem \ref{prob-WLRA-PW-reform},
    \item[-] the matrices $A$ and $W$ given in Problem \ref{prob-WLRA-PW-reform}, 
    \item[-] positive scalars $a$, $b$, $\Theta$ and $\Phi_{\min}$, 
    \item[-] a sequence $\{c_t\}_{t=0}^\infty$ of positive numbers satisfying $\sum_{t=0}^\infty c_t =\infty$ and $\sum_{t=0}^\infty c_t^2 <\infty$, 
    \item[-] an initial iterate $(U_0,\mathbf{x}_0,V_0) \in V_k(\R^m)\times \R^k \times V_k(\R^n)$.
\end{itemize}

\textbf{Output:} A sequence of iterates $\{(U_t,\mathbf{x}_t,V_t)\}_{t=0}^\infty \subset V_k(\R^m)\times \R^k \times V_k(\R^n)$.
\begin{itemize}
	\item \emph{for $t=0,1,2\dots$ do}
	\begin{enumerate}[1.]
        \item Select a random element $(\eta_t, \gamma_t)$ from $\Omega_{m, n}$ with the probability measure $\mu$ independent of $\{(\eta_{\tau}, \gamma_{\tau})\}_{\tau=1}^{t-1}$.
        \item Define $A_t$ and $B_t$ by
        \begin{align}
        A_t: &= \frac{1}{a}\max\left\{\max\{0, ~4(a_{\eta,\gamma}-(1-\frac{\lambda}{w_{\eta,\gamma}})p_{t,\eta,\gamma})p_{t,\eta,\gamma} - 4\lambda \|\mathbf{x}_t\|^2\}~\big{|}~(\eta,\gamma) \in \Omega_{m\times n}\right\}, \label{eq-vf-bound-a-t-WLRA-PW} \\
        B_t: &= \frac{1}{b} \max\left\{ \sqrt{8\sum_{l=1}^k (-(a_{\eta,\gamma}-(1-\frac{\lambda}{w_{\eta,\gamma}})p_{t,\eta,\gamma})u_{t,\eta,l} v_{t,\gamma,l} + \lambda x_{t,l})^2}~\big{|}~ (\eta,\gamma) \in \Omega_{m\times n}\right\}, \label{eq-vf-bound-b-t-WLRA-PW}
        \end{align}
        where $U_t = [u_{t,i,j}], ~V_t = [v_{t,i,j}]$ and $P_t = [p_{t,i,j}]:=  U_t D^{k\times k}_k(\mathbf{x}_t) V_t^T \in \R^{m\times n}$.
        \item Define random positive number $\vf_t$ by
        \begin{equation}\label{eq-def-vf-t-WLRA-PW}
        \vf_t := \max\{A_t,B_t,\frac{c_t}{\Theta}, \Phi_{\min}\}.
        \end{equation}
	\item Set
	\begin{equation}\label{eq-confined-SGD-recursion-WLRA-PW}
	(U_{t+1},\mathbf{x}_{t+1},V_{t+1})= R_{(U_t,\mathbf{x}_t,V_t)}(-\frac{c_t}{\vf_t}\nabla \tilde{g}_{\eta_t,\gamma_t}(U_t,\mathbf{x}_t,V_t)),
	\end{equation}
        where $\nabla \tilde{g}_{\eta_t,\gamma_t}(U_t,\mathbf{x}_t,V_t)$ is given in Corollary \ref{cor-random-hat-g-gradient-WLRA-PW}.
	\end{enumerate}
	\item \emph{end for}
\end{itemize}
\noindent\makebox[\linewidth]{\rule{\textwidth}{1pt}}
\end{algorithm}

\begin{proposition}\label{prop-confined-SGD-WLRA-PW}
Let $\hat{G}$ be the function given in Equation \eqref{eq-def-g-WLRA-PW} and ${\{(U_t,\mathbf{x}_t,V_t)\}_{t=0}^\infty \subset V_k(\R^m)\times \R^k \times V_k(\R^n)}$ be the sequence from the Algorithm \ref{alg-confined-SGD-WLRA-PW}. Then:
\begin{enumerate}
    \item $\{(U_t,\mathbf{x}_t,V_t)\}_{t=0}^\infty$ is contained in the compact subset \\ ${\{(U,\mathbf{x},V)\in V_k(\R^m)\times \R^k \times V_k(\R^n) ~\big{|}~\|\mathbf{x}\|^2\leq \rho_0 + ca + \frac{b^2\sigma}{2}\}}$ of $V_k(\R^m)\times \R^k \times V_k(\R^n)$ where $\rho_0$ given in Equation \eqref{eq-def-rho-0-RWLRA-1}, $a$ and $b$ are given in Algorithm \ref{alg-confined-SGD-WLRA-PW}, and $c= \max \{c_t~\big{|}~t\geq 0\}$ and $\sigma=\sum_{t=0}^\infty c_t^2$ with the sequence $\{c_t\}_{t=0}^\infty$ given in Algorithm \ref{alg-confined-SGD-WLRA-PW}. Therefore, $\{(U_t,\mathbf{x}_t,V_t)\}_{t=0}^\infty$ has convergent subsequences;
    \item $\{\hat{G}(U_t,\mathbf{x}_t,V_t)\}_{t=0}^\infty$ converges almost surely to a finite number;
    \item $\{\|\nabla \hat{G}(U_t,\mathbf{x}_t,V_t)\|\}_{t=0}^\infty$ converges almost surely to $0$;
    \item any limit point of $\{(U_t,\mathbf{x}_t,V_t)\}_{t=0}^\infty$ is almost surely a stationary point of $\hat{G}$.
\end{enumerate}
\end{proposition}

\begin{proof}
Based on the preceding discussion in this appendix, this proposition is a direct consequence of Theorem \ref{thm-confined-SGD}. The only thing not immediately clear is that the sequence $\{\vf_t\}_{t=0}^\infty$ is also bounded above. To see this, note that
    \begin{eqnarray*}
   && A_t \leq  \frac{1}{a}\sup\left\{\max\{0, ~4(a_{\eta,\gamma}-(1-\frac{\lambda}{w_{\eta,\gamma}})p_{\eta,\gamma})p_{\eta,\gamma} - 4\lambda \|\mathbf{x}\|^2\}~\big{|}~(U,\mathbf{x},V)\in K_1,~(\eta,\gamma) \in \Omega_{m\times n}\right\}, \\
   && B_t \leq  \frac{1}{b} \sup\left\{ \sqrt{8\sum_{l=1}^k (- (a_{\eta,\gamma}-(1-\frac{\lambda}{w_{\eta,\gamma}})p_{\eta,\gamma})u_{\eta,l} v_{\gamma,l} + \lambda \sum_{i=1}^m \sum_{j=1}^n p_{i,j}u_{i,l} v_{j,l})^2}~\big{|}~ (U,\mathbf{x},V)\in K_1, ~(\eta,\gamma) \in \Omega_{m\times n}\right\}, \\
    && \frac{c_t}{\Theta} \leq \frac{c}{\Theta},
    \end{eqnarray*}
where $K_1 = \{(U,\mathbf{x},V)\in V_k(\R^m)\times \R^k \times V_k(\R^n) ~\big{|}~ \|\mathbf{x}\|^2\leq \rho_0+ ca + \frac{b^2\sigma}{2}\}$ and $U = [u_{i,j}]$, $V = [v_{i,j}]$, $P = [p_{i,j}]  :=  U D^{k\times k}(\mathbf{x}) V^T$. Together, these imply that $\{\vf_t\}_{t=0}^\infty$ is bounded above.
\end{proof}

To use Algorithm \ref{alg-confined-SGD-WLRA-PW} directly, one needs to compute the $\vf_t$ given by Equation \eqref{eq-def-vf-t-WLRA-PW} in each iteration of the algorithm. When $mn$ is large, such computations may be costly. So we provide manual estimations for these $\{\vf_t\}$, which speed up the computation of each iteration of the algorithm at the expense of having somewhat smaller step sizes. Our estimates of $\vf_t$ also provide suggestions for the values of the constants $\lambda$, $a$ and $b$ that should be used in the algorithm.

\begin{lemma}\label{lemma-A-t-B-t-bound-WLRA-PW}
For $t\geq 0$, the $A_t$ and $B_t$ given in Algorithm \ref{alg-confined-SGD-WLRA-PW} are upper bounded by $\tilde{A}_t$ and $\tilde{B}_t$ given below:
\begin{equation}\label{eq-A-t-bound-WLRA-PW}
A_t\leq \tilde{A}_t := \begin{cases}
0 & \text{if } \|\mathbf{x}_t\|^2 \geq \frac{\alpha}{4\lambda(1-\frac{\lambda}{w_0})}, \\ 
\frac{4(\sqrt{\alpha} + \|\mathbf{x}_t\|)\|\mathbf{x}_t\|+4\lambda \|\mathbf{x}_t\|^2}{a} < \frac{\lambda +2\sqrt{\lambda}+1}{a\lambda(1-\frac{\lambda}{w_{0}})}\alpha &\text{if } \|\mathbf{x}_t\|^2 < \frac{\alpha}{4\lambda(1-\frac{\lambda}{w_0})},
\end{cases}
\end{equation}
\begin{equation}\label{eq-B-t-bound-WLRA-PW}
B_t\leq \tilde{B}_t := \frac{1}{b}\sqrt{16k (2 \alpha + (2+ \lambda^2) \|\mathbf{x}_t\|^2)} \leq \frac{1}{b}\sqrt{16k (2 \alpha + (2+ \lambda^2) \rho_1)}.
\end{equation}
\end{lemma}

\begin{proof}
Note that, for $(U,\mathbf{x},V)=([u_{i,j}],[x_1,\dots,x_k]^T,[v_{i,j}])\in V_k(\R^m)\times \R^k \times V_k(\R^n)$ and ${P = [p_{i,j}]  := U D^{k\times k}(\mathbf{x}) V^T}$, we have that $|p_{\eta,\gamma}|\leq \|P\|_F=\|\mathbf{x}\|=\sqrt{\rho(U,\mathbf{x},V)}$ and $|a_{\eta,\gamma}| \leq \sqrt{\alpha}$. Moreover, 
\[
4(a_{\eta,\gamma}-(1-\frac{\lambda}{w_{\eta,\gamma}})p_{\eta,\gamma})p_{\eta,\gamma} - 4\lambda \|\mathbf{x}\|^2\leq 0
\]
if $\|\mathbf{x}\|^2\geq \rho_0\geq \frac{\alpha}{4\lambda(1-\frac{\lambda}{w_0})}$. So, for $t\geq 0$, we have that
\begin{align*}
&\max\{0, ~4(a_{\eta,\gamma}-(1-\frac{\lambda}{w_{\eta,\gamma}})p_{t,\eta,\gamma})p_{t,\eta,\gamma} - 4\lambda \|\mathbf{x}_t\|^2\} \\
&\begin{cases}
=0 & \text{if } \|\mathbf{x}_t\|^2 \geq \frac{\alpha}{4\lambda(1-\frac{\lambda}{w_0})}, \\ 
\leq 4(\sqrt{\alpha} + \|\mathbf{x}_t\|)\|\mathbf{x}_t\|+4\lambda \|\mathbf{x}_t\|^2 < \frac{\lambda +2\sqrt{\lambda}+1}{\lambda(1-\frac{\lambda}{w_{0}})}\alpha &\text{if } \|\mathbf{x}_t\|^2 < \frac{\alpha}{4\lambda(1-\frac{\lambda}{w_0})}.
\end{cases}
\end{align*}
Thus, for the $A_t$ given in Algorithm \ref{alg-confined-SGD-WLRA-PW}, we have Equation \eqref{eq-A-t-bound-WLRA-PW}.

By Equation \eqref{eq-hessian-WLRA-PW}, for any $(U,\mathbf{x},V)=([u_{i,j}],[x_1,\dots,x_k]^T,[v_{i,j}])\in V_k(\R^m)\times \R^k \times V_k(\R^n)$, ${P = [p_{i,j}] :=  U D^{k\times k}(\mathbf{x}) V^T}$ and any $\theta \in [-\Theta,\Theta]$, we have 
\begin{eqnarray*}
0 &\leq & \Hess(\rho \circ R_{(U,\mathbf{x},V)})|_{\theta\nabla g_{\eta,\gamma}(U,\mathbf{x},V)}(\nabla g_{\eta,\gamma}(U,\mathbf{x},V) ,\nabla g_{\eta,\gamma}(U,\mathbf{x},V)) \\
& = & 8\sum_{l=1}^k (- (a_{\eta,\gamma}-(1-\frac{\lambda}{w_{\eta,\gamma}})p_{\eta,\gamma})u_{\eta,l} v_{\gamma,l} + \lambda x_l)^2 \\
& \leq & 16 \sum_{l=1}^k \left(( (a_{\eta,\gamma}-(1-\frac{\lambda}{w_{\eta,\gamma}})p_{\eta,\gamma})u_{\eta,l} v_{\gamma,l})^2 + \lambda^2 x_l^2\right).
\end{eqnarray*}
Since $(U,\mathbf{x},V)  \in V_k(\R^m)\times \R^k \times V_k(\R^n)$, we have that $\sum_{i=1}^m u_{i,l}^2 = \sum_{j=1}^n v_{j,l}^2 =1$ for $l=1,\dots,k$. In particular, $|u_{i,l}|,~|v_{j,l}|\leq 1$ for $l=1,\dots,k$. So 
\begin{align*}
((a_{\eta,\gamma}-(1-\frac{\lambda}{w_{\eta,\gamma}})p_{\eta,\gamma})u_{\eta,l} v_{\gamma,l})^2 
& \leq (a_{\eta,\gamma}-(1-\frac{\lambda}{w_{\eta,\gamma}})p_{\eta,\gamma})^2 \leq 2(a_{\eta,\gamma}^2+(1-\frac{\lambda}{w_{\eta,\gamma}})^2p_{\eta,\gamma}^2)\\
& \leq  2\alpha + 2\|P\|^2 =2\alpha + 2\|\mathbf{x}\|^2.
\end{align*}
Note that $x_l^2 = \|\mathbf{x}\|^2$. Combining the above, we get
\begin{eqnarray*}
0&\leq & \Hess(\rho \circ R_{(U,\mathbf{x},V)})|_{\theta\nabla g_{\eta,\gamma}(U,\mathbf{x},V)}(\nabla g_{\eta,\gamma}(U,\mathbf{x},V) ,\nabla g_{\eta,\gamma}(U,\mathbf{x},V)) \\
& \leq & 16k (2 \alpha + (2+ \lambda^2) \|\mathbf{x}\|^2) \leq 16k (2 \alpha + (2+ \lambda^2)\rho_1)
\end{eqnarray*}
for $(U,\mathbf{x},V)\in V_k(\R^m)\times \R^k \times V_k(\R^n)$ satisfying $\rho(U,\mathbf{x},V) \leq \rho_1$. This shows that, for the $B_t$ given in Algorithm \ref{alg-confined-SGD-WLRA-PW}, we have Equation \eqref{eq-B-t-bound-WLRA-PW}.
\end{proof}

With the estimations in Lemma \ref{lemma-A-t-B-t-bound-WLRA-PW}, we have the following corollary.

\begin{corollary}\label{cor-SGD-mfd-WLRA-PW-simplified-vf-t}
Proposition \ref{prop-confined-SGD-WLRA-PW} remains true if we replace the $\vf_t$ in it by 
\begin{equation}\label{eq-def-vf-t-WLRA-PW-simplified-vf-t-generalized}
\tilde{\vf}_t := \max\{\tilde{A}_t,\tilde{B}_t,\frac{c_t}{\Theta}, \Phi_{\min}\},
\end{equation}
where $\tilde{A}_t$ and $\tilde{B}_t$ are given in Equations \eqref{eq-A-t-bound-WLRA-PW} and \eqref{eq-B-t-bound-WLRA-PW}.
\end{corollary}

Again, as in Section \ref{sec-WLRA-confined-SGD}, $\tilde{\vf}_t$ is larger than $\vf_t$, but is much easier to compute when $mn$ is very large. Similar to the situation in Section \ref{sec-WLRA-confined-SGD}, we can choose the values of $\lambda$, $a$ and $b$ so that there exist finite upper bounds for $\{\vf_t\}$ and $\{\tilde{\vf}_t\}$ that remain finite as $w_0\rightarrow 0^+$. We pick
\begin{equation}\label{eq-lambda-a-b-choices-WLRA-PW}
\lambda = \frac{w_0}{2}, \ a = \frac{1}{w_0} \text{ and } b = \frac{1}{\sqrt{w_0}},
\end{equation}
where $w_0$ is given in Equation \eqref{eq-def-w-0}

\begin{corollary}\label{cor-a-b-choices-WLRA-PW}
For the choices of $\lambda$, $a$ and $b$ given in Equation \eqref{eq-lambda-a-b-choices-WLRA-PW}, we have
\begin{align*}
\vf_t \leq \tilde{\vf}_t
 \leq & \max \{4\alpha\left(\frac{w_0}{2} +2\sqrt{\frac{w_0}{2}}+1\right), 
 \sqrt{16k \left(2\alpha w_0+ \left(2+ \frac{w_0^2}{4}\right)\left(w_0\rho_0 + c + \frac{\sigma}{2}\right) \right)}, 
 \left. \frac{c}{\Theta}, \Phi_{\min}\right\} \\
 & =  O(1) \text{ as } w_0 \rightarrow 0^+.
\end{align*}
\end{corollary}

\begin{proof}
By Inequalities \eqref{eq-A-t-bound-WLRA-PW} and \eqref{eq-B-t-bound-WLRA-PW}, we have that
\begin{eqnarray*}
\vf_t \leq \tilde{\vf}_t & \leq & \max\left\{ \frac{\lambda +2\sqrt{\lambda}+1}{a\lambda(1-\frac{\lambda}{w_{0}})}\alpha, \frac{\sqrt{16k (2 \alpha + (2+ \lambda^2) \rho_1)}}{b}, \frac{c}{\Theta}, \Phi_{\min}\right\} \\
& \leq & \max\left\{ \frac{\lambda +2\sqrt{\lambda}+1}{a\lambda(1-\frac{\lambda}{w_{0}})}\alpha, \frac{\sqrt{16k (2\alpha+ (2+ \lambda^2)(\rho_0 + ca + \frac{b^2\sigma}{2}))}}{b}, \frac{c}{\Theta}, \Phi_{\min}\right\},
\end{eqnarray*}
where $\alpha$ is given in Equation \eqref{eq-def-alpha-RWLRA-1}. With the choices of $\lambda$, $a$ and $b$ in Equation \eqref{eq-lambda-a-b-choices-WLRA-PW}, 
\begin{eqnarray*}
&& \frac{\lambda +2\sqrt{\lambda}+1}{a\lambda(1-\frac{\lambda}{w_{0}})}\alpha = 4\alpha\left(\frac{w_0}{2} +2\sqrt{\frac{w_0}{2}}+1\right), \\
&& \frac{\sqrt{16k (2\alpha+ (2+ \lambda^2)(\rho_0 + ca + \frac{b^2\sigma}{2}))}}{b} = \sqrt{16k \left(2\alpha w_0+ \left(2+ \frac{w_0^2}{4}\right)\left(w_0\rho_0 + c + \frac{\sigma}{2}\right) \right)}.
\end{eqnarray*}
\end{proof}

Corollary \ref{cor-a-b-choices-WLRA-PW} means that, even for very small $w_0>0$, with the choices of $\lambda$, $a$ and $b$ given in Equation \eqref{eq-lambda-a-b-choices-WLRA-PW}, we can control by how much we shrink the preferred step size $c_t$ to get the actual step sizes $\frac{c_t}{\vf_t}$ and $\frac{c_t}{\tilde{\vf}_t}$ used in the gradient descent.

\begin{corollary} \label{cor-SGD-mfd-WLRA-PW-single-vf-t}
In Algorithm \ref{alg-confined-SGD-WLRA-PW}, choose $\lambda$, $a$ and $b$ as given by Equation \eqref{eq-lambda-a-b-choices-WLRA-PW}, and $\Theta$ and $\Phi_{\min}$ given by
\[
\Phi_{\min} \geq \max\left\{4\alpha \left(\frac{w_0}{2} +2\sqrt{\frac{w_0}{2}}+1\right), \sqrt{16k \left(2\alpha w_0+ \left(2+ \frac{w_0^2}{4}\right)\left(w_0\rho_0 + c + \frac{\sigma}{2}\right) \right)}\right\},
\]
and
\[
\Theta = \frac{c}{\Phi_{\min}}.
\]
For these values of $\lambda$, $a$, $b$, $\Theta$ and $\Phi_{\min}$, the sequence $\{\vf_t\}_{t=0}^{\infty}$ in Algorithm \ref{alg-confined-SGD-WLRA-PW} is given by the constant value
\[
\vf_t = \Phi_{\min}.
\]
\end{corollary} 

Next, similar to Appendix \ref{sec-RWLRA-2}, we give a special case of Algorithm \ref{alg-confined-SGD-WLRA-PW} as a benchmark. 
In this case, we fix the sequence $\{c_t\}_{t=0}^\infty$ with $c_t = \frac{1}{t+1}$. Thus, $c = \max \{c_t~\big{|}~t\geq 0\} = 1$ and $\sigma = \sum_{t=0}^\infty c_t^2 = \frac{\pi^2}{6}$. Further, we choose a positive scalar $\Phi_{\min}$ satisfying
\begin{equation}\label{eq-phi-min-choice-WLRA-PW-specific}
\Phi_{\min} = K \max\left\{4\alpha \left(\frac{w_0}{2} +2\sqrt{\frac{w_0}{2}}+1\right), \sqrt{16k \left(2\alpha w_0+ \left(2+ \frac{w_0^2}{4}\right)\left(w_0\rho_0 + \frac{6 + \pi^2}{6}\right) \right)}\right\}
\end{equation}
where $\alpha$ is given in Equation \eqref{eq-def-alpha-RWLRA-1}, $\rho_0$ is given in Equation \eqref{eq-rho-0-WLRA-PW}, $w_0$ is given in Equation \eqref{eq-def-w-0} and $K \geq 1$ is one of the constant inputs for Algorithm \ref{alg-confined-SGD-WLRA-PW-specific} below. Fix positive scalar $\Theta$ in Algorithm \ref{alg-confined-SGD-WLRA-PW-specific} as
\begin{equation}\label{eq-theta-choice-WLRA-PW-specific}
\Theta = \frac{1}{\Phi_{\min}}.
\end{equation}

Next, we are ready to give the special case of Algorithm \ref{alg-confined-SGD-WLRA-PW} .

\begin{algorithm}[A Special Case of Algorithm \ref{alg-confined-SGD-WLRA-PW}]\label{alg-confined-SGD-WLRA-PW-specific}\

\noindent\makebox[\linewidth]{\rule{\textwidth}{1pt}}

\textbf{Input:} 
\begin{itemize}
    \item[-] the random function $\tilde{g}$ given in Definition \ref{def-random-f-eta-gamma-RWLRA-2}, 
    \item[-] the retraction $R$ given in Equation \eqref{eq-retraction-RWLRA-2},
    \item[-] the positive integers $m, n$ and $k$ given in Problem \ref{prob-RWLRA-2}, 
    \item[-] the matrices $A$ and $W$ given in Problem \ref{prob-RWLRA-2},
    \item[-] a scalar $K \geq 1$, 
    \item[-] an initial iterate $(U_0,\mathbf{x}_0,V_0) \in V_k(\R^m)\times \R^k \times V_k(\R^n)$.
\end{itemize}

\textbf{Output:} A sequence of iterates $\{(U_t,\mathbf{x}_t,V_t)\}_{t=0}^\infty  \subset V_k(\R^m)\times \R^k \times V_k(\R^n)$.
\begin{itemize}
	\item \emph{for $t=0,1,2\dots$ do}
	\begin{enumerate}[1.]
        \item Select a random element $(\eta_t, \gamma_t)$ from $\Omega_{m, n}$ with the probability distribution $\mu$ independent of $\{(\eta_{\tau}, \gamma_{\tau})\}_{\tau=0}^{t-1}$.
	\item Set
	\begin{equation}\label{eq-confined-SGD-WLRA-PW-recursion}
	(U_{t+1},\mathbf{x}_{t+1},V_{t+1})= R_{(U_t,\mathbf{x}_t,V_t)}\left(-\frac{1}{(t+1)\Phi_{\min}}\nabla \tilde{g}_{\eta_t,\gamma_t}(U_t,\mathbf{x}_t,V_t)\right),
	\end{equation}
        where $\nabla \tilde{g}_{\eta_t,\gamma_t}(U_t,\mathbf{x}_t,V_t)$ is given in Corollary \ref{cor-random-hat-g-gradient-WLRA-PW} and $\Phi_{\min}$ is given in Equation \eqref{eq-phi-min-choice-WLRA-PW-specific}.
	\end{enumerate}
	\item \emph{end for}
\end{itemize}
\noindent\makebox[\linewidth]{\rule{\textwidth}{1pt}}
\end{algorithm}

\begin{proposition}\label{prop-SGD-mfd-WLRA-PW-specific}
Let $\hat{G}$ be the function given in Equation \eqref{eq-def-g-WLRA-PW} and ${\{(U_t,\mathbf{x}_t,V_t)\}_{t=0}^\infty \subset V_k(\R^m)\times \R^k \times V_k(\R^n)}$ be the sequence from the Algorithm \ref{alg-confined-SGD-WLRA-PW-specific}. Then:
\begin{enumerate}
    \item $\{(U_t,\mathbf{x}_t,V_t)\}_{t=0}^\infty$ is contained in the compact subset \\ ${\{(U,\mathbf{x},V)\in V_k(\R^m)\times \R^k \times V_k(\R^n) ~\big{|}~\|\mathbf{x}\|^2\leq \rho_0 + \frac{\pi^2 + 12}{12w_0}\}}$ of $V_k(\R^m)\times \R^k \times V_k(\R^n)$ where $\rho_0$ given in Equation \eqref{eq-rho-0-WLRA-PW} and $w_0$ is given in Equation \eqref{eq-def-w-0}. Therefore, $\{(U_t,\mathbf{x}_t,V_t)\}_{t=0}^\infty$ has convergent subsequences;
    \item $\{\hat{G}(U_t,\mathbf{x}_t,V_t)\}_{t=0}^\infty$ converges almost surely to a finite number;
    \item $\{\|\nabla \hat{G}(U_t,\mathbf{x}_t,V_t)\|\}_{t=0}^\infty$ converges almost surely to $0$;
    \item any limit point of $\{(U_t,\mathbf{x}_t,V_t)\}_{t=0}^\infty$ is almost surely a stationary point of $\hat{G}$.
\end{enumerate}
\end{proposition}

\begin{proof}
Let $\Phi_{\min}$ be as in Equation \eqref{eq-phi-min-choice-WLRA-PW-specific} and $\Theta$ in Equation \eqref{eq-theta-choice-WLRA-PW-specific}. Proposition \ref{prop-SGD-mfd-WLRA-PW-specific} follows from Corollary \ref{cor-SGD-mfd-WLRA-PW-single-vf-t} as a special case.
\end{proof}

\begin{remark}\label{rk-SGD-mfd-WLRA-PW-single-vf-t}
Proposition \ref{prop-SGD-mfd-WLRA-PW-specific} follow from Corollary \ref{cor-confined-SGD} with
\[
\vf = \Phi_{\min} 
= K \max\left\{4\alpha \left(\frac{w_0}{2} +2\sqrt{\frac{w_0}{2}}+1\right), \sqrt{16k \left(2\alpha w_0+ \left(2+ \frac{w_0^2}{4}\right)\left(w_0\rho_0 + \frac{6 + \pi^2}{6}\right) \right)}\right\}.
\]
Theoretically, when we hold $K$ constant for different $w_0$, we still have convergent stochastic gradient descents and it is possible to observe 
\begin{equation}\label{eq-def-vf-t-WLRA-PW-single-vf-t-specific}
\vf = O(1) \text{ as } w_0 \rightarrow 0^+.
\end{equation}
\end{remark}

\subsection{An Accelerated Line Search for Problem \ref{prob-WLRA-PW-reform}}\label{subsec-WLRA-PW-ALS}
To design an accelerated line search algorithm for Problem \ref{prob-WLRA-PW-reform}, we first compute the gradient of the function $\hat{G}: V_k(\R^m)\times \R^k \times V_k(\R^n)\rightarrow \R$ given in Definition \ref{def-func-hat-G}.

\begin{corollary}\label{cor-g-gradient-WLRA-PW}
For function $\hat{F}:\R^{m\times n} \rightarrow \R$ given in Problem \ref{prob-WLRA-PW-reform}, define the matrices
\begin{eqnarray*}
\nabla_U \hat{F} & = & \left[\frac{\partial \hat{F}}{\partial u_{i,l}}\right]_{m\times k} = (-2W \odot (A - P))(V D_{k \times k}(x)), \\
\nabla_V \hat{F} & = & \left[\frac{\partial \hat{F}}{\partial v_{j,l}}\right]_{n\times k} = (-2W \odot (A - P))^\top(U D_{k \times k}(x)), \\
\nabla_{\mathbf{x}} \hat{F} & = & \left[\frac{\partial \hat{F}}{\partial x_l}\right]_k = \diag{\left(U^\top(-2W \odot (A - P))V^\top\right)},
\end{eqnarray*}
where $\odot$ denotes the Hadamard product of matrices. Then, for the function $\hat{G}: V_k(\R^m)\times \R^k \times V_k(\R^n)\rightarrow \R$ given in Definition \ref{def-func-hat-G}, the gradient of $\hat{G}$ at any $(U,\mathbf{x},V) \in V_k(\R^m)\times \R^k \times V_k(\R^n)$ is 
\begin{equation}\label{eq-hat-G-gradient-WLRA-PW}
\nabla \hat{G}(U,\mathbf{x},V) = (\Pi_U(\nabla_U \hat{F}), \nabla_{\mathbf{x}} \hat{F}, \Pi_V(\nabla_V \hat{F})) \in T_U V_k(\R^m)\times \R^k \times T_VV_k(\R^n),
\end{equation}
where $\Pi_U$ and $\Pi_V$ are the orthogonal projections given in Lemma \ref{lemma-stiefel-orthogonal-proj}.
\end{corollary}

\begin{proof}
Note that, $\hat{F} = \sum_{\eta=1}^m\sum_{\gamma=1}^n w_{\eta,\gamma}\hat{f}_{\eta,\gamma}$. Therefore,
\begin{align*}
\frac{\partial \hat{F}}{\partial u_{i,l}} &=
\sum_{\eta=1}^m\sum_{\gamma=1}^n w_{\eta,\gamma}\frac{\partial \hat{f}_{\eta,\gamma}}{\partial u_{i,l}}
= \sum_{j=1}^n -2w_{i,j}(a_{i,j} - p_{i,j})x_l v_{j,l} \\
\frac{\partial \hat{F}}{\partial v_{j,l}} &=
\sum_{\eta=1}^m\sum_{\gamma=1}^n w_{\eta,\gamma}\frac{\partial \hat{f}_{\eta,\gamma}}{\partial v_{j,l}}
=\sum_{i=1}^m -2w_{i,j}(a_{i,j} - p_{i,j})x_l u_{i,l}, \\
\frac{\partial \hat{F}}{\partial x_l} &=
\sum_{\eta=1}^m\sum_{\gamma=1}^n w_{\eta,\gamma}\frac{\partial \hat{f}_{\eta,\gamma}}{\partial x_l}
= \sum_{i=1}^m\sum_{j=1}^n -2w_{i,j}(a_{i,j} - p_{i,j})u_{i, l}v_{j, l}.
\end{align*}
Also, note that,
\begin{eqnarray*}
\left[\frac{\partial \hat{F}}{\partial u_{i,l}}\right]_{m\times k}
&=& \left[\sum_{j=1}^n -2w_{i,j}(a_{i,j} - p_{i,j})x_l v_{j,l}\right]_{m\times k}
= (-2W \odot (A - P))(V D_{k \times k}(x)) \\
\left[\frac{\partial \hat{F}}{\partial v_{j,l}}\right]_{n\times k}
&=& \left[\sum_{i=1}^m -2w_{i,j}(a_{i,j} - p_{i,j})x_l u_{i,l}\right]_{n\times k}
= (-2W \odot (A - P))^\top(U D_{k \times k}(x)), \\
\left[\frac{\partial \hat{F}}{\partial x_l}\right]_k
&=& \left[\sum_{i=1}^m\sum_{j=1}^n -2w_{i,j}(a_{i,j} - p_{i,j})u_{i, l}v_{j, l}\right]_k
= \diag{\left(U^\top(-2W \odot (A - P))V^\top\right)}.
\end{eqnarray*}
Then, by Lemmas \ref{lemma-stiefel} and \ref{lemma-gradient-submfd}, the gradient of $\hat{G}$ is given by Equation \eqref{eq-hat-G-gradient-WLRA-PW}.
\end{proof}

Next, our accelerated line search algorithm for Problem \ref{prob-WLRA-PW-reform} is given below.

\begin{algorithm}[An Accelerated Line Search for Problem \ref{prob-WLRA-PW-reform}]\label{alg-ALS-WLRA-PW}\

\noindent\makebox[\linewidth]{\rule{\textwidth}{1pt}}

\textbf{Input:} 
\begin{itemize}
    \item[-] the function $\hat{G}$ given in Definition \ref{def-func-hat-G}, 
    \item[-] the retraction $R$ given in Definition \ref{def-retraction-GS-prod}, 
    \item[-] the positive integers $m, n$ and $k$ given in Problem \ref{prob-WLRA-PW-reform}, 
    \item[-] the matrices $A$ and $W$ given in Problem \ref{prob-WLRA-PW-reform}, 
    \item[-] scalars $\overline{\alpha}>0$ and $\beta,\iota \in (0,1)$, 
    \item[-] an initial iterate $(U_0,\mathbf{x}_0,V_0) \in V_k(\R^m)\times \R^k \times V_k(\R^n)$.
\end{itemize}

\textbf{Output:} A sequence of iterates $\{(U_t,\mathbf{x}_t,V_t)\}_{t=0}^\infty \subset V_k(\R^m)\times \R^k \times V_k(\R^n)$.
\begin{itemize}
	\item \emph{for $t=0,1,2\dots$ do}
	\begin{enumerate}[1.]
		\item Select $(U_{t+1},\mathbf{x}_{t+1},V_{t+1})$ so that
		\begin{equation}\label{eq-WLRA-PW-ALS-recursion}
		\hat{G}(U_t,\mathbf{x}_t,V_t)-\hat{G}(U_{t+1},\mathbf{x}_{t+1},V_{t+1}) \geq \hat{G}(U_t,\mathbf{x}_t,V_t)-\hat{G}(R_{(U_t,\mathbf{x}_t,V_t)}(-\tau_t^A \nabla \hat{G}(U_t,\mathbf{x}_t,V_t))),
		\end{equation}
		where $\nabla \hat{G}(U_t,\mathbf{x}_t,V_t)$ is given in Corollary \ref{cor-g-gradient-WLRA-PW} and $\tau_t^A$ is the Armijo step size for the given $\overline{\alpha},\beta,\iota$ and $-\nabla \hat{G}(U_t,\mathbf{x}_t,V_t)$ (see Definition \ref{def-Armijo}).
	\end{enumerate}
	\item \emph{end for}
\end{itemize}
\noindent\makebox[\linewidth]{\rule{\textwidth}{1pt}}
\end{algorithm}

\begin{proposition}\label{prop-ALS-WLRA-PW}
Let $\{(U_t,\mathbf{x}_t,V_t)\}_{t=0}^\infty$ be the sequence from Algorithm \ref{alg-ALS-WLRA-PW}. Then $\{(U_t,\mathbf{x}_t,V_t)\}_{t=0}^\infty$ has a subsequence converging to a critical point of $\hat{G}$ and ${\lim_{t \rightarrow \infty} \|\nabla \hat{G}(U_t,\mathbf{x}_t,V_t)\|=0}$.
\end{proposition}

\begin{proof}
Note that, with $(U_0,\mathbf{x}_0,V_0)$ given in Algorithm \ref{alg-ALS-WLRA-PW},
\begin{eqnarray*}
\hat{G}(U_0,\mathbf{x}_0,V_0) & \geq & \hat{G}(U,\mathbf{x},V) = \hat{F}(U D^{k\times k}_k(\mathbf{x}) V^T)
= \sum_{i=1}^m \sum_{j=1}^n w_{i,j}(a_{i,j}-p_{i,j})^2 \\
& \geq & \sum_{i=1}^m \sum_{j=1}^n w_0 (a_{i,j}-p_{i,j})^2
= w_0 \| A - P \|_F^2,
\end{eqnarray*}
where $w_0$ is given in Equation \eqref{eq-def-w-0}. Then, by Equation \eqref{eq-norm-equal},
\[
\|\mathbf{x}\| = \|P\|_F \leq \|A\|_F + \|A - P\|_F \leq \|A\|_F + \sqrt{\frac{\hat{G}(U_0,\mathbf{x}_0,V_0)}{w_0}},
\]
which implies that $\{(U,\mathbf{x},V)  \in V_k(\R^m)\times \R^k \times V_k(\R^n) ~|~ \hat{G}(U,\mathbf{x},V) \leq \hat{G}(U_0,\mathbf{x}_0,V_0)\}$ is a compact subset of $V_k(\R^m)\times \R^k \times V_k(\R^n)$. Since the sequence $\{-\nabla \hat{G}(U_t,\mathbf{x}_t,V_t)\}_{t=0}^\infty$ is gradient related to $G$, we conclude that Proposition \ref{prop-ALS-WLRA-PW} follows from Theorem \ref{thm-ALS-converges}.
\end{proof}

\section{A New Proof of the Eckart-Young Theorem}\label{sec-new-proof-eyt}

The Eckart-Young Theorem plays a key role in the theory of low-rank approximation. While studying Problem \ref{prob-WLRA}, we found a new proof of this theorem using a direct critical point analysis, instead of the tricky interplay between the spectral and Frobenius norms employed by the traditional proof (for instance, \cite{Wikipedia-Low-Rank-Approximation:2024}).

\subsection{The Eckart-Young Theorem}\label{subsec-eyt}
\begin{theorem}[Eckart-Young]\label{thm-eyt}
For a matrix $A \in \R^{m\times n}$, counting multiplicity, denote by $s_1\geq\cdots\geq s_r$ the positive singular values of $A$. 
For any matrix $P \in \R^{m\times n}$ of satisfying $rankP \leq k$, we have that $\|A-P\|_F^2\geq \sum_{j=k+1}^\infty s_j^2$, where
\begin{itemize}
	\item $\|X\|_F:= \sqrt{\Tr (X^T X)}= \sqrt{\Tr (X X^T)}$ is the Frobenius norm of a real matrix $X$,
	\item $s_j=0$ if $j>r$.
\end{itemize}
In particular, $\|A-P\|_F^2= \sum_{j=k+1}^\infty s_j^2$ if and only if there are a matrix $V \in O_m$ and a matrix $U \in O_n$ such that 
\begin{eqnarray}
\label{eq-svd-A} A & =& V \diag_{m\times n}\{s_1,\dots, s_r,0,\dots,0\}U^T, \\
\label{eq-svd-P} P & =& V \diag_{m\times n}\{s_1,\dots, s_k,0,\dots,0\}U^T,
\end{eqnarray}
where $\diag_{m\times n}\{t_1,\dots, t_{\min\{m,n\}}\}$ is the $m\times n$ diagonal matrix $[x_{i,j}]_{m\times n}$ satisfying $x_{i,j}=\begin{cases}
t_i & \text{if }i=j, \\
0 & \text{otherwise.}
\end{cases}$.
\end{theorem}

\subsection{A New Proof}\label{subsec-proof}

As mentioned above, we prove the Eckart-Young Theorem via a direct critical point analysis. The linear algebra arguments involved in our proof come largely from rehashing the proof of the Singular Value Decomposition. We will be using the following notations in our proof.
\begin{itemize}
	\item  Denote by $\left\langle A, B\right\rangle  :=\Tr(A^TB)= \Tr(BA^T)$ the standard bi-invariant inner product of $\R^{m\times n}$. Note that $\|A\|_F^2 = \left\langle A, A \right\rangle$ for $A \in \R^{m\times n}$. 
	\item $\R^{m\times n}_k$ is the set of $m \times n$ real matrices of rank at most $k$.
\end{itemize}

\begin{lemma}\label{lemma-min-exists}
Define $f:\R^{m\times n}_k \rightarrow \R$ to be the function $f(P)=\|A-P\|_F^2$ for all $P \in \R^{m\times n}_k$. Then $f$ attains its minimal value.
\end{lemma}
\begin{proof}
Note that $\R^{m\times n}_k=\{P\in\R^{m\times n}~|~\text{the determinants of all } (k+1)\times(k+1) \text{ minors of } P \text{ are } 0\}$. Since determinants of minors are continuous functions of $P$, this implies that $\R^{m\times n}_k$ is a closed subset of $\R^{m\times n}$. Furthermore, $f(P) \geq (\|A\|_F-\|P\|_F)^2$. So $f(P)> \|A\|_F^2 = f(0)$ if $\|P\|_F> 2\|A\|_F$. Therefore, 
\[
\inf\{f(P)~|~ P \in \R^{m\times n}_k\}= \inf\{f(P)~|~ P \in \R^{m\times n}_k \text{ and } \|P\|_F\leq 2\|A\|_F\}.
\] 
But $\{P \in \R^{m\times n}_k ~|~ \|P\|_F\leq 2\|A\|_F\}$ is a compact set since it is the intersection of the closed set $\R^{m\times n}_k$ and the compact set  $\{P \in \R^{m\times n} ~|~ \|P\|_F\leq 2\|A\|_F\}$. Thus, $f$ attains its minimal value by the Extreme Value Theorem.
\end{proof}

\begin{lemma}\label{lemma-min-character}
Assume that $P_0\in \R^{m\times n}_k$ is a minimizer of the function $f:\R^{m\times n}_k \rightarrow \R$ given by $f(P)=\|A-P\|_F^2$. Then $A^TP_0=P_0^TP_0$ and $P_0A^T=P_0P_0^T$.
\end{lemma}

\begin{proof}
For any matrix $R \in \R^{m\times m}$ and matrix $S \in \R^{n\times n}$, define $g:\R^2\rightarrow \R$ by 
\[
g(x,y)= \|A-(I_m+ xR)P_0(I_n+ yS)\|_F^2,
\]
where $I_m \in \R^{m\times m}$ and $I_n \in \R^{n\times n}$ are the identity matrices. Note that $(I_m+ xR)P_0(I_n+ yS) \in \R^{m\times n}_k$ for all $(x,y)\in \R^2$ since $P_0\in \R^{m\times n}_k$. One can see that $g(x,y) = f((I_m+ xR)P_0(I_n+ yS)) \geq f(P_0) = g(0,0)$ for all $(x,y)\in \R^2$ since $P_0$ is a minimizer of $f$. So $\frac{\partial g}{\partial x} (0,0) = \frac{\partial g}{\partial y} (0,0) =0$. Then 
\[
0 = \frac{\partial g}{\partial x} (0,0) = 2 \left\langle -RP_0 , A-P_0\right\rangle = -2\Tr ((A-P_0)P_0^TR^T) = -2 \left\langle R , (A-P_0)P_0^T\right\rangle.
\]
But the above equation is true for every matrix $R \in \R^{m\times m}$. This implies that $(A-P_0)P_0^T=0$, that is $AP_0^T = P_0P_0^T$. Taking transpositions of both sides of the latter equation, one gets that $P_0A^T=P_0P_0^T$. Similarly, using $\frac{\partial g}{\partial y} (0,0) =0$, one gets that $A^TP_0=P_0^TP_0$.
\end{proof}

\begin{lemma}\label{lemma-min-decomp}
Assume that $A, P\in  \R^{m\times n}$ satisfy $A^TP=P^TP$ and $PA^T=PP^T$. Then there are 
\begin{itemize}
	\item a matrix $V \in O_m$,
	\item a matrix $U \in O_n$,
	\item a non-negative integer $r$ and $\sigma_1,\dots,\sigma_r>0$,
	\item a non-negative integer $l$ satisfying $l \leq r$,
\end{itemize}
such that 
\begin{eqnarray*}
A & =& V \diag_{m\times n}\{\sigma_1,\dots, \sigma_r,0,\dots,0\}U^T, \\
P & =& V \diag_{m\times n}\{\sigma_1,\dots, \sigma_l,0,\dots,0\}U^T.
\end{eqnarray*}
\end{lemma}

\begin{proof}
Set $r=\rank A$ and $l =\rank P$. Then $r =\rank A \geq \rank PA^T = \rank PP^T=\rank P=l\geq 0$. Denote by $\sigma_1,\dots,\sigma_l$ the positive singular values of $P$ counting multiplicity. Applying the orthogonal diagonalization algorithm to the symmetric matrix $P^TP$, one gets an orthonormal subset $\{u_1,\dots,u_l\}$ of $\R^n$ such that $P^TPu_i = \sigma_i^2 u_i$ for $i=1,\dots,l$. Define $v_i := \frac{1}{\sigma_i}Pu_i$ for  $i=1,\dots,l$. Then $\{v_1,\dots,v_l\}$ is an orthonormal subset of $\R^m$ and 
$P[u_1,\dots,u_l]=[v_1,\dots,v_l]\diag_{l\times l}\{\sigma_1,\dots,\sigma_l\}$. Let $\{\hat{u}_{l+1},\dots,\hat{u}_n\}$ be an orthonormal basis for the eigenspace of $P^TP$ belonging to the eigenvalue $0$.  Also, extend $\{v_1,\dots,v_l\}$ to an orthonormal basis $\{v_1,\dots,v_l, \hat{v}_{l+1},\dots,\hat{v}_m\}$ of $\R^m$. Set
\[
\hat{U}=[u_1,\dots,u_l,\hat{u}_{l+1},\dots,\hat{u}_n] \in \R^{n \times n} \text{ and } \hat{V}=[v_1,\dots,v_l,\hat{v}_{l+1},\dots,\hat{v}_m] \in \R^{m \times m}.
\]
These are orthogonal matrices. And $P\hat{U} = \hat{V} \diag_{m\times n}\{\sigma_1,\dots,\sigma_l,0,\dots,0\}$. That is, 
\begin{equation}\label{eq-P-hat-decomp}
P = \hat{V} \diag_{m\times n}\{\sigma_1,\dots,\sigma_l,0,\dots,0\}\hat{U}^T.
\end{equation}
So 
\begin{eqnarray*}
&& \hat{V} \diag_{m\times n}\{\sigma_1,\dots,\sigma_l,0,\dots,0\}\hat{U}^T A^T = PA^T=PP^T \\
& = &  \hat{V} \diag_{m\times n}\{\sigma_1,\dots,\sigma_l,0,\dots,0\}\hat{U}^T \hat{U} \diag_{n\times m}\{\sigma_1,\dots,\sigma_l,0,\dots,0\} \hat{V}^T \\
& = & \hat{V} \diag_{m\times m}\{\sigma_1^2,\dots,\sigma_l^2,0,\dots,0\} \hat{V}^T.
\end{eqnarray*}
Multiplying $\hat{V}^T$ on the left of both sides of the above equation, one gets that
\[
\diag_{m\times n}\{\sigma_1,\dots,\sigma_l,0,\dots,0\}\hat{U}^T A^T  =  \diag_{m\times m}\{\sigma_1^2,\dots,\sigma_l^2,0,\dots,0\} \hat{V}^T.
\]
Taking transpositions on both side, this gives
\[
A [u_1,\dots,u_l,\hat{u}_{l+1},\dots,\hat{u}_n] \diag_{n\times m}\{\sigma_1,\dots,\sigma_l,0,\dots,0\} = [v_1,\dots,v_l,\hat{v}_{l+1},\dots,\hat{v}_m] \diag_{m\times m}\{\sigma_1^2,\dots,\sigma_l^2,0,\dots,0\}.
\]
The first $l$ columns of this equation imply that 
\begin{equation}\label{eq-Au=Pu}
A [u_1,\dots,u_l] = [\sigma_1v_1,\dots,\sigma_lv_l]=[Pu_1,\dots,Pu_l].
\end{equation}
Multiplying $A^T$ on the left of both sides of Equation \eqref{eq-Au=Pu}, one gets that 
\[
A^TA[u_1,\dots,u_l] = [A^TPu_1,\dots,A^TPu_l] = [P^TPu_1,\dots,P^TPu_l] = [\sigma_i^2 u_1,\dots, \sigma_i^2 u_l].
\]
Consequently, 
\begin{eqnarray*}
&& \hat{U}^TA^TA\hat{U} = \hat{U}^T[A^TAu_1,\dots,A^TAu_l,A^TA\hat{u}_{l+1},\dots,A^TA\hat{u}_n] \\
& = & \hat{U}^T[\sigma_i^2 u_1,\dots, \sigma_i^2 u_l, A^TA\hat{u}_{l+1},\dots,A^TA\hat{u}_n] \\
& = & \begin{bmatrix}
u_1^T \\
\vdots \\
u_l^T \\
\hat{u}_{l+1}^T \\
\vdots \\
\hat{u}_n^T
\end{bmatrix} [\sigma_i^2 u_1,\dots, \sigma_i^2 u_l, A^TA\hat{u}_{l+1},\dots,A^TA\hat{u}_n] = \begin{bmatrix}
\diag_{l\times l}\{\sigma_1^2,\dots,\sigma_l^2\} & C_{(n-l)\times l} \\
0_{l \times (n-l)} & B_{(n-l)\times (n-l)} ,
\end{bmatrix}
\end{eqnarray*} 
where, in the last step, the lower indices indicate the dimensions of the blocks. But $\hat{U}^TA^TAU$ is a positive semidefinite symmetric matrix. So we have that $C=0$ and $B$ is also a positive semidefinite symmetric matrix. Applying the Orthogonal Diagonalization Algorithm to $B$, one gets a matrix $W \in O_{n-l}$ and $\sigma_{l+1},\dots,\sigma_r>0$ such that
\[
B = W\diag_{(n-l)\times (n-l)}\{\sigma_{l+1}^2,\dots,\sigma_r^2,0,\dots,0\} W^T.
\]

Define 
\begin{equation}\label{eq-form-U}
U = \hat{U}\begin{bmatrix}
I_l & 0 \\
0 & W
\end{bmatrix}\in \R^{n\times n}.
\end{equation}
Then $U$ is orthogonal. And $U^TA^TAU = \diag_{n\times n}\{\sigma_1^2,\dots,\sigma_r^2,0, \dots,0\}$. By the definition of $U$, it has the same first $l$ columns as $\hat{U}$. So $U=[u_1,\dots, u_l,u_{l+1},\dots,u_n]$. By Equation \eqref{eq-Au=Pu}, we have that $v_i= \frac{1}{\sigma_i}Au_i$ for  $i=1,\dots,l$. Let $v_j=\frac{1}{\sigma_j}Au_j$ for $j=l+1, \dots, r$. Note that $\{v_1,\dots,v_r\}$ is an orthonormal subset of $\R^m$. Extend this set to an orthonormal basis $\{v_1,\dots,v_m\}$ of $\R^m$. Denote by $V$ the matrix $[v_1,\dots,v_m] \in O_{m}$. Note that $\{v_{l+1},\dots,v_m\} \subset \{v_1,\dots,v_l\}^{\perp} = \Span\{\hat{v}_{l+1},\dots,\hat{v}_m\}$. So, there is a matrix $Z \in O_{m-l}$ such that 
\begin{equation}\label{eq-form-V}
V = \hat{V}\begin{bmatrix}
I_l & 0 \\
0 & Z
\end{bmatrix}.
\end{equation}
By the definitions of  the the matrices $U \in O_n$ and $V \in O_m$, we have that $AU = V \diag_{m\times n}\{\sigma_1,\dots,\sigma_r,0, \dots,0\}$. That is, $A = V \diag_{m\times n}\{\sigma_1,\dots,\sigma_r,0, \dots,0\}U^T$. By Equations \eqref{eq-P-hat-decomp}, \eqref{eq-form-U} and \eqref{eq-form-V}, we have that
\begin{eqnarray*}
PU & = & \hat{V} \diag_{m\times n}\{\sigma_1,\dots,\sigma_l,0,\dots,0\}\hat{U}^T\hat{U}\begin{bmatrix}
I_l & 0 \\
0 & W
\end{bmatrix} = \hat{V} \diag_{m\times n}\{\sigma_1,\dots,\sigma_l,0,\dots,0\}\begin{bmatrix}
I_l & 0 \\
0 & W
\end{bmatrix} \\
& = & \hat{V} \diag_{m\times n}\{\sigma_1,\dots,\sigma_l,0,\dots,0\} = \hat{V} \begin{bmatrix}
I_l & 0 \\
0 & Z
\end{bmatrix}\diag_{m\times n}\{\sigma_1,\dots,\sigma_l,0,\dots,0\} \\
& = & V \diag_{m\times n}\{\sigma_1,\dots,\sigma_l,0,\dots,0\}.
\end{eqnarray*}
That is, $P = V \diag_{m\times n}\{\sigma_1,\dots, \sigma_l,0,\dots,0\}U^T$.
\end{proof}

With the above lemmas, the proof of Theorem \ref{thm-eyt} is straightforward.

\begin{proof}[Proof of Theorem \ref{thm-eyt}]
By Lemma \ref{lemma-min-exists}, the minimum of $\|A-P\|_F^2$ under the constraint $\rank P\leq k$ is attained at some $P_0\in \R^{m\times n}_k$. By Lemmas \ref{lemma-min-character} and \ref{lemma-min-decomp}, there are 
\begin{itemize}
	\item a matrix $V \in O_m$,
	\item a matrix $U \in O_n$,
	\item a non-negative integer $r$ and $\sigma_1,\dots,\sigma_r>0$,
	\item a non-negative integer $l$ satisfying $l \leq\min\{r,k\}$,
\end{itemize}
such that 
\begin{eqnarray}
\label{eq-svd-A-sigma} A & =& V \diag_{m\times n}\{\sigma_1,\dots, \sigma_r,0,\dots,0\}U^T, \\
\label{eq-svd-P-sigma} P_0 & =& V \diag_{m\times n}\{\sigma_1,\dots, \sigma_l,0,\dots,0\}U^T.
\end{eqnarray}
Note that, counting multiplicity, $\{\sigma_1,\dots,\sigma_r\}$ is a permutation of $\{s_1,\cdots, s_r\}$. So
\[
\|A-P_0\|_F^2 = \sum_{j=l+1}^r \sigma_j^2 \geq \sum_{j=l+1}^\infty s_j^2 \geq \sum_{j=k+1}^\infty s_j^2.
\]
This shows that
\[
\min\{\|A-P\|_F^2~|~ P \in\R^{m\times n}_k\} \geq \sum_{j=k+1}^\infty s_j^2.
\]

Clearly, for any singular value decomposition of $A$ given by Equation \eqref{eq-svd-A}, we have that ${\|A-P\|_F^2 =  \sum_{j=k+1}^\infty s_j^2}$ if $P$ is given by Equation \eqref{eq-svd-P}. This shows that
\[
\min\{\|A-P\|_F^2~|~ P \in\R^{m\times n}_k\} = \sum_{j=k+1}^\infty s_j^2
\]
and that the minimum is attained at any $P$ given by Equation \eqref{eq-svd-P}. 

It remains to show that the minimum is only attained at matrices $P$ given by Equation \eqref{eq-svd-P}. Assume again that $P_0$ is a minimizer of $\|A-P\|_F^2$ under the constraint $\rank P\leq k$. Then Equations \eqref{eq-svd-A-sigma} and \eqref{eq-svd-P-sigma} remain true for $P_0$. Additionally, we must have that $\sum_{j=l+1}^r \sigma_j^2 = \sum_{j=k+1}^\infty s_j^2$, which is only true if $\{\sigma_1,\dots, \sigma_l\} =\{s_1,\dots, s_k\}$ counting multiplicity. One can then re-order $\{\sigma_1,\dots, \sigma_l\}$ and $\{\sigma_{l+1},\dots, \sigma_r\}$ into descending sequences in Equations \eqref{eq-svd-A-sigma} and \eqref{eq-svd-P-sigma} by re-ordering the columns of $U$ and $V$. This makes Equations \eqref{eq-svd-A-sigma} and \eqref{eq-svd-P-sigma} into the formats given by Equations \eqref{eq-svd-A} and \eqref{eq-svd-P}. Thus, $P_0$ must be a matrix given by Equation \eqref{eq-svd-P}.
\end{proof}

\end{document}